\documentclass[11pt]{amsart}
\usepackage{
   a4wide, 
  amssymb,
  graphicx, 
  times,
  color,
  hyphenat, 
  mathabx, 
  stmaryrd, datetime, bbm 
}  

\usepackage{tikz,etex,bbm,xspace}
\usepackage[colorlinks=true, pdfstartview=FitV, linkcolor=blue, citecolor=blue, urlcolor=blue]{hyperref}%
\input xy
\xyoption{all}
\DeclareMathOperator{\nh}{NH}
\DeclareMathOperator{\opp}{op}

\newcommand{\sbim}{{\mathbf{Bim}^s}}
\newcommand{\extflag}{{\mathbf{ExtFlag}_\lambda}}
\newcommand{\omg}{{\overline{\Omega}}}
\newcommand{\Omg}{{\widecheck{\Omega}}}
\newcommand{\Omgq}{{\frac{\Omega_{k(k+1)k}}{\Omega_{k(k-1)k}}}}
\newcommand\E{{\sf{E}}}
\newcommand\F{{\sf{F}}}
\newcommand\K{{\sf{K}}}
\newcommand\Q{{\sf{Q}}}
\newcommand{\bV}{\raisebox{0.03cm}{\mbox{\footnotesize$\textstyle{\bigwedge}$}}}

\newcommand{\slt}{{\mathfrak{sl}_{2}}}
\newcommand{\und}[1]{{\underline{#1}}}

\newcommand{\brak}[1]{\langle #1\rangle}
\newcommand{\pp}[1]{(\!( #1 )\!)}
\newcommand{\n}{\noindent}
\newcommand{\qbin}[2]{\left[{{#1}\atop {#2}}\right]}

\def\E{{\sf{E}}}
\def\F{{\sf{F}}}

\DeclareMathOperator{\amod}{\mathrm{-}mod}
\DeclareMathOperator{\smod}{\mathrm{-}smod}

\DeclareMathOperator{\lfmod}{\mathrm{-}mod_{lf}}
\DeclareMathOperator{\fgmod}{\mathrm{-}mod_{lfg}}
\DeclareMathOperator{\lfmods}{\mathrm{-}smod_{lf}}

\DeclareMathOperator{\lfmodsl}{\mathrm{-}smod_{lf}^\lambda}

\DeclareMathOperator{\prmod}{\mathrm{-}pmod_{lfg}}

\DeclareMathOperator{\prmodsl}{\mathrm{-}psmod_{lfg}^\lambda}

\DeclareMathOperator{\fgmods}{\mathrm{-}smod_{lfg}}
\DeclareMathOperator{\End}{End}
\DeclareMathOperator{\Ext}{Ext}
\DeclareMathOperator{\gdim}{gdim}
\DeclareMathOperator{\sdim}{sdim}

\newtheorem{thm}{Theorem}[section]
\newtheorem{lem}[thm]{Lemma}
\newtheorem{cor}[thm]{Corollary}
\newtheorem{prop}[thm]{Proposition}
\newtheorem{exe}[thm]{Example}

\theoremstyle{definition}
\newtheorem{defn}[thm]{Definition}

\newtheorem{rem}[thm]{Remark}

\definecolor{bordeaux}{rgb}{0.6,0,0.3}
\definecolor{rose}{rgb}{0.8,0,0.6}
\definecolor{darkmagenta}{rgb}{0.5,0,0.6}
\definecolor{dmarine}{rgb}{0.3,0.5,0.8}
\definecolor{gmarine}{rgb}{0.1,0.6,0.3}
\definecolor{dblue}{rgb}{0,0,0.5}
\definecolor{dred}{rgb}{0.9,0,0}
\definecolor{dgreen}{rgb}{0,0.6,0}
\definecolor{myred}{rgb}{0.9,0,0}
\definecolor{mygreen}{rgb}{0,0.7,0}

\newcommand{\tM}{\mathfrak{M}}

\newcommand{\un}{\mathbbm{1}}

\newcommand{\bN}{\mathbb{N}}
\newcommand{\bZ}{\mathbb{Z}}
\newcommand{\bQ}{\mathbb{Q}}
\newcommand{\bR}{\mathbb{R}}
\newcommand{\bC}{\mathbb{C}}

\newcommand{\cA}{\mathcal{A}}

\newcommand{\cC}{\mathcal{C}}
\newcommand{\cD}{\mathcal{D}}

\newcommand{\cK}{\mathcal{K}}

\newcommand{\cM}{\mathcal{M}}
\newcommand{\cN}{\mathcal{N}}

\newcommand{\cU}{\mathcal{U}}
\newcommand{\cV}{\mathcal{V}}

\DeclareMathOperator{\Hom}{Hom}
\DeclareMathOperator{\HOM}{HOM}
\DeclareMathOperator{\RHom}{\mathbf RHom}

\DeclareMathOperator{\Image}{im}
\DeclareMathOperator{\id}{Id}

\DeclareMathOperator{\Ind}{Ind}
\DeclareMathOperator{\Res}{Res}

\newcommand{\xra}[1]{\xrightarrow{#1}}

\title{An approach to categorification of Verma modules}
\author{Gr\'egoire Naisse}
\address{Institut de Recherche en Math\'ematique et Physique\\
Universit\'e Catholique de Louvain\\ 
Chemin du Cyclotron 2\\ 
1348 Louvain-la-Neuve\\ 
Belgium}
\email{gregoire.naisse@uclouvain.be}
\author{Pedro Vaz}
\address{Institut de Recherche en Math\'ematique et Physique\\
Universit\'e Catholique de Louvain\\ 
Chemin du Cyclotron 2\\ 
1348 Louvain-la-Neuve\\ 
Belgium}
\email{pedro.vaz@uclouvain.be}
\begin{document}
 \usetikzlibrary{decorations.pathreplacing,backgrounds,decorations.markings}
\tikzset{wei/.style={draw=red,double=red!40!white,double distance=1.5pt,thin}}
\tikzset{bdot/.style={fill,circle,color=blue,inner sep=3pt,outer sep=0}}
%
\newdimen\captionwidth\captionwidth=\hsize
%
%
\begin{abstract}
We give a geometric categorification of the Verma modules $M(\lambda)$ for 
quantum $\slt$. 
\end{abstract}
\maketitle

{\hypersetup{hidelinks}
\tableofcontents
}
%
%
\pagestyle{myheadings}
\markboth{\em\small Gr\'egoire Naisse and Pedro Vaz}{\em\small Categorification of Verma modules}
%
%
%
%

\section{Introduction}\label{sec:intro}

After the pioneer works of Frenkel-Khovanov-Stroppel~\cite{fks} and Chuang-Rouquier~\cite{CR},
there have been various 
developments in higher representation theory in different directions and with different flavors.
One popular approach consists of using structures from geometry to construct
ca\-te\-go\-rical actions of Lie algebras. 
This is already present in the foundational papers~\cite{CR} and~\cite{fks}, where cohomologies of finite-dimensional
Grassmannians and partial flag varieties play an important role in the categorification 
of finite-dimensional irreducible representations of quantum $\slt$.
In the context of algebraic geometry Cautis, Kamnitzer and Licata~\cite{ckl-cotangent} have defined and studied geome\-tric
categorical $\slt$-actions, and
Zheng gave a categorification of integral representations of quantum groups~\cite{zheng2} and of tensor products of
$\slt$-modules~\cite{zheng}. 

\medskip 

In a remarkable series of papers, Khovanov and Lauda, and independently Rouquier, constructed categorifications
of all quantum Kac-Moody algebras~\cite{KL1,KL2,KL3,L1,L2,L3,R1} and some of their 2-rep\-re\-sen\-ta\-tions~\cite{R1}. 
In Khovanov and Lauda's formulation the categorified quantum group is a 2-category $\dot{\cU}$,
defined diagrammatically
by generators and relations. Khovanov and Lauda conjectured that certain quotients of $\dot{\cU}$ categorify integrable representations
of these quantum Kac-Moody algebras.
This conjecture was first proved in finite type $A_n$ by Brundan and Kleshshev~\cite{BK}. 
Based on Khovanov and Lauda, Rouquier and Zheng's work, Webster gave in~\cite{webster} a diagrammatic
categorification of tensor products of integrable representations of symmetrizable quantum Kac-Moody algebras and used it to categorify the Witten-Reshetikhin-Turaev link invariant.
Moreover, he proved Khovanov and Lauda's conjecture on categorification of integrable representations for all quantum Kac-Moody algebras.  
This was also done independently by Kang and Kashiwara in~\cite{kk}.

\medskip

All these constructions share one common feature: they only categorify finite-dimensional representations
in the finite type (representations in other types  are 
locally finite-dimensional when res\-tricted to any simple root quantum $\slt$).   
In this paper we make a step towards a categorification of infinite dimensional and non-integrable
representations of quantum Kac-Moody algebras. 
We start with the simplest case by proposing a framework for categorification of the Verma modules for quantum $\slt$.

\medskip

This paper was motivated by the 
geometric categorification of finite-dimensional irreducible representations of $\slt$ 
by Frenkel, Khovanov and Stroppel in~\cite[\S6]{fks} and the subsequent
constructions in~\cite{L1} and~\cite{L2}. In particular, 
the inspiration from Lauda's description in~\cite{L1,L2} should be clear.

\medskip  

Without further delays we now pass to describe our construction.  

\subsection{Sketch of the construction}

\subsubsection{Verma modules}\label{ssec:vermam}

Let  $\mathfrak{b}$ be the Borel subalgebra of $\slt$, 
and $\lambda = q^c$ for some $c$ either formal or integer.
Denote by $V_{\lambda}$ 
 the 1-dimensional
$U_q(\mathfrak{b})$-module of weight $\lambda$, with $E$ acting trivially. 
The universal Verma module $M(\lambda)$ 
with highest weight $\lambda$
 is the induced module
\[
M(\lambda) = U_q(\slt)\otimes_{U_q(\mathfrak{b})}V_{\lambda} .
\] 
We follow the notation in~\cite{Jantzen}
(cf. \S2.2. and \S2.4. in~\cite{Jantzen}), 
but in the special case when $\lambda = q^n$ for an integer number $n$, we write $M(n)$ instead of $M(q^n)$.
The Verma module $M(\lambda)$
 is
irreducible unless $\lambda = q^n$ with $n$ a non-negative integer. 
In the latter case $M(n)$ contains $M(-n-2)$ as a unique nontrivial proper submodule 
and the quotient $M(n)/M(-n-2)$ is isomorphic to the irreducible $U_q(\slt)$-module $V(n)$
of dimension $n+1$.  
Throughout this paper, we will treat $\lambda$ as a formal parameter and we will think of $M(\lambda)$ as a module over $\bQ\pp{ q, \lambda}$. Here $\bQ\pp{ q, \lambda}$ means the field of formal Laurent series in the variables $q$ and $\lambda$. We will also consider $M_A(\lambda)$ and $M^*_A(\lambda)$, where we replace the ground field $\bQ\pp{q,\lambda}$ by the ring $A = \bQ\pp{q}[\lambda,\lambda^{-1}]$. They are Verma modules over $\dot U_\lambda$, the shifted idempotented quantum $\slt$ defined below in \S\ref{ssec:defidemp}, and are given respectively by the canonical and dual canonical basis of $M(\lambda)$, also presented in \S\ref{ssec:defidemp}.


\subsubsection{Categorification of the weight spaces of $M(\lambda q^{-1})$}

We work over the field of rationals $\bQ$ and $\otimes$ means $\otimes_\bQ$.  
Let $G_{k}$ be the Grassmannian of $k$-planes in $\bC^\infty$ and $H(G_k)$ its cohomology ring
with rational coefficients. It is a graded algebra freely generated by the Chern classes
$\und{x}_k=(x_1,\dotsc , x_k)$ with $\deg(x_i)=2i$~\cite[\S3.1.1]{L2} (see also \cite[\S3]{hiller} and \cite{fulton} for more about cohomology of flag varieties).
The ring $H(G_k)$ has a unique irreducible module up to isomorphism and grading shift,
which is isomorphic to $\bQ$.  
Let $\Ext_{H(G_k)}(\bQ,\bQ)$ denote the algebra of self-extensions of $\bQ$
(this can be seen as the opposite algebra of the Koszul dual of $H(G_k)$) and for $k\geq 0$, 
define   
\[ 
\Omega_k =  H(G_k)\otimes \Ext_{H(G_k)}(\bQ,\bQ). 
\]
We have 
\[ 
\Omega_k \cong \bQ[\und{x}_k]\otimes\bV^\bullet(\und{s}_k) , 
\] 
which we regard as a $\bZ\times\bZ$-graded superring, with even generators 
$x_i$  having degree  $\deg(x_i)=(2i,0)$, and odd generators $s_i$ with $\deg(s_i)=(-2i,2)$   
(the first grading is quantum and the second cohomological). We denote by $\brak{r,s}$ the grading shift 
up by $r$ units on the quantum grading and by $s$ units on the cohomological grading.  
In the sequel we use the term bigrading for a  $\bZ\times\bZ$-grading.  

\medskip 

The superring $\Omega_k$ has a unique irreducible supermodule up to isomorphism and (bi)grading shift,
which is isomorphic to $\bQ$ and denoted $S_k$, 
and a unique projective indecomposable supermodule, again up to isomorphism and (bi)grading shift,
which is isomorphic to $\Omega_k$.

\medskip 


In \S\ref{sec:completedgrothendieck} we develop several versions of ``topological" Grothendieck groups. 
The topological split Grothendieck group
$K_0(\Omega_k)$ 
and topological Grothendieck group $G_0(\Omega_k)$ 
are one-dimensional modules over $\bZ_\pi\llbracket q \rrbracket[q^{-1},\lambda^{\pm 1}]$,
where $\bZ_\pi = \bZ[\pi]/(\pi^2-1)$, and generated respectively by the class of $\Omega_k$,
and by the class of $S_k$. 
In another version, the topological Grothendieck group   
$\widehat G_0(\Omega_k) = \boldsymbol G_0(\Omega_k\lfmods)$ is a one-dimensional module
over $\bZ_\pi\pp{q,\lambda}$, and is generated
either by~$[\Omega_k]$,
either by~$[S_k]$.

\medskip

For each non-negative integer $k$ we define \( \cM_k = \Omega_k\lfmods  \) and take $\cM_k$ as a 
categorification of the $(\lambda q^{-1-2k})$th-weight space.

\subsubsection{The categorical $\slt$-action}\label{ssec:catactslt}

To construct functors $\F$ and $\E$ that move between categories $\cM_k$ we look for superrings 
$\Omega_{k+1,k}$ and (natural) maps
\[\xymatrix
{
 & \Omega_{k+1,k} & 
\\
\Omega_{k+1}\ar[ur]^{\psi^*_{k+1}}  && \Omega_k\ar[ul]_{\phi^*_k}
}
\]
that turn $\Omega_{k+1,k}$ into a $(\Omega_{k+1},\Omega_k)$-superbimodule such that, up to an overall shift,   
\begin{itemize}
\item $\Omega_{k+1,k}$ is a free right $\Omega_k$-supermodule of bigraded superdimension 
$\frac{\lambda q^{-k-1}-\lambda^{-1}q^{k+1}}{ q-q^{-1} }$,\\ 
\item $\Omega_{k+1,k}$ is a free left $\Omega_{k+1}$-supermodule of bigraded superdimension  $[k+1]$.
\end{itemize}

\begin{rem}
The superring  $H(G_{k,k+1})\otimes\Ext_{H(G_{k,k+1})}(\bQ,\bQ)$ does not have these properties.  
\end{rem} 

\medskip 

Let $G_{k,k+1}$ be the infinite partial flag variety 
\[
G_{k,k+1} = \{ (U_k,U_{k+1}) \vert \dim_{\bC}U_k = k, \dim_{\bC}U_{k+1} = k+1, \ 
0 \subset U_k \subset U_{k+1} \subset \bC^{\infty} \}. 
\]
Its rational cohomology is a graded ring, generated by the Chern classes:
\[ 
H(G_{k,k+1}) = \bQ[\und{x}_k,\xi], \mspace{50mu} \deg(x_i)= 2i,\ \ \deg(\xi)=2 .
\] 

The forgetful maps 
\[
\xymatrix
{
 & G_{k,k+1}\ar[dr]^{p_k}\ar[dl]_{p_{k+1}} & 
\\
G_{k+1} && G_{k}
}
\]
induce maps in cohomology
\[\xymatrix
{
 & H(G_{k,k+1}) & 
\\
H(G_{k+1})\ar[ur]^{\psi_{k+1}}  && H(G_{k})\ar[ul]_{\phi_k}
}
\]
which make $H(G_{k,k+1})$ an $(H(G_{k+1}),H(G_k))$-superbimodule.   
As a right $H(G_k)$-supermodule, the bimodule $H(G_{k,k+1})$ is a free, bigraded module isomorphic to $H(G_k)\otimes \bQ[\xi]$.

\medskip 


We take 
\[
\Omega_{k+1,k}  = H(G_{k,k+1}) \otimes \Ext_{H(G_{k+1})}(\bQ,\bQ) .
\]
We put $\psi^*_{k+1}=\psi_{k+1}\otimes 1 \colon \Omega_{k+1}\to \Omega_{k+1,k}$ and 
define $\phi^*_{k}\colon \Omega_{k}\to \Omega_{k+1,k}$ as the map sending
$x_i$ to $x_i$ and $s_i$ to $s_i+\xi s_{i+1}$:   
\[\xymatrix
{
 & \Omega_{k+1,k} & 
\\
\Omega_{k+1}\ar[ur]^{\psi^*_{k+1} =\psi_{k+1}\otimes 1\ \ }  && \Omega_k\ar[ul]_{\;\;\phi^*_k}.
}
\]
This gives $\Omega_{k+1,k}$ the structure of an $(\Omega_{k+1},\Omega_k)$-superbimodule.  
We write $\Omega_{k,k+1}$ for $\Omega_{k+1,k}$ when seen as an  $(\Omega_{k},\Omega_{k+1})$-superbimodule. 
It is easy to see that up to an overall shift, the superring $\Omega_{k+1,k}$ has the desired properties. 

\medskip

For each $k\geq 0$ define exact functors\footnote{All our functors are, in fact, superfunctors which we
  tend to see as functors between categories endowed with a $\bZ/2\bZ$-action, whence the use of the
  terminology \emph{functor.}} 
$\F_k \colon \cM_k \to \cM_{k+1}$  
and 
$\E_k\colon \cM_{k+1}\to \cM_{k}$ 
by 
\begin{align*}
\F_k(-) &= \Res_{k+1}^{k+1,k}\circ\, \Omega_{k+1,k}\otimes_{\Omega_k}(-)\brak{-k,0} ,
\intertext{and} 
\E_k(-) &= \Res_{k}^{k+1,k}\circ\,  \Omega_{k,k+1}\otimes_{\Omega_{k+1}}(-)\brak{k+2,-1} . 
\end{align*}

\medskip

\n Functors $(\F,\E)$ form an adjoint pair up to grading shifts, 
but $\F$ does not admit $\E$ as a left adjoint. 
We would like to  stress that this is necessary to prevent us from falling in the situation 
of $Q$-strong 2-representations from~\cite{cl}. 
In that case, the construction would lift to a 2-representation of $\dot{\cU}$ and
Rouquier's results in~\cite{R1} would imply that if the functor $\E$ kills a highest weight then its
biadjoint functor would kill a lowest weight  
(this can be proved with a clever trick using degree zero bubbles from~\cite{L1} 
to tunnel from the lowest weight to the highest weight\footnote{We thank Aaron Lauda for explaining this
  to us.}).
We should not expect a biadjunction between the functors $\E$ and $\F$ since it 
can be interpreted as a categorification of the involution exchanging
operators $E$ and $F$ (up to coefficients).  

\medskip

Denote by $\Omega_k[\xi]$ the polynomial ring in $\xi$ with coefficients in $\Omega_k$ 
and by $\Q_k$ the functor of tensoring on the left 
with the $(\Omega_k,\Omega_k)$-superbimodule 
$\Pi\Omega_k[\xi]\brak{1,0}$, where $\Pi$ is the parity change functor. 
The categorical $\slt$-action is encoded in a short exact sequence of functors,
\begin{equation} \label{eq:SESintro}
0 \xra{\quad} 
\F_{k-1}\circ\E_{k-1} 
\xra{\quad}
\E_k\circ\F_k  
\xra{\quad} 
\Q_k\brak{-2k-1,1}\ \oplus \Pi\Q_k\brak{2k+1,-1} 
\xra{\quad} 0 .
\end{equation}

From the work of Lauda in~\cite{L1,L2} adjusted to our context it follows that for each $n\geq 0$ 
there is an action of the nilHecke algebra $\nh_n$ on $\F^n$ and on $\E^n$. 
As a matter of fact, there is an enlargement of $\nh_n$, which we denote $A_n$, acting on $\F^n$ and $\E^n$, 
also admitting a nice diagrammatic description (see \S\ref{ssec:diagintro} below for a sketch).  

\medskip 

We define $\cM$ as the direct sum of all the $\cM_k$'s and functors
$\F$, $\E$ and $\Q$ in the obvious way. One of the main results in this paper is the following. 
\newtheorem*{thma}{Theorem~\ref{thm:vermacatKo}}

\begin{thma}
   The functors $\F$ and $\E$ induce an action of quantum $\slt$ on the Grothendieck groups $K_0(\cM)$, $G_0(\cM)$ and $\widehat G_0(\cM)$, after specializing $\pi = -1$. With this action there are $\bQ\llbracket q \rrbracket[q^{-1},\lambda^{\pm 1}]$-linear isomorphisms
\begin{align*}
K_0(\cM) \cong M_A(\lambda),\mspace{100mu} G_0(\cM) &\cong M_A^*(\lambda),
\end{align*}
of $\dot{U}_\lambda $ module and a $\bQ\pp{q, \lambda}$-linear isomorphism 
\[
\widehat G_0(\cM) \cong M(\lambda),
\]
of $U_q(\slt)$-representations.
  Moreover, these isomorphisms send classes of projective indecomposables to canonical basis elements and
  classes of simples to dual canonical elements, whenever this makes sense.
\end{thma}

Form the
2-category $\tM(\lambda q^{-1})$ which is the completion under extensions of the 2-category whose objects 
are the ca\-te\-go\-ries $\cM_k$, the 1-morphisms are cone bounded, locally finite direct sums of shifts of functors
from $\{E_k, F_k, Q_k, \id_k\}$ and the 
2-morphisms are (grading preserving) natural transformations of functors. 
In this case the 2-category $\tM(\lambda q^{-1})$ is an example of a \emph{2-Verma module} for  $\slt$.

\medskip

\subsubsection{Categorification of the Verma module with integral highest weight}

Forgetting the cohomo\-lo\-gical degree on the superrings $\Omega_k$ and $\Omega_{k,k+1}$ defines a forgetful functor  
into the category of $\bZ$-graded $\Omega_k(-1)$-supermodules, 
where $\Omega(-1)$ is the ring $\Omega$ with the cohomological degree collapsed. This defines a category $\cM(-1)$. 
A direct consequence of Theorem 5.11 is that the Grothendieck group $K_0(\cM)$ is  
isomorphic to the Verma module $M(-1)$.  

\medskip

Our strategy to categorify $M(n)$ is to first define for each $n\in\bZ$  certain sub-superrings 
$\Omega^n_k$ and $\Omega_{k,k+1}^n$ of $\Omega_k$ and $\Omega_{k,k+1}$, that agree with these for $n=-1$,  
and such that an immediate application of the procedure as before 
results in a categorification of the Verma module $M(\lambda q^n)$. 
We then apply the forgetful functor to define a categorification of $M(n)$.

\medskip
As in the case of $\cM$, for each $m\geq 0$ 
there is an action of the nilHecke algebra $\nh_m$ and of its enlargement $A_m$ on $\F^m$ and on $\E^m$.

\subsubsection{A categorification of the $(n+1)$-dimensional irreducible representation from $\cM(\lambda)$}

To recover the categorification of the finite dimensional irreducible $V(n)$ from~\cite{CR} and~\cite{fks} 
we define for each $n\in\bN$ a differential $d_n$ on the superrings $\Omega_k$ and $\Omega_{k,k+1}$, turning them in DG-algebras and DG-bimodules. 
These DG-algebras and DG-bimodules are quasi-isomorphic to the cohomologies 
of finite dimensional Grassmannians and 1-step flag manifolds in $\bC^n$, 
as used in~\cite{CR,fks}. 
Moreover, the short exact sequence~(\ref{eq:SESintro}) can be turned into a short exact sequence of DG-bimodules that descends in the homology to the 
direct sums decompositions categorifying the $\slt$-commutator relation in~\cite{CR,fks}. 
The nilHecke algebra action descends to the usual nilHecke algebra action on integrable 2-representations of $\slt$ 
(see~\cite{cl,L1,L2,R1}).  
The differential $d_n$ descends to the superrings $\Omega_k(n)$ and $\Omega_{k,k+1}(n)$ yielding the 
same result as in $\cM(\lambda)$.

\subsubsection{A diagrammatic presentation for the enlargement of the nilHecke algebra}\label{ssec:diagintro}

The enlarged nilHecke algebra $A_{n}$ can be  given a presentation in the spirit of
KLR algebras~\cite{KL1,L1,R1} as 
isotopy classes of braid-like diagrams modulo some relations.  

\medskip

Our diagrams are isotopy classes of KLR diagrams with some extra structure. 
Besides the KLR dots we have another type of dot we call a \emph{white dot} 
(we keep the name \emph{dot} for the KLR dots).
White dots are made to satisfy the exterior algebra relations (see below). 
The enlarged algebra is in fact a bigraded superalgebra, where nilHecke generators 
are even and white dots are odd.  
Moreover, regions in a diagram are labeled with integer numbers, 
and crossing a strand from left to right increases the label by one:
\[ 
\tikz[very thick,xscale=1.2,baseline={([yshift=-.5ex]current bounding box.center)}]{
          \draw (0,-.6)-- (0,.6) ;
    	  \node at (-0.5,0){$k$};
    	  \node at ( 0.8,0){$k+1$};
        }
\]
Fix a base ring $\Bbbk$. The $\Bbbk$-superalgebra $A_n$ consists of $\Bbbk$-linear combinations of $n$-strand diagrams as described above.  
The multiplication is given by concatenation of diagrams whenever the labels of the regions 
agree and zero otherwise. 
The superalgebra $A_n$ is bigraded with the $q$-degree of the white dot given by minus two times the label of the region at his right:  
\begin{align*}
\deg\left(
\tikz[very thick,xscale=1.2,baseline={([yshift=-.5ex]current bounding box.center)}]{
          \draw (0,-.5)-- (0,.5) node [midway,fill=black,circle,inner
          sep=2pt]{};
    	  \node at (-0.35,0){$k$};
        }\quad
\right) &= (2,0), &
\deg\left(
\tikz[very thick,xscale=1.2,baseline={([yshift=-.5ex]current bounding box.center)}]{
          \draw (0,-.5)-- (0,.5) node [midway,fill=white, draw=black,circle,inner
          sep=2pt]{};
    	  \node at (-0.35,0){$k$};
        } \quad
\right) &= (-2k-2,2), &
\deg\left(
\tikz[very thick,xscale=1.2,baseline={([yshift=-.5ex]current bounding box.center)}]{
          \draw (.1,-.5)-- (.9,.5);
          \draw (.9,-.5)-- (.1,.5);
    	  \node at (-0,0){$k$};
        }\quad
\right) &= (-2,0).
\end{align*}
The generators are subject to the following local relations:
\begin{gather*}
	\tikz[very thick,xscale=1.2,baseline={([yshift=-.5ex]current bounding box.center)}]{
	          \draw (.1,-.5)-- (.1,.5) node [near start,fill=white, draw=black,circle,inner sep=2pt]{};
      	 	   \node at (.5,0){$\cdots$};
	          \draw (.9,-.5)-- (.9,.5) node [near end,fill=white, draw=black,circle,inner sep=2pt]{};
	    	  \node at (-0.25,0){$k$};
  	}
	\quad=\quad -\quad
	\tikz[very thick,xscale=1.2,baseline={([yshift=-.5ex]current bounding box.center)}]{
	          \draw (.1,-.5)-- (.1,.5) node [near end,fill=white, draw=black,circle,inner sep=2pt]{};
      	 	   \node at (.5,0){$\cdots$};
	          \draw (.9,-.5)-- (.9,.5) node [near start,fill=white, draw=black,circle,inner sep=2pt]{};
	    	  \node at (-0.25,0){$k$};
	} \mspace{10mu},\mspace{90mu}
	\tikz[very thick,xscale=1.2,baseline={([yshift=-.5ex]current bounding box.center)}]{
	          \draw (.1,-.5)-- (.1,.5) node [near start,fill=white, draw=black,circle,inner sep=2pt]{}
			node [near end,fill=white, draw=black,circle,inner sep=2pt]{};
	    	  \node at (-0.25,0){$k$};
	}
	\quad=\quad 0 \mspace{10mu},
	\\[1ex]
	\tikz[very thick,xscale=1.2,baseline={([yshift=-.5ex]current bounding box.center)}]{
		\node at (-0,0){$k$};
	      \draw  +(.1,-.75) .. controls (1,0) ..  +(.1,.75);
	      \draw  +(.9,-.75) .. controls (0,0) ..  +(.9,.75);
	 }
	\quad=\quad 0 \mspace{10mu},\mspace{90mu}
	\tikz[very thick,xscale=1.2,baseline={([yshift=-.5ex]current bounding box.center)}]{
		\node at (-0.1,0){$k$};
	     	 \draw  +(.75,-.75) .. controls (0,0) ..  +(.75,.75);
		 \draw (0,-.75)-- (1.5,.75);
	          \draw (0,.75)-- (1.5,-.75);
	 }
	\quad=\quad 
	 \tikz[very thick,xscale=1.2,baseline={([yshift=-.5ex]current bounding box.center)}]{
		\node at (-0,0){$k$};
	     	 \draw  +(.75,-.75) .. controls (1.5,0) ..  +(.75,.75);
		 \draw (0,-.75)-- (1.5,.75);
	          \draw (0,.75)-- (1.5,-.75);
	}\mspace{10mu},
\end{gather*}

\begin{gather*}
	\tikz[very thick,xscale=1.2,baseline={([yshift=-.5ex]current bounding box.center)}]{
	          \draw (.1,-.5)-- (.9,.5);
	          \draw (.9,-.5)-- (.1,.5) node [near end,fill=black,circle,inner sep=2pt]{};
	    	  \node at (-0,0){$k$};
	}
	\quad=\quad
	\tikz[very thick,xscale=1.2,baseline={([yshift=-.5ex]current bounding box.center)}]{
	          \draw (.1,-.5)-- (.9,.5);
	          \draw (.9,-.5)-- (.1,.5) node [near start,fill=black,circle,inner sep=2pt]{};
	    	  \node at (-0,0){$k$};
	}
	\quad + \quad
	\tikz[very thick,xscale=1.2,baseline={([yshift=-.5ex]current bounding box.center)}]{
	          \draw (.1,-.5)-- (.1,.5);
	          \draw (.9,-.5)-- (.9,.5);
	    	  \node at (-0.25,0){$k$};
	} \mspace{10mu},
	\\[1ex]
	\tikz[very thick,xscale=1.2,baseline={([yshift=-.5ex]current bounding box.center)}]{
	          \draw (.1,-.5)-- (.9,.5) node [near start,fill=black,circle,inner
	          sep=2pt]{};
	          \draw (.9,-.5)-- (.1,.5);
	    	  \node at (-0,0){$k$};
	} 
	\quad=\quad
	\tikz[very thick,xscale=1.2,baseline={([yshift=-.5ex]current bounding box.center)}]{
	          \draw (.1,-.5)-- (.9,.5) node [near end,fill=black,circle,inner
	          sep=2pt]{};
	          \draw (.9,-.5)-- (.1,.5);
	    	  \node at (-0,0){$k$};
	}
	\quad + \quad
	\tikz[very thick,xscale=1.2,baseline={([yshift=-.5ex]current bounding box.center)}]{
	          \draw (.1,-.5)-- (.1,.5);
	          \draw (.9,-.5)-- (.9,.5);
	    	  \node at (-0.25,0){$k$};
	} \mspace{10mu},
\end{gather*}

\begin{gather*}
	\tikz[very thick,xscale=1.2,baseline={([yshift=-.5ex]current bounding box.center)}]{
	          \draw (.1,-.5)-- (.9,.5);
	          \draw (.9,-.5)-- (.1,.5) node [near start,fill=white, draw=black,circle,inner sep=2pt]{};
	    	  \node at (-0,0){$k$};
	} 
	\quad = \quad
	\tikz[very thick,xscale=1.2,baseline={([yshift=-.5ex]current bounding box.center)}]{
	          \draw (.1,-.5)-- (.9,.5) node [near end,fill=white, draw=black,circle,inner sep=2pt]{};
	          \draw (.9,-.5)-- (.1,.5) ;
	    	  \node at (-0,0){$k$};
	} \mspace{10mu},  \\[1ex]
	\tikz[very thick,xscale=1.2,baseline={([yshift=-.5ex]current bounding box.center)}]{
	          \draw (.1,-.5)-- (.9,.5) node [near start,fill=white, draw=black,circle,inner sep=2pt]{};
	          \draw (.9,-.5)-- (.1,.5) ;
	    	  \node at (-0,0){$k$};
	}
	\quad + \quad
	\tikz[very thick,xscale=1.2,baseline={([yshift=-.5ex]current bounding box.center)}]{
	          \draw (.1,-.5)-- (.9,.5) node [pos=.85,fill=white, draw=black,circle,inner sep=2pt]{}
				 node [pos=.65,fill=black,circle,inner sep=2pt]{};
	          \draw (.9,-.5)-- (.1,.5) ;
	    	  \node at (-0,0){$k$};
	}
	\quad = \quad
	\tikz[very thick,xscale=1.2,baseline={([yshift=-.5ex]current bounding box.center)}]{
	          \draw (.1,-.5)-- (.9,.5);
	          \draw (.9,-.5)-- (.1,.5) node [near end,fill=white, draw=black,circle,inner sep=2pt]{};
	    	  \node at (-0,0){$k$};
	}
	\quad + \quad
	\tikz[very thick,xscale=1.2,baseline={([yshift=-.5ex]current bounding box.center)}]{
	          \draw (.1,-.5)-- (.9,.5);
	          \draw (.9,-.5)-- (.1,.5) node [pos=.15,fill=white, draw=black,circle,inner sep=2pt]{}
				 node [pos=.35,fill=black,circle,inner sep=2pt]{};
	    	  \node at (-0,0){$k$};
	} \mspace{10mu}. 
\end{gather*}
All other isotopies are allowed (e.g. switching the relative height of a dot and a white dot).
The relations above respect the bigrading as well as the parity. 

\medskip 

We define $A_{n}(m)$ as the sub-superalgebra consisting of all diagrams with label $m$ at the
leftmost region and 
\[ 
A(m) = \bigoplus_{n\geq 0}A_{n}(m). 
\]
The usual inclusion $A_{n}(m)\hookrightarrow A_{n+1}(m)$ that adds a strand at the right of a diagram from $A_{n}(m)$ 
gives rise to induction and restriction functors 
on $A(m)\lfmods$ that satisfy the $\slt$-relations. 
Our results imply that together with this functors,
$A(m)\lfmods$ categorifies the Verma module $M(\lambda q^{m-1})$. 
The categorification of Verma modules with integral highest weight using the algebras $A_n$ follows 
as a consequence of our results.
Moreover, we define for each $m \in \bN$ a differential on $A_n(m)$ turning it into a DG-algebra, which is quasi-isomorphic to a cyclotomic quotient of the nilHecke algebra.

%
%

\subsection{Acknowledgments}
The authors would like to thank Mikhail Khovanov and Marco Mackaay for helpful discussions, exchanges of email, 
and for comments on an early stage of this project and to Qi You, 
Antonio Sartori, Catharina Stroppel, Daniel Tubbenhauer and Peng Shan
for helpful discussions, comments and suggestions. 
We would also like to thank Hoel Queffelec and Can Ozan for comments on earlier versions of this paper. 
G.N. is a Research Fellow of the Fonds de la Recherche Scientifique - FNRS, under Grant no.~1.A310.16.
P.V. was supported by the Fonds de la Recherche Scientifique - FNRS under Grant no.~J.0135.16.

\section{$U_q(\slt)$ and its representations}\label{sec:sl2}

\subsection{Forms of quantum $\slt$}

The notions below are well-known and can be found for example in~\cite{Jantzen} or~\cite{lus}.  
\begin{defn} 
The \emph{quantum algebra $U_q(\slt)$} is the unital associative algebra over $\bQ(q)$ with
ge\-ne\-ra\-tors $E$, $F$, $K$ and $K^{-1}$ subject to the relations:
\begin{align*}
KK^{-1} &= 1 =  K^{-1}K , 
\\
KE &= q^2 EK , 
\\
KF &= q^{-2} FK , 
\\
EF - FE &= \frac{K-K^{-1}}{q-q^{-1}} . 
\end{align*}
We denote by $U_q(\mathfrak{b})$ the subalgebra of $U_q(\slt)$ generated by $E$, $K$ and $K^{-1}$. 
\end{defn}

Define the quantum integer $[a]=\tfrac{q^a-q^{-a}}{q-q^{-1}}$,
the quantum factorial $[a]!=[a][a-1]!$ with $[0]!=1$,  
and the quantum binomial coefficient 
$\qbin{a}{b}=\tfrac{[a]!}{[b]![a-b]!}$ for $0\leq b\leq a$, and put $\{a\}=q^{a-1}[a]$.  
For $a\geq 0$ define also the divided powers  
\begin{equation*}
E^{(a)} = \frac{E^a}{[a]!} 
\mspace{40mu}\text{and}\mspace{40mu} 
F^{(a)} = \frac{F^a}{[a]!} .
\end{equation*}

\smallskip

Following~\cite[\S2.1-2.2]{L1} we now introduce some important algebra (anti)automorphisms 
on $U_q(\slt)$. 
Let $\overline{\textcolor{white}{q}}$, $\psi$, $\tau$  and $\rho$ be as follows:

\begin{itemize}
\item $\overline{\textcolor{white}{q}}$ is the $\bQ$-linear involution that maps $q$ to $q^{-1}$.
\item $\psi$ is the $\bQ(q)$-antilinear algebra automorphism of $U_q(\slt)$ given by
\begin{align*}
&\psi(E)=E,\mspace{60mu} \psi(F)=F,\mspace{60mu} \psi(K)=K^{-1}, 
\\
&\psi(pX)=\overline{p}\psi(X), \quad \text{ for $p\in\bQ(q)$ and $X\in U_q(\slt)$.}
\end{align*}
\item $\tau\colon U_q(\slt) \to U_q(\slt)^{op}$ is the $\bQ(q)$-antilinear isomorphism given by
\begin{equation}\label{eq:antitau}
\tau(E) = q^{-1}K^{-1}F, 
\mspace{60mu} 
\tau(F) = q^{-1}KE , 
\mspace{60mu} 
\tau(K)=K^{-1} ,
\end{equation}
and 
\begin{align*}
&\tau(pX)=\overline{p}\tau(X),  \quad \text{ for $p\in\bQ(q)$ and $X\in U_q(\slt)$},
\\
&\tau(XY)= \tau(Y)\tau(X),  \quad \text{ for $X,Y \in U_q(\slt)$}.
\end{align*}
\item $\rho$ is the $\bQ(q)$-linear algebra anti-involution defined by
\begin{equation}\label{eq:antirho}
\rho(E) = q^{-1}K^{-1}F, 
\mspace{60mu} 
\rho(F) = q^{-1}KE , 
\mspace{60mu} 
\rho(K)=K,
\end{equation}
and 
\begin{align*}
&\rho(pX)=p\rho(X),  \quad \text{ for $p\in\bQ(q)$ and $X\in U_q(\slt)$},
\\
&\rho(XY)= \rho(Y)\rho(X),  \quad \text{ for $X,Y \in U_q(\slt)$}.
\end{align*}
\end{itemize}
The inverse of $\tau$ is given by $\tau^{-1}(E)=q^{-1}FK$, $\tau^{-1}(F)=q^{-1}EK^{-1}$, and
$\tau^{-1}(K)=K^{-1}$.

\begin{rem}
The $\rho$ defined above should be $\psi \rho \psi$ in the notations from \cite{L1}.
\end{rem}

\subsubsection{Deformed idempotented $U_q(\slt)$}\label{ssec:defidemp}

For $c$ either integer or formal parameter, the shifted weight lattice is given by $c+\bZ$.
For $n\in\bZ$ we denote by $e_n$ the idempotent corresponding to the projection onto the $(\lambda q^n)$-th weight space. On this weight space,
$K$ acts as multiplication by $\lambda q^n$:
\begin{equation} \label{eq:idemplambda}
e_nK=Ke_n=\lambda q^{n}e_n .
\end{equation}

In the spirit of Lusztig~\cite[Chapter 23]{lus} we now adjoin to $U_q(\slt)$ the idempotents $e_n$ for all $n\in \bZ$.
 Denote by $I$ the ideal generated by the relations~\eqref{eq:idemplambda}
above together with 
\begin{equation}  
e_ne_m = \delta_{n,m}e_n , 
\mspace{40mu}
Ee_n=e_{n+2}E , 
\mspace{60mu} 
Fe_n = e_{n-2}F.
\end{equation}

\begin{defn}
Define the \emph{shifted idempotented quantum $\slt$} as the $\bQ(q)[\lambda^{\pm 1}]$-algebra
\begin{equation*}
\dot{U}_\lambda = \bigl(\bigoplus\limits_{m,n\in\bZ}e_n \left(U_q(\slt)\right) e_m \bigr)/ I.
\end{equation*}
\end{defn}

In this deformed version the main $\slt$-relation becomes 
\begin{equation}\label{eq:slcommutator}
EFe_n-FEe_n = \frac{\lambda q^{n} - \lambda^{-1}q^{-n}}{q-q^{-1}} e_n = [\lambda,n]e_n . 
\end{equation}
In the special case $\lambda = q^n$, we will see $\dot{U}_\lambda  = \dot{U}_n$ as a $\bQ(q)$-algebra. 

\smallskip


The involution $\overline{\textcolor{white}{q}}$ and the algebra maps $\psi$, $\tau$  and $\rho$
introduced above 
extend to $\dot{U}_\lambda$ if we put
\begin{equation*}
\overline{\lambda}=\lambda^{-1},
\mspace{50mu}
\psi(e_n)=e_n,
\mspace{50mu}
\tau(e_n)=e_n,
\mspace{50mu}
\rho(e_n)=e_n.
\end{equation*}
The extended versions of $\psi$, $\tau$ and $\rho$ then take the form
\begin{align*}
\psi(q^s e_{n+2}Ee_n) &= q^{-s} e_{n+2}Ee_n,
&
\psi(q^s e_{n}Fe_{n+2}) &= q^{-s} e_{n}Fe_{n+2},\\
\tau(q^s e_{n+2}Ee_n) &= \lambda^{-1} q^{-s-1-n}e_nFe_{n+2} ,
&
\tau(q^se_nFe_{n+2}) &= \lambda q^{-s+1+n}e_{n+2}Ee_n,
\intertext{and}
\rho(q^s e_{n+2}Ee_n) &= \lambda^{-1} q^{s-1-n}e_nFe_{n+2} ,
&
\rho(q^se_nFe_{n+2}) &= \lambda q^{s+1+n}e_{n+2}Ee_n.
\end{align*}

\medskip

\subsection{Representations}\label{ssec:reps-sl}

We only consider modules of type I in this paper and we follow the notations from~\cite{Jantzen}.  
As before, we write $\lambda=q^{c}$ for $c$ formal, and treat it as a formal parameter itself.
Then there is an infinite dimensional $U_q(\slt)$ highest weight-module $M(\lambda)$
with highest weight $\lambda$, called the \emph{universal Verma module}  
(as explained in \S\ref{ssec:vermam} we follow the notation 
in~\cite{Jantzen} for Verma modules for quantum groups). 
Let $\mathfrak{b}$ be the Borel subalgebra of $\slt$ and
let $\bC_\lambda=\bQ \pp{q,\lambda} v_\lambda$ be a 1-dimensional representation of $U_q(\mathfrak{b})$ with $E$ acting trivially
while $Kv_\lambda = \lambda v_\lambda$. 
The Verma module $M(\lambda)$ with highest weight $\lambda$ is the induced module 
\[
M(\lambda) =  U_q(\slt)  \otimes_{U_q(\mathfrak{b})}\bC_{\lambda} .
\]

\smallskip

\n It has basis $m_0, m_1, \dotsc, m_k, \dotsc$ 
such that for all $i\geq 0$
\begin{equation}  \label{eq:cbaction}
\begin{split}
Km_i &= \lambda q^{-2i} m_i,
\\
Fm_i &= [i+1]m_{i+1},
\\ 
Em_i &=
\begin{cases}
0  &\text{ if }  i=0,
\\
\dfrac{\lambda q^{-i+1} - \lambda^{-1}q^{i-1}}{q-q^{-1}} m_{i-1}  & \text{otherwise.}
\end{cases}
\end{split}
\end{equation}
We call this basis the \emph{canonical basis} of $M(\lambda)$. 
The change of basis $m_i'=[i]!m_i$ gives $M(\lambda)$ the following useful presentation of $M(\lambda)$:
\begin{equation}
\begin{split}
Km'_i &= \lambda q^{-2i} m'_i,
\\
Fm'_i &= m'_{i+1},
\\ 
Em'_i &=
\begin{cases}
0  &\text{ if } i=0,
\\
[i]\dfrac{\lambda q^{-i+1} - \lambda^{-1}q^{i-1}}{q-q^{-1}} m'_{i-1}  & \text{else} .
\end{cases}
\end{split} 
\end{equation}

\medskip

We denote by $M_{\alpha}$ the 1-dimensional weight spaces of weight $\alpha$. We can picture $M(\lambda)$ as the following diagram:
\[
 \xy
 (-62,0)*{\cdots};
  (-50,0)*+{M_{\lambda-2k}}="1";
  (-16,0)*+{M_{\lambda-4}}="2";
  (0,0)*+{M_{\lambda-2}}="3";
  (16,0)*+{M_{\lambda}}="4";
    {\ar@/^1.0pc/^E "1";(-34,1)*+{}};
    {\ar@/^1.0pc/^E "2";"3"};
    {\ar@/^1.0pc/^E "3";"4"};
    {\ar@/^1.0pc/^F (-34,-2)*+{}; "1"};
    {\ar@/^1.0pc/^F "3";"2"};
    {\ar@/^1.0pc/^F "4";"3"};
  (-28,0)*{\cdots};
 \endxy.
\]

The Verma module $M(\lambda)$ is the unique infinite dimensional module of highest weight $\lambda$, and it is irreducible unless $c\in\bN$. 
To keep the notation simple we write $M(n)$ instead of $M(q^n)$ whenever $c=n\in\bZ$.    
In this case $\Hom_{U_q(\slt)}(M(n'),M(n))$ is zero unless $n' =n$ or $n'=-n-2$,
and there is a monomorphism $\phi\colon M(-n-2)\to M(n)$, uniquely determined up to scalar multiples. 
Moreover, the quotient $M(n)/M(-n-2)$ is isomorphic to the irreducible $U_q(\slt)$-module $V(n)$
of dimension $n+1$, and all finite-dimensional irreducibles can be obtained this way.  
Under this quotient, the canonical basis of $M(n)$ descends to a particular case of
Lusztig-Kashiwara canonical basis in finite-dimensional irreducible representations
of quantum groups introduced in~\cite{Lusztig-can} and independently in~\cite{Kashiwara-crystals}. 

\smallskip

The Verma module $M(\lambda)$ is \emph{universal} in the sense that any given Verma module with 
integral highest weight can be obtained from $M(\lambda)$. 
This means that for each $n\in\bZ$ there is an \emph{evaluation map} $\mathrm{ev}_n\colon M(\lambda)\to M(n)$ which is a surjection 
(see~\cite{RamT} and also~\cite{Kashiwara,Kamita} for details).

\smallskip

Throughout this paper we will take $M(\lambda q^{-1})$ as the universal Verma module  
and we will call $M(\lambda q^{-1+n})$ ($n\in\bZ$) the \emph{shifted Verma modules} (see~\cite{RamT} for details). 
In our conventions, evaluating $M(\lambda q^{-1})$ at $n$ means putting $\lambda = q^{n+1}$.
The evaluation map $\mathrm{ev}_n$ is then the composite of a shift with~$\mathrm{ev}_{-1}$.

\subsection{Bilinear form}\label{ssec:ushapo} 
The \emph{universal Shapovalov form} $\left(-,-\right)_\lambda$ is the bilinear form on $M(\lambda q^{-1})$ 
such that for any $m,m'\in M(\lambda q^{-1})$, $u\in U_q(\slt) $, and   $f\in \bQ\pp{q,\lambda}$ we have
 
\begin{itemize}
\item $\left(m_0,m_0\right)_\lambda = 1$, 
\item $\left(um,m'\right)_\lambda = \left(m,\rho(u)m'\right)_\lambda$,  
where $\rho$ is the $\bQ(q)$-linear antiautomorphism defined in Equation~\eqref{eq:antirho},
\item $f\left(m,m'\right)_\lambda=\left(fm,m'\right)_\lambda=\left(m,fm'\right)_\lambda$. 
\end{itemize}

The involution $\overline{\textcolor{white}{q}}$ does not extend to $\bQ\pp{q, \lambda}$ (for example $\sum_{k \ge 0} q^{k}$ would be sent to $\sum_{k \ge 0} q^{-k}$ which is not an element of $\bQ\pp{q,\lambda}$, see \S\ref{ssec:laurent} for more details about $\bQ\pp{q,\lambda}$). 
However, when restricting the ground field to $\bQ(q, \lambda)$ instead of $\bQ\pp{q,\lambda}$ 
(we write $M_{\bQ(q, \lambda)}(\lambda q^{-1})$ in this case) 
there is another form we can define. 
We refer to it as the \emph{twisted Shapovalov form}, and it is the sesquilinear form uniquely defined by 
\begin{itemize}
\item $\brak{m_0,m_0}_\lambda = 1$, 
\item $\brak{um,m'}_\lambda = \brak{m,\tau(u)m'}_\lambda$,  
where $\tau$ is the $q$-antilinear antiautomorphism defined in Equation~\eqref{eq:antitau},
\item $f\brak{m,m'}_\lambda=\brak{\bar{f}m,m'}_\lambda=\brak{m,fm'}_\lambda$,  
where $^-$ is the $\bQ$-linear involution of $\bQ(q,\lambda)$ 
which maps $q$ to $q^{-1}$ and $\lambda$ to $\lambda^{-1}$,
\end{itemize}
for any  $m,m'\in M_{\bQ(q, \lambda)}(\lambda q^{-1})$, $u\in U_q(\slt) $, and   $f\in \bQ(q,\lambda)$.

For example,   
$$\brak{F^nm_0,F^nm_0}_\lambda = \lambda^n q^{-n(1+n)}[n]![\lambda,-1][\lambda,-2]\dotsm [\lambda,-n], $$
the notation $[\lambda,m]$ being introduced in~\eqref{eq:slcommutator}.

\medskip

Evaluation of $M(\lambda q^{-1})$ at $n$ reduces to the well-known $\bQ(q)$-valued bilinear form 
(see~\cite{RamT} and~\cite{shapovalov} for the original 
definition in the non-quantum context as well as a proof of uniqueness).
The \emph{$q$-Shapovalov form} $(-,-)_n$ is the unique bilinear form on $M(n)$ such that 
for any $m,m'\in M(n)$, $u\in U_q(\slt)$, and   $f\in \bQ(q)$ we have 
\begin{itemize}
\item $(v_0,v_0)_n = 1$, 
\item $(um,m')_n = (m,\rho(u)m')_n$,  
where $\rho$ is the $q$-linear antiautomorphism defined in Equation~\eqref{eq:antirho},
\item $f(m,m')_n=(fm,m')_n=(m,fm')_n$.
\end{itemize}

For $n\geq 0$ the radical of $(-,-)_n$ is the maximal proper submodule $M(-n-2)$ of $M(n)$, and hence  
we have $V(n)=M(n)/\mathrm{Rad}(-,-)_n$
and the $q$-Shapovalov form descends to a bilinear form on $V(n)$. 

\medskip 

Using the Shapovalov form we define the \emph{dual canonical basis} 
$\{m^i\}_{i\in\bN_0}$ of $M(\lambda)$ by  
\[
(m_i', m^j)_\lambda = \delta_{i,j}.
\]
Define $[\lambda, j]!$ recursively by
\[ 
[\lambda, 0]!=1, \mspace{30mu}[\lambda, j]! = [\lambda, j-1]! [\lambda, j] .
\] 
Then 
\[  
m^k = \frac{[k]!}{[\lambda ,-k]! \lambda^kq^{-k(k+1)}} m_k,
\]
and the action of $F$, $E$ and $K$ on the dual canonical basis is 
\begin{equation}\label{eq:dcbaction} 
\begin{split}
K^{\pm 1} m^k &= (\lambda q^{-2k})^{\pm 1}m^k ,
\\[1ex]
Fm^k  &= \frac{\lambda q^{-k} - \lambda q^{k}}{q-q^{-1}}\lambda q^{-2k-1}m^{k+1} , 
\\[1ex]  
Em^{k} &= [k] \lambda^{-1} q^{2k-1} m^{k-1} .
\end{split}
\end{equation} 

The above reduces without any changes to the case of $M(n)$ for $n\notin\bN_0$. 
For $n\in\bN_0$ the procedure cannot be applied on $M(n).$ However it can be used in the 
finite dimensional quotient $V(n)$ yielding the usual dual canonical basis of finite-dimensional representations, up to a normalization 
(see for example the exposition in~\cite[\S1.2]{fks}).  

\smallskip

Restricting the ground field to the ring $\bQ\llbracket q \rrbracket [q^{-1},\lambda^{\pm1}]$ yields two $\dot U_\lambda$-modules, one given by the basis $\{m_k\}_{k\in \bN}$ and the other one by $\{m^k\}_{k\in \bN}$, which are non-isomorphic. We denote them respectively $M_A(\lambda)$ and $M_A^*(\lambda)$, where $A = \bQ\llbracket q \rrbracket [q^{-1},\lambda^{\pm1}]$.



%
%
\section{The geometry of the infinite Grassmannian}\label{sec:grass}

\subsection{Grassmannians and their $\Ext$ algebras}\label{ssec:hoch}

Let $G_k$ be the Grassmannian variety of $k$-planes in $\bC^\infty$.   
This space classifies $k$-dimensional complex vector bundles over a manifold $\cN$, 
in the sense that there is a tautological bundle over $G_k$, 
and every $k$-dimensional vector bundle over $\mathcal{N}$ is a pull-back of the tautological bundle by 
some map from $\cN$ to $G_k$. 
Since this pull-back is invariant under homotopy  we actually study homotopy classes of maps from $\cN$ to $G_k$. 
The cohomology ring of $G_k$ is generated by the Chern classes 
(see for example~\cite[Chapter 14]{ms} for details),  
\[ 
H(G_k) \cong \bQ[x_{1,k},\dotsc ,x_{k,k}, Y_{1,k}, \dotsc, Y_{i,k}, \dotsc] / I_{k,\infty} ,
\]  
where $I_{k, \infty}$ is the ideal generated the homogeneous components in $t$ satisfying the equation
\begin{equation}\label{eq:Ikinfty}
(1 +  x_{1,k}t + \dotsc + x_{k,k}t^k)(1 + Y_{1,k}t + \dotsc + Y_{i,k} t^i + \dotsc) = 1.
\end{equation}
This ring is $\bZ$-graded with 
$\deg_q(x_{i,k}) = \deg_q(Y_{i,k}) =2i$.  
Note that \eqref{eq:Ikinfty} yields recursively 
\begin{equation}\label{eq:recY}
Y_{i,k} = -\sum_{\ell = 1}^i x_{\ell,k} Y_{i-\ell,k},
\end{equation}
where  $Y_{i,k} = 0$ if $i<0$, $Y_{0,k} =1$ and $x_{j,k} = 0$ for $j>k$.
Since every $Y_{i,k}$ can be written as a combination of $x_{j,k}$, we have
\[
H(G_k) \cong \bQ[x_{1,k},\dotsc ,x_{k,k}].
\]

\medskip

Now let $G_{k,k+1}$ be the infinite partial flag variety 
\[
G_{k,k+1} = \{ (U_k,U_{k+1}) \vert \dim_{\bC}U_k = k, \dim_{\bC}U_{k+1} = k+1, \ 
0 \subset U_k \subset U_{k+1} \subset \bC^{\infty} \}. 
\]
As it turns out, the infinite Grassmannian $G_k$ is homotopy equivalent to the classifying 
space $BU(k)$ of the unitary group $U(k)$, 
and we have a fibration  
\[
B \to BU(k) \times BU(1) \to BU(k+1)
\]
induced by the inclusion $U(k) \times U(1) \to U(k+1)$. 
The fibre has the homotopy type of the quotient  
$U(k+1) / ( U(k) \times U(1) )$
and corresponds to $G_{k,k+1}$ in the sense that 
specifying $U_k\subset U_{k+1}$  in $\bC^\infty $ corresponds to specifying $U_k$ in $\bC^\infty$ 
and a 1-dimensional Grassmannian in $U_k$. 
As a consequence, we get that the cohomology of $B$, and therefore of $G_{k,k+1}$, 
is generated by the Chern classes 
\[  
H(G_{k,k+1}) \cong \bQ[w_{1,k},\dotsc ,w_{k,k}, \xi_{k+1}, Z_{1,k+1}, \dotsc , Z_{i,k+1}, \dotsc]/I_{k,k+1,\infty},
\]  
with $I_{k,k+1,\infty}$ given by the equation
\[
(1 +  w_{1,k}t + \dotsc + w_{k,k}t^k)(1 + \xi_{k+1} t)(1 + Z_{1,k+1}t + \dotsc + Z_{i,k+1} t^i + \dotsc) = 1.
\]
Without surprise, $H(G_{k,k+1})$ has a natural structure of a $\bZ$-graded ring  
with
\[
\deg_q(w_{i,k}) = \deg_q(Z_{i,k+1}) =2i,\mspace{15mu}\deg_q(\xi_{k+1})=2. 
\] 
Again, we can write every $Z_{i,k+1}$ as a combination of $w_{j,k}$ and $\xi_{k+1}$ to get
\[
H(G_{k,k+1}) \cong \bQ[w_{1,k},\dotsc ,w_{k,k}, \xi_{k+1}].
\]

\medskip

The ring $H(G_k)$ is a graded positive noetherian ring  
which has a unique simple module, up to isomorphism and grading shift,  
$H(G_k)/H(G_k)_+ \cong \bQ$, where 
$H(G_k)_+$ is the submodule of $H(G_k)$ generated by the elements of nonzero degree. 
Let $\Ext_{H(G_k)}(\bQ,\bQ)$ be the algebra of self-extensions of $\bQ$, which 
is an exterior algebra in $k$ variables,
\[ 
\Ext_{H(G_k)}(\bQ,\bQ) \cong \bV^\bullet(s_{1,k},\dotsc ,s_{k,k}).
\]   
It is a $\bZ\times\bZ$-graded ring with $\deg_q(s_{i,k})=-2i$ and $\deg_\lambda(s_{i,k})=2$. 
The first grading  is induced by the grading in $H(G_k)$ and we call it quantum, while the second grading is cohomological. 
Sometimes we write $\deg_{q,\lambda}(x)$ for the ordered pair $(\deg_q(x),\deg_\lambda(x))$.  

\smallskip

In another way of looking at this we note that $H(G_k)$ is a Koszul algebra and
therefore quadratic.
Indeed let $V_k$ be the $\bQ$-vector space with basis $\brak{x_{1,k}, \dots, x_{k,k}}$
and $R = \{x_{i,k}x_{j,k} - x_{j,k}x_{i,k} | 1 \le i,j \le k\}$,
then $H(G_k) \cong T(V_k)/(R)$ and the Koszul dual of $H(G_k)$ 
coincides with the quadratic dual
$H(G_k)^! = T(V_k^*)/(R^\perp)$, with $R^\perp = \{f \in V_k^* \otimes V_k^* | f(R) = 0\}$
(see~\cite[\S2.10]{bgs}).
An easy exercise shows that $H(G_k)^! \cong \bV^\bullet(s_{1,k},\dotsc ,s_{k,k})$, 
where we identify $s_{i,k}$ with $x_{i,k}^* : V_k \rightarrow \bQ$. In conclusion we have an isomorphism $H(G_k) \otimes H(G_k)^! \cong H(G_k) \otimes \left(H(G_k)^!\right)^{op} \cong H(G_k) \otimes \Ext_{H(G_k)} (\bQ, \bQ)$.

\begin{defn}\label{def:omegak}
For each $k \in \bN$ we form the bigraded rings
\begin{align*} 
\Omega_k &= H(G_k)\otimes \Ext_{H(G_k)}(\bQ, \bQ) ,
\intertext{and}
\Omega_{k,k+1} &= H(G_{k,k+1})\otimes \Ext_{H(G_{k+1})}(\bQ, \bQ) . 
\end{align*}
\end{defn}
Note we do not use extensions of $H(G_{k,k+1})$-modules    
and also $\Omega_k$ is isomorphic to the Ho\-chs\-child cohomology of $H(G_k)$. 
In order to fix some notation and avoid any possibility of confusion in future computations 
we fix presentations of these rings as 
\[ 
\Omega_k = \bQ[\und{x}_k,\und{s}_k],\mspace{40mu}\text{and}\mspace{40mu}\Omega_{k,k+1} = \bQ[\und{w}_k,\xi_{k+1},\und{\sigma}_{k+1}] ,\]  
where we write $\und{t}_m$ for an array $(t_{1,m},\dotsc ,t_{m,m})$ of $m$ variables
and where it is abusively implied that the variables $s_i$ and $\sigma_i$ are anticommutative.

\medskip

Rings $\Omega_k$ and $\Omega_{k,k+1}$ are in fact (supercommutative rings) superrings with an inherent $\bZ_2$-grading,
called \emph{parity}, given by 
\[
p(x_{i,k}) = 0,\quad p(s_{i,k})= 1,
\]   
for $x_{i,k}$, $s_{i,k}\in\Omega_k$, and 
\[
p(w_{i,k}) = p(\xi_{k+1})=0,\quad p(\sigma_{i,{k+1}})= 1,
\]   
for $w_i$, $\xi_{k+1}$ and $\sigma_i\in\Omega_{k,k+1}$.

\subsection{Superbimodules}
Let $R$ be a superring.
A left (resp. right) $R$-supermodule is a $\bZ_2$-graded left (resp. right) $R$-module.
A left supermodule map $f : M \rightarrow N$ is a homogeneous group homomorphism that supercommutes with the action of $R$, 
\[
f(r \bullet m) = (-1)^{p(f)p(r)} r \bullet f(m) , 
\]
for all $r \in R$ and $m \in M$. A right supermodule map is a homogeneous right module homomorphism.
An $(R,R')$-superbimodule is both a left $R$-supermodule and a right $R'$-supermodule, with compatible actions. A superbimodule map is both a left supermodule map and a right supermodule map. 

Then,  if $R$ has a supercommutative ring structure and if we view it as an $(R,R)$-superbimodule, multiplying at the
left by an element of $R$ gives rise to a superbimodule endomorphism. 

Let $M$ and $N$ be respectively an $(R',R)$ and an $(R,R'')$-superbimodules. One form their tensor product over $R$ in the usual way for bimodules, giving a superbimodule. Given two superbimodule maps $f : M \rightarrow M'$ and $g : N \rightarrow N'$, we can form the tensor product $f \otimes g : M \otimes N \rightarrow M' \otimes N'$, which is defined by
\[
(f \otimes g)(b \otimes m) = (-1)^{p(g)p(b)} f(b) \otimes g(m),
\]
and gives a superbimodule map.

Now define the \emph{parity shift} of a supermodule $M$, denoted $\Pi M = \{\pi(m) | m \in M\}$, where $\pi(m)$ is the element $m$ with the parity inversed, and if $M$ is a left supermodule (or superbimodule) with left action given by  
\[
r \bullet \pi(m) = (-1)^{p(r)} \pi( r \bullet m),
\]
for $r \in R$ and $m \in M$. The action  on the right remains the same.

In this context, the map $R \rightarrow \Pi R$ defined by $r \mapsto \pi(ar)$ for some odd element $a \in R$ is a
$\bZ/2\bZ$-grading preserving homomorphism of $(R,R)$-superbimodules.

Let $\pi : M \rightarrow \Pi M$ denote the change of parity map $x \mapsto \pi(x)$.
It is a supermodule map with parity $1$ and satisfies $\pi^2 = \id$.
The map $\pi \otimes \pi :\Pi M \otimes N \rightarrow M \otimes \Pi N$ is $\bZ/2\bZ$-grading preserving and such that $(\pi \otimes \pi)^2 = - \id$, thus
\[
\Pi ( M \otimes N) \cong \Pi M \otimes N \cong M \otimes \Pi N 
\]
are isomorphisms of supermodules. All the above is presented with a
more categorical flavour in~\cite{EllisLauda,KangKashOh} (see also~\cite{KangKashTsuch}), showing that the supermodules and superbimodules give supercategories. 
Of course, all the above extends to the case when $R$ has additional gradings, making it a multigraded superring.

\subsection{Graded dimensions}\label{ssec:gdims}

Recall that a $\bZ\times\bZ$-graded supermodule 
\[
M=\bigoplus\limits_{i,j,k\,\in,\bZ\times\bZ\times\bZ_2}M_{i,j,k}\] 
is \emph{locally of finite rank} 
if each $M_{i,j,k}$ has finite rank.
The same notion applies for bigraded vector spaces.
We denote $M\brak{r,s}$ the supermodule with the $q$-grading shifted up by $r$ 
and the $\lambda$-grading shifted up by $s$.

In the context of locally finite rank supermodules and vector spaces 
it makes sense to talk about graded ranks and graded dimensions. 
In the cases under consideration the graded dimension of $M$ is the Poincar\'e series
\begin{equation}\label{eq:bgdim}
\gdim(M) = \sum_{i,j,k\,\in\,\bZ\times\bZ\times\bZ_2}\pi^k\lambda^jq^i\dim(M_{i,j,k}) \in \bZ_\pi \llbracket q^{\pm 1}, \lambda ^{\pm1} \rrbracket, 
\end{equation} 
where $\bZ_\pi = \bZ[\pi]/(\pi^2-1)$. For example 
\[
\gdim(\Omega_k)= \prod_{s=1}^k (1+\pi\lambda^{2}q^{-2s})(1+q^{2s} + q^{4s} + \dots) . 
\]
In this case, we can view $\gdim(\Omega_k)$ as living inside of $\bZ_\pi \pp{q,\lambda}$ (see \S\ref{ssec:laurent} below for more details about $\bZ\pp{q,\lambda}$) and it gives
\[
\gdim(\Omega_k)=  \prod_{s=1}^k\frac{1+\pi\lambda^{2}q^{-2s}}{1-q^{2s}}.
\]
When we refer to the (graded) superdimension, we will mean we specialize $\pi = -1$ in the graded dimension, giving a series in $\bZ \llbracket q^{\pm 1}, \lambda ^{\pm 1} \rrbracket $. We denote it $\sdim(M)$.

Sometimes it is useful to consider a direct sum of objects (e.g. supermodules or superbimodules) 
where the $q$-degree of each summand has been shifted by a different amount. 
In this case we use the notion of shifting an object by a Laurent polynomial: 
given $f=\sum f_jq^j\in\bN[q,q^{-1}]$ we write 
$\oplus_f M$ or $M^{\oplus f}$ for the direct sum over $j\in\bZ$  
of $f_j$ copies of $M\brak{j,0}$.
In a further notational simplication will write $M\brak{j}$ for $M\brak{j,0}$
whenever convenient.

\subsection{The superbimodules $\Omega_{k,k+1}$}\label{ssec:modbim}

The forgetful maps 
\[
\xymatrix
{
 & G_{k,k+1}\ar[dr]^{p_k}\ar[dl]_{p_{k+1}} & 
\\
G_{k+1} && G_{k}
}
\]
induce maps in the cohomology
\begin{equation}\label{eq:correspgrass}
\raisebox{8mm}{
\xymatrix
{
 & H(G_{k,k+1}) & 
\\
H(G_{k+1})\ar[ur]^{\psi_{k+1}}  && H(G_{k})\ar[ul]_{\phi_k}
}}
\end{equation}
given by
\[
\phi_k\colon H(G_k)\to H(G_{k,k+1}),\qquad x_{i,k}\mapsto w_{i,k}, \quad Y_{i,k} \mapsto Z_{i,k+1} + \xi_{k+1} Z_{i-1,k+1} ,
\]
and 
\[
\psi_{k+1}\colon H(G_{k+1})\to H(G_{k,k+1}),\qquad  
x_{i,k+1}\mapsto w_{i,k} + \xi_{k+1}w_{i-1,k} , \quad Y_{i,k+1} \mapsto Z_{i,k+1} , 
\] 
with the understanding that $w_{0,k} = Z_{0,k+1}=1$ and $w_{k+1,k}=0$. 

\medskip

These inclusions make $H(G_{k,k+1})$ an $(H(G_{k+1}),H(G_k))$-bimodule.   
As a right $H(G_k)$-module, $H(G_{k,k+1})$ is a free, graded module,
isomorphic to $H(G_k)\otimes_\bQ \bQ[\xi_{k+1}]$.

\medskip 

To get a correspondence in terms 
of $\Omega_k$, $\Omega_{k+1}$ and $\Omega_{k,k+1}$ 
we use the maps $\phi_k$ and $\psi_{k+1}$ above to construct maps $\phi_k^*$ and $\psi_{k+1}^*$ 
between the various rings involved, as in~\eqref{eq:correspgrass},  
\[ 
\xymatrix
{
 & \Omega_{k_,k+1} & 
\\
\Omega_{k+1}\ar[ur]^{\psi_{k+1}^*} && \Omega_{k}\ar[ul]_{\phi_{k}^*}.
} 
\]

Let $V_{k,k+1}$ be the $\bQ(\xi_{k+1})$-vector space with basis $\brak{x_{1, k}, \dots, x_{k,k}}$. The maps $\phi_k$ and $\psi_{k+1}$ induce $\bQ$-linear injective maps $V_k \rightarrow V_{k,k+1}$ and $V_{k+1} \rightarrow V_{k,k+1}$. 
Now recall that we can view $s_{i,k}$ as the $\bQ$-linear map $x_{i,k}^* : V_k \rightarrow \bQ$ and $s_{i,k+1}$ as $x_{i,k+1}^* : V_{k+1} \rightarrow \bQ$. They can both be extended to $\bQ(\xi_{k+1})$-linear maps $\widetilde x_{i,k}^*, \widetilde x_{i,k+1}^* : V_{k,k+1} \rightarrow \bQ(\xi_{k+1})$.
We have $\psi_{k+1}(x_{i,k+1}) = \phi_k(x_{i,k}) + \xi_{k+1} \phi_k(x_{i-1, k})$, which gives $\phi(x_{i,k}) = \sum_{\ell = 0}^i  (-1) ^\ell  \xi_{k+1}^\ell  \psi_{k+1}(x_{i-\ell, k+1})$, and thus
\[
\widetilde x_{i,k}^* =\widetilde x_{i,k+1}^* +  \xi_{k+1} \widetilde x_{i+1, k+1}^*.
\]
Translated to the language of $s_i$'s we get that $s_{i,k}$ should be equivalent to $s_{i,k+1} + \xi_{k+1} s_{i+1, k+1}$. 
Hence we define the map $\phi_k^*\colon \Omega_k\to \Omega_{k,k+1}$ as 
\begin{equation}\label{eq:xiraction}
\phi_k^*\colon \Omega_k\to \Omega_{k,k+1},\mspace{20mu} 
\begin{cases}
x_{i,k} \mapsto w_{i,k},  \\
s_{i,k} \mapsto \sigma_{i,k+1} + \xi_{k+1}\sigma_{i+1,k+1} ,
\end{cases}
\end{equation}
and $\psi_{k+1}^*\colon \Omega_{k+1}\to \Omega_{k,k+1}$ as
\begin{equation}\label{eq:xilaction}
\psi_{k+1}^*\colon \Omega_{k+1}\to \Omega_{k,k+1},\mspace{20mu}
\begin{cases}
x_{i,k+1} \mapsto w_{i,k} + \xi_{k+1} w_{i-1,k},
\\[1.5ex]
s_{i,k+1} \mapsto \sigma_{i,k+1}. 
\end{cases}
\end{equation}
with $w_{0,k+1} = 1$ and $w_{k+1,k} = 0$.  Since every $\sigma_{i,k+1}$ and $w_{i,k}$ can be obtained from $s_{i,k+1}$ and $x_{i,k}$, we write $\Omega_{k,k+1}$ in this basis as 
\[
\Omega_{k,k+1} \cong \bQ[\und{x}_k,\xi_{k+1},\und{s}_{k+1}].
\]
We will also write $Y_{i,k+1}$ for $Z_{i,k+1} = \psi^*_{k+1}(Y_{i,k+1})$ in $\Omega_{k,k+1}$.

\begin{rem}
Note the $s_i$'s and $\sigma_i$'s behave like $Y_{-i}$'s and $Z_{-i}$'s with a (supposed) negative index $-i$. This will be useful to recover the finite case, as we will see in \S\ref{ssec:dgbimod}.
\end{rem}

As expected, maps $\phi_k^*$ and $\psi_{k+1}^*$ give $\Omega_{k,k+1}$ the structure of 
an $(\Omega_{k},\Omega_{k+1})$-superbimodule. 
Since these rings are supercommutative we can also think of  $\Omega_{k,k+1}$ as an 
$(\Omega_{k+1},\Omega_{k})$-superbimodule which we denote by $\Omega_{k+1,k}$. 
When dealing with tensor products of superbimodules we simplify the notation and 
write $\otimes_k$ for $\otimes_{\Omega_k}$ and $\otimes$ for $\otimes_{\bQ}$.

We use the notation $\mathrm{smod}\text{-}\Omega_k$ and $\Omega_k\smod$ for right and left
$\Omega_k$-supermodules respectively. 
As a right $\Omega_k$-supermodule $\Omega_{k+1,k}\cong \bQ[\xi_{k+1}, s_{k+1}]\otimes\Omega_k$ 
is a free graded polynomial supermodule, which is of graded dimension
\[
\gdim_{\mathrm{smod}\text{-}\Omega_k}(\Omega_{k+1,k}) = \frac{1+\pi\lambda^{2}q^{-2k-2}}{1-q^{2}}. 
\] 
As a left $\Omega_{k+1}$-supermodule $\Omega_{k+1,k}\cong \oplus_{\{k+1\}}\Omega_{k+1}$ 
is a free graded supermodule, with $\{k+1\} =  1 + q^2 + \dots + q^{2k}$, using the convention from \S\ref{ssec:gdims}. Thus $\Omega_{k+1,k}$  is of graded dimension
\[
\gdim_{\Omega_{k+1}\smod}(\Omega_{k+1,k}) = \{k+1\}  = 1 + q^2 + \dots + q^{2k}. 
\] 

\medskip

 Due to the specific nature of our superrings and (super)categories, there are several
 notions and results that can be borrowed unchanged from the nonsuper case, as the notion of sweetness for bimodules
 below.  
Recall that a superbimodule is sweet if it is projective as a left supermodule and as a right supermodule.
Tensoring with a superbimodule yields an exact functor that sends projectives to projectives if and only if the
superbimodule is sweet. 
The superbimodule $\Omega_{k,k+1}$ is \emph{sweet}.

\medskip

Let $G_{k,k+1,\dotsm , k+m}$ be the 1-step flag variety  
$$  
\{ (U_k,U_{k+1},\dotsm ,U_{k+m}) \vert \dim_{\bC}U_{k+i} = k+i, \ 
0 \subset U_k \subset U_{k+1} \subset \dotsm \subset U_{k+m} \subset \bC^{\infty} \}.$$ 
As in the cases of $G_k$ and $G_{k,k+1}$ the cohomology of $G_{k,\dotsm ,k+m}$ has a description in terms of Chern classes,
\[
H(G_{k,\dotsm ,k+m})\cong \bQ[\und{w}_k,\und{\xi}_m] , 
\]
with $\deg_q(w_{i,k})=2i$ and $\deg_q(\xi_{j,m})=2$. 
Paralleling the case of $\Omega_{k,k+1}$, we define the bigraded superring 
\[
\Omega_{k,\dotsc ,k+m}=H(G_{k,\dotsm ,k+m})\otimes \Ext_{H(G_{k+m})}( \bQ, \bQ ) 
\cong \bQ[\und{w}_{k},\und{\xi}_m, \und{\sigma}_{k+m}] , 
\] 
with $\deg_{\lambda ,q}(\sigma_{j,k+m}) = ( -2j,2 )$.  
In this case we also have maps 
\[
\phi_{k,m}^*\colon \Omega_k\to \Omega_{k,\dotsc ,k+m},\mspace{20mu} 
\begin{cases}
x_{i,k} \mapsto w_{i,k},  \\
s_{i,k} \mapsto \sum\limits_{j=0}^m \sigma_{i+j,k+m} e_{j}(\und{\xi}_m),
\end{cases}
\]
and
\[  
\psi_{k+m,m}^*\colon \Omega_{k+m}\to \Omega_{k,\dotsc ,,k+m},\mspace{10mu}
\begin{cases}
x_{i,k+m} \mapsto \sum\limits_{j=0}^i w_{j,k}e_{i-j}(\und{\xi}_m) , 
\\[1.5ex]
s_{i,k+m} \mapsto \sigma_{i,k+m},  
\end{cases}
\] 
where $e_j(\und{\xi}_m)$ is the $j$-th elementary symmetric polynomial in the variables $\xi_1, \dots, \xi_m$.

\begin{lem}\label{lem:bimtens}
The superring 
$\Omega_{k,k+1,\dotsc ,k+m}$ is a bigraded $(\Omega_{k}, \Omega_{k+m})$-superbimodule, which is isomorphic to 
$\Omega_{k,k+1}\otimes_{k+1}\Omega_{k+1,k+2}\otimes_{k+2} \dotsm \otimes_{k+n-1}\Omega_{k+n-1,k+n}$.   
\end{lem}

We can form more general superbimodules. 
For a sequence $k_1,\dotsc ,k_m$ of nonnegative integers we define the 
$(\Omega_{k_1},\Omega_{k_m})$-superbimodule
\begin{equation*}
\Omg_{k_1,\dotsc ,k_m} = \Omega_{k_1,k_2}\otimes_{k_2}\Omega_{k_2,k_3}\otimes_{k_3}\dotsm\otimes_{k_{m-1}}\Omega_{k_{m-1},k_m} .
\end{equation*}

This superbimodule has an interpretation in terms of the geometry of partial flag varieties. 
Consider the variety $G_{k_1,\dotsc ,k_m}$ consisting of sequences 
$(U_{k_1},\dotsc ,U_{k_m})$ of linear subspaces of $\bC^\infty$ such that 
$\dim(U_{k_i})=k_i$ and $U_{k_i}\subset U_{k_{i+1}}$ if $k_i\leq k_{i+1}$ 
and $U_{k_i}\supset U_{k_{i+1}}$ if $k_i\geq k_{i+1}$.
As before, the forgetful maps  
$$\xymatrix
{
 & G_{k_1,\dotsc,k_m}\ar[dl]_{p_{k_1}}\ar[dr]^{p_{k_m}} & 
\\
G_{k_1} && G_{k_m}
}
$$
induce maps of the respective cohomology rings. 
Proceeding as above one can construct maps 
$$\xymatrix
{
 & \Omega_{k_1,\dotsc,k_m} & 
\\
\Omega_{k_1} \ar[ur]^{\phi_{k_1}^*} && \Omega_{k_m}.\ar[ul]_{\psi_{k_m}^*}
}
$$
As expected, the  $(\Omega_{k_1},\Omega_{k_m})$-superbimodules 
$\Omg_{k_1,\dotsc,k_m}$ and $\Omega_{k_1,\dotsc,k_m}$ are isomorphic. 

\medskip

In particular the isomorphism from $\Omega_{0,k}$ to $\Omg_{0,1,\dotsc,k}$ is explicitly given by 
\begin{align*} 
  x_{i,k} \mapsto e_i(\und\xi_k),  \mspace{50mu}
  Y_{i,k} \mapsto (-1)^i h_i(\und\xi_k),  \mspace{50mu}
  s_{i,k} \mapsto \sigma_{i,k},
\end{align*}  
with $h_i(\und \xi_k)$ being the $i$-th complete homogeneous symmetric polynomial in variables $\xi_1, \dots, \xi_k$.

%
%
\section{The 2-category $\extflag$}\label{sec:higher}

\subsection{The $2$-category $\extflag$}\label{subsec:commutator}

Let $\sbim$ denote the (super) 2-category of superbimodules, with objects given by superrings,
1-morphisms by superbimodules and 2-morphisms by degree preserving superbimodule maps.
The superbimodules introduced in the previous section can be used to define a
locally full sub 2-category\footnote{This means it induces a full embedding between the corresponding
$\Hom$-categories}
of $\sbim$, which we now describe.

\begin{defn}
The 2-category $\extflag$ is defined as:
\begin{itemize}
\item Objects: the bigraded superrings $\Omega_k$ for each $k \in \bN$. 
\item 1-morphisms: generated by the graded $(\Omega_k,\Omega_k)$-superbimodules $\Omega_k$ and    $\Omega_{k}^\xi=\Omega_k[\xi]$,  
the graded $(\Omega_k,\Omega_{k+1})$-superbimodule $\Omega_{k,k+1}$ 
and the graded $(\Omega_{k+1},\Omega_{k})$-superbimodule $\Omega_{k+1,k}$, together with their bidegree and parity shifts.
The superbimodules  $\Omega_k$ are the identity 1-morphisms. 
A generic 1-morphism from $\Omega_{k_1}$ to $\Omega_{k_m}$ is a direct sum of 
bigraded superbimodules of the form 
\[ 
\Pi^{\pi}\Omega_{k_m,k_{m-1}}\otimes_{k_{m-1}}\Omega_{k_{m-1},k_{m-2}}\otimes_{k_{m-2}}
\ \dotsm\
\otimes_{k_2}\Omega_{k_{2},k_1} \otimes_{k_1} \Omega_{k_1}[\xi_1, \dots, \xi_\ell] \brak{s,t} 
\] 
with $\vert k_{i+1}-k_i\vert =1$ for all $1\leq i\leq m$ and $\pi \in \{0,1\}$. 
\item 2-morphisms: degree-preserving superbimodule maps.
\end{itemize}
\end{defn}

As in other instances of categorical $\slt$-actions, 
the $(\Omega_k,\Omega_k)$-superbimodules $\Omega_{k,k+1}\otimes_{k+1}\Omega_{k+1,k}$ 
and $\Omega_{k,k-1}\otimes_{k-1}\Omega_{k-1,k}$ 
are related through a categorical version of the commutator relation (\ref{eq:slcommutator}).
To make our formulas simpler when dealing with tensor products of superbimodules we
write $\Omega_{k(k+1)k}$ instead of $\Omega_{k,k+1} \otimes_{k+1} \Omega_{k+1,k}$
and $\Omega_{k(k-1)k}$ instead of $\Omega_{k,k-1} \otimes_{k-1} \Omega_{k-1,k}$.

\smallskip

To be able to state and prove this categorical version of the commutator in $\extflag$, we need some preparation.

\begin{lem}\label{lem:kk+1rel}
In $\Omega_{k,k+1}$, the following identities hold for all $i,\ell \ge 0$:
\begin{align}
x_{\ell,k} &= \sum_{p=0}^\ell (-1)^p \psi^*(x_{\ell-p,k+1}) \xi_{k+1}^p, \label{eq:ytosum} \\
Y_{\ell,(k+1)} &= \sum_{p=0}^\ell (-1)^p \phi^*(Y_{\ell-p,k})\xi_{k+1}^p,  \label{eq:Ytosum} \\
\xi_{k+1}^i &=  (-1)^i \sum_{\ell = 0}^i x_{\ell,k} Y_{i-\ell,k+1}. 
 \label{eq:xitosum}
\end{align}
\end{lem}

\begin{proof}
The three relations are obtained by induction on (\ref{eq:xiraction}) and (\ref{eq:xilaction}).
\end{proof}

\begin{lem}\label{lem:decomp}
Each element of $\Omega_{k(k-1)k}$ decomposes uniquely as a sum
$$(f_0 \otimes_{k-1} g_0) + (f_1\xi_{k}\otimes_{k-1} g_1) + \dots + (f_{k-1}\xi_{k}^{k-1} \otimes_{k-1} g_{k-1})$$
with $f_i, g_i \in \psi^*_k(\Omega_{k})$.
\end{lem}

\begin{proof}
From $(\ref{eq:xitosum})$ we see that every element of $\Omega_{k-1,k}$ decomposes uniquely as a sum
$$\alpha_0 x_{0,k-1} + \alpha_1 x_{1,k-1} + \dotsc + \alpha_{k-1} x_{k-1,k-1},$$
with $\alpha_i \in \psi^*_{k}(\Omega_k)$.
Then, sliding every $x_{i,k-1}$ over the tensor product we get that every element of $\Omega_{k(k-1)k}$ can be written as
\[
(h_0 \otimes_{k-1} \alpha_0) + (h_1 \otimes_{k-1} \alpha_1) + \dotsc + (h_{k-1} \otimes_{k-1} \alpha_{k-1}),
\]
with $h_i \in \Omega_{k,k-1}$.
Moreover,  by \eqref{eq:ytosum}, every element of $\Omega_{k,k-1}$ can be decomposed as a sum
\begin{equation}\label{eq:bimdechi}
  \beta_0 + \beta_1 \xi_{k} + \dots + \beta_{k-1}\xi_{k}^{k-1} , 
\end{equation}   
with $\beta_i \in \psi^*_k(\Omega_k)$.
Using~ \eqref{eq:bimdechi}  to decompose every $h_i$ 
we get a decomposition as in the statement. 
\end{proof}

\begin{prop} \label{prop:xislides}
In $\Omega_{k(k+1)k}$, the following identity holds:
\begin{equation*}
\sum_{\ell=0}^k (-1)^{\ell} x_{l,k} \otimes_{k+1} \xi_{k+1}^{k-\ell} = \sum_{\ell=0}^k (-1)^{\ell}  \xi_{k+1}^{k-\ell}  \otimes_{k+1}  x_{\ell,k}.
\end{equation*}
Moreover, the $\xi_{k+1}$ slides over this sum and therefore over the tensor product:
\begin{equation*}
\xi_{k+1}\sum_{\ell=0}^k (-1)^{\ell} x_{\ell,k} \otimes_{k+1} \xi_{k+1}^{k-\ell} = \sum_{\ell=0}^k (-1)^{\ell}  x_{\ell,k}  \otimes_{k+1}  \xi_{k+1}^{k-\ell+1}.
\end{equation*}
\end{prop}

\begin{proof}
  The same computations as in~\cite[\S3.2]{L2} can be used here since the polynomial
  side of $\Omega_{k(k+1)k}$ is the cohomology of the 1-step flag manifold as in the reference. 
\end{proof}

\begin{defn}\label{def:unit}
  We construct injective superbimodule morphisms of degrees $(2k, 0)$ by setting
\begin{align*}
\iota &: \Omega_{k}^\xi \hookrightarrow \Omega_{k(k+1)k},& \xi^i &\mapsto \xi_{k+1}^i\sum_{\ell=0}^k (-1)^\ell x_{\ell,k} \otimes_{k+1} \xi_{k+1}^{k-\ell}, \\
\eta &: \Omega_{k} \hookrightarrow \Omega_{k(k+1)k}, &  1 &\mapsto \sum_{\ell=0}^k (-1)^\ell x_{\ell,k} \otimes_{k+1} \xi_{k+1}^{k-\ell},
\end{align*}
and extending by the $(\Omega_k,\Omega_k)$-superbimodule structure~\eqref{eq:xiraction}. 
\end{defn}

Note these maps are superbimodule morphisms since the variable $\xi_{k+1}$ slides over the tensor product thanks to Proposition \ref{prop:xislides}, and thus multiplying at the left or at the right gives the same result.
Injectivity is a straightforward consequence of the fact that all our superbimodules are free as $\bQ$-modules.

\begin{prop}\label{prop:iotainv}
The left inverse of $\iota$ is given by
\begin{align*}
\pi &: \Omega_{k(k+1)k} \twoheadrightarrow \Omega_k^\xi,& \begin{cases}
\xi_{k+1}^{i} \otimes_{k+1} \xi_{k+1}^{j} &\mapsto (-1)^{i+j-k}Y^\xi_{i+j-k,k}, \\
\xi_{k+1}^{i} \otimes_{k+1} \xi_{k+1}^{j}s_{k+1,k+1} &\mapsto 0,
\end{cases}
\end{align*}
with $Y^\xi_{m,k} = 0$ for $m<0$, $Y^\xi_{0,k} =1$, and $Y^\xi_{i,k}$ is defined recursively by 
$Y^\xi_{i,k} = (-\xi)^i -\sum_{\ell = 1}^i  x_{\ell,k} Y^\xi_{i-\ell,k}$.
\end{prop}

\begin{proof}
We observe that for all $i \ge 0$ we have
\begin{align*}
(\pi \circ \iota) (\xi^i) &= \pi \left(\xi_{k+1}^i.\sum_{\ell=0}^k (-1)^\ell \xi_{k+1}^{k-\ell} \otimes_{k+1} x_{\ell,k} \right) \\
&=  \sum_{\ell=0}^k (-1)^i  x_{\ell,k} Y^\xi_{i-\ell,k} = (-1)^i Y^\xi_{i,k} +  (-1)^i \sum_{\ell=1}^k x_{\ell,k} Y^\xi_{i-\ell,k} \\
&= \xi^i - (-1)^i \sum_{\ell=1}^i  x_{\ell,k} Y^\xi_{i-\ell,k} + (-1)^i \sum_{\ell=1}^k x_{\ell,k}Y^\xi_{i-\ell,k} = \xi^i,
\end{align*}
with the last equality coming from the fact that $Y_{i-\ell,k} = 0$ for $\ell > i$ and $x_{\ell,k} = 0$ for $\ell > k$.
\end{proof}

\begin{rem}
Note that $Y_{i,k}^\xi$ has the same expression in $x_{r,k}$ as $Z_{i,k+1}$ in $w_{r,k}$ when we identify $\xi_{k+1}$ with $\xi$. Indeed we have 
$$
Z_{i,k+1} = -\sum_{\ell=1}^i \left(w_{\ell,k} Z_{i-\ell,k+1} + \xi_{k+1}w_{\ell-1,k} Z_{i-\ell,k+1} \right) = (-\xi_{k+1})^i - \sum_{\ell =1}^i w_{\ell,k} Z_{i-\ell,k+1}.
$$
\end{rem}

\begin{defn}\label{def:counit}
We also define a surjective morphism of degree $(-2k +2 , 0)$ by
\[
\epsilon : \Omega_{k(k-1)k} \twoheadrightarrow \Omega_k, \qquad \xi_{k}^{i} \otimes_{k-1} \xi_{k}^{j} \mapsto (-1)^{i+j-k+1} Y_{(i+j-k+1),k}.
\]
\end{defn}

\begin{rem}\label{rem:unitcounit}
We see that
\begin{align*}
  (\epsilon\otimes_{k+1} \id_{\Omega_{k+1,k}}) \circ (\id_{\Omega_{k+1,k}} \otimes_{k} \eta) &= \id_{\Omega_{k+1,k}} ,
  \intertext{and}
  (\id_{\Omega_{k+1,k}} \otimes_k \epsilon) \circ (\eta \otimes_{k+1} \id_{\Omega_{k+1,k}}) &= \id_{\Omega_{k+1,k}},
\end{align*}
by a computation similar to the one in~\cite[Lemma~4.5]{L2}. 
\end{rem}

\begin{defn}
We define a surjective superbimodule morphism of degree $(2k+2, -2)$ by
\begin{align*}
\mu &:  \Omega_{k(k+1)k} \twoheadrightarrow \Pi\Omega_{k}^\xi, \quad \begin{cases}
\xi_{k+1}^i \otimes_{k+1} \xi_{k+1}^j &\mapsto 0, \\
\xi_{k+1}^i \otimes_{k+1} \xi_{k+1}^j s_{k+1,k+1}  &\mapsto (-1)^{i+j} Y^\xi_{i+j,k},
\end{cases}
\end{align*}
and extending to $\Omega_{k(k+1)k}$ using the superbimodule structure~\eqref{eq:xiraction}. 
\end{defn}

We now define maps which allow connecting our construction to the nilHecke algebra later on.  

\begin{defn}
We define the \emph{nilHecke maps} by
\begin{align*}
X^- : \Omega_{k,k+1,k+2}&\rightarrow \Omega_{k,k+1,k+2}, \\
\xi_{k+1}^{i} \otimes_{k+1} \xi_{k+2}^{j} &\mapsto \sum_{\ell=0}^{i-1} \xi_{k+1}^{i+j-1-\ell} \otimes_{k+1} \xi_{k+2}^{\ell}
 -  \sum_{\ell=0}^{j-1} \xi_{k+1}^{i+j-1-\ell} \otimes_{k+1} \xi_{k+2}^{\ell}, \\
X^+ : \Omega_{k+2,k+1,k} &\rightarrow \Omega_{k+2,k+1,k},  \\
\xi_{k+2}^{i} \otimes_{k+1} \xi_{k+1}^{j} &\mapsto \sum_{\ell=0}^{j-1} \xi_{k+2}^{i+j-1-\ell} \otimes_{k+1} \xi_{k+1}^{\ell}
 -  \sum_{\ell=0}^{i-1} \xi_{k+2}^{i+j-1-\ell} \otimes_{k+1} \xi_{k+1}^{\ell}, 
\end{align*}
and extending using the right (for $X^-$) and the left (for $X^+$) supermodule structures.
These maps are both of degree $(-2, 0)$.
\end{defn}

\begin{lem}\label{lem:XnilHecke}
For all $i,j \ge 0$ we have 
\begin{equation}\label{eq:XnilHecke}
\xi^{i} \otimes_{k+1} \xi^{j} = X^\pm(\xi^{i+1} \otimes_{k+1} \xi^{j}) - X^\pm(\xi^{i} \otimes_{k+1} \xi^{j})\xi = \xi X^\pm(\xi^{i} \otimes_{k+1} \xi^{j}) - X^\pm(\xi^{i} \otimes_{k+1} \xi^{j+1})
\end{equation}
and thus
\begin{align*}
X^\pm(\xi^{i+1} \otimes _{k+1}\xi^{j+1}) = \xi X^\pm(\xi^{i} \otimes_{k+1} \xi^{j})\xi.
\end{align*}
\end{lem}

\begin{proof}
The proof is a direct computation, which is done in \cite[Lemma 7.10]{L1}.
\end{proof}

\begin{prop}\label{prop:nilbim}
The maps $X^-$ and $X^+$ are superbimodule morphisms.
\end{prop}

\begin{proof}
  Since by definition $X^-$ is a right supermodule morphism, we only need to prove that it is also a left supermodule morphism.
  This means we have to show that
\begin{align*}
X^-(x_{\alpha,k} \xi_{k+1}^{i} \otimes_{k+1} \xi_{k+2}^{j}) &= x_{\alpha,k} X^-(\xi_{k+1}^{i} \otimes_{k+1} \xi_{k+2}^{j}), \\
X^-((s_{\alpha,k+1} + \xi_{k+1}s_{\alpha+1,k+1}) \xi_{k+1}^{i} \otimes_{k+1} \xi_{k+2}^{j}) &= (s_{\alpha,k+1} + \xi_{k+1}s_{\alpha+1,k+1}) X^-(\xi_{k+1}^{i} \otimes_{k+1} \xi_{k+2}^{j}) , 
\end{align*}
for all $i,j \ge 0$ and $\alpha \le k$.
Using Lemma~\ref{lem:kk+1rel} we compute 
\begin{align*}
X^-(x_{\alpha,k}  \xi_{k+1}^{i} \otimes_{k+1} \xi_{k+2}^{j}) &\overset {(\ref{eq:ytosum})}{=} \sum_{\ell=0}^\alpha (-1)^\ell X^-(\xi_{k+1}^{i+\ell} \otimes_{k+1} \xi_{k+2}^{j} x_{\alpha-\ell,k+1}) \\
&\overset {(\ref{eq:ytosum})}{=} \sum_{\ell=0}^\alpha \sum_{p=0}^{\alpha-\ell} (-1)^{\ell+p} X^-(\xi_{k+1}^{i+\ell} \otimes_{k+1} \xi_{k+2}^{j+p})\psi^*(x_{\alpha-\ell-p,k+2}), \\
 x_{\alpha,k} X^-(\xi_{k+1}^{i} \otimes_{k+1} \xi_{k+2}^{j}) &\overset {(\ref{eq:ytosum})}{=}  \sum_{\ell=0}^\alpha (-1)^\ell \xi_{k+1}^\ell X^-(\xi_{k+1}^{i} \otimes_{k+1} \xi_{k+2}^{j})x_{\alpha-\ell,k+1} \\
&\overset {(\ref{eq:ytosum})}{=} \sum_{\ell=0}^\alpha \sum_{p=0}^{\alpha-\ell} (-1)^{\ell+p} \xi_{k+1}^{\ell}X^-(\xi_{k+1}^{i} \otimes_{k+1} \xi_{k+2}^{j})\xi_{k+2}^p\psi^*(x_{\alpha-\ell-p,k+2}). 
\end{align*}
These sums are equal by Lemma~\ref{lem:XnilHecke}.

To prove the second relation in the statement we slide $s_{\alpha,k+1}$ and $s_{\alpha+1,k+1}$ to the right through the tensor products $\otimes_{k+1}$ to get
\begin{align*}
X^-((s_{\alpha,k+1}& + \xi_{k+1}s_{\alpha+1,k+1}) \xi_{k+1}^{i} \otimes_{k+1} \xi_{k+2}^{j}) \\ 
\overset {(\ref{eq:xiraction})}{=}& X^-(\xi_{k+1}^{i} \otimes_{k+1} \xi_{k+2}^{j}) s_{\alpha,k+2} +  X^-(\xi_{k+1}^{i} \otimes_{k+1} \xi_{k+2}^{j+1}) s_{\alpha+1,k+2}  \\
&+X^-(\xi_{k+1}^{i+1} \otimes_{k+1} \xi_{k+2}^{j})s_{\alpha+1,k+2} + X^-(\xi_{k+1}^{i+1} \otimes_{k+1} \xi_{k+2}^{j+1}) s_{\alpha+2,k+2}\\
\overset{(\ref{eq:XnilHecke})}{=}& X^-(\xi_{k+1}^{i} \otimes_{k+1} \xi_{k+2}^{j}) s_{\alpha,k+2} +  X^-(\xi_{k+1}^{i} \otimes_{k+1} \xi_{k+2}^{j})\xi_{k+2} s_{\alpha+1,k+2}  \\
&+\xi_{k+1} X^-(\xi_{k+1}^{i} \otimes_{k+1} \xi_{k+2}^{j}) s_{\alpha+1,k+2} + \xi_{k+1}X^-(\xi_{k+1}^{i} \otimes_{k+1} \xi_{k+2}^{j}) \xi_{k+2} s_{\alpha+2,k+2}\\
\overset {(\ref{eq:xiraction})}{=}& (s_{\alpha,k+1} + \xi_{k+1}s_{\alpha+1,k+1})  X^-(\xi_{k+1}^{i} \otimes_{k+1} \xi_{k+2}^{j}) .
\end{align*}
The proof for $X^+$ is similar. 
\end{proof}

\begin{prop}\label{prop:injOmega}
There is an injective superbimodule map
$$u : \Omega_{k(k-1)k} \hookrightarrow \Omega_{k(k+1)k},$$
preserving the degree and given by 
\begin{align*}
u &= (\epsilon \otimes_k \id)\circ(\id \otimes_{k-1} X^- \otimes_{k+1} \id)\circ(\id \otimes_k \eta)\\
&= (\id \otimes_k \epsilon)\circ(\id \otimes_{k+1} X^+ \otimes_{k-1} \id)\circ(\eta \otimes_k \id).
\end{align*}
Moreover, this morphism takes the form 
\begin{equation}\label{eq:u}
u(\xi_{k}^{i} \otimes_{k-1} \xi_{k}^{j}) = -\xi_{k+1}^{j} \otimes_{k+1} \xi_{k+1}^{i} , 
\end{equation}
for all $i+j < k$.
\end{prop}

\begin{proof}
Thanks to Lemma~\ref{lem:decomp}, it is enough to show that the two superbimodule morphisms $(\epsilon \otimes \id)\circ(\id \otimes X^- \otimes \id)\circ(\id \otimes \eta)$ and $(\id \otimes \epsilon)\circ(\id \otimes X^+ \otimes \id)\circ(\eta \otimes \id )$ take the form (\ref{eq:u}).
 First, we suppose that $u =  (\epsilon \otimes \id)\circ(\id \otimes X^- \otimes \id)\circ(\id \otimes \eta)$ and
   we compute, for $i < k$,
\begin{align*}
u(\xi_{k}^{i} \otimes_{k-1} 1) &= -\sum_{\ell=0}^k \sum_{p=0}^{k-\ell-1} (-1)^{i-p} \phi^*(Y_{i-\ell-p,k})\xi_{k+1}^p \otimes_{k+1} x_{\ell,k} . 
\end{align*}
By Lemma~\ref{lem:kk+1rel} we have
\begin{align*}
1 \otimes_{k+1} \xi_{k+1}^{i} &\overset{(\ref{eq:xitosum}), (\ref{eq:Ytosum})}{=} (-1)^{i} \sum_{\ell=0}^{i} \sum_{p=0}^{i-\ell} (-1)^{p}\phi^*(Y_{i-\ell-p,k})\xi_{k+1}^p \otimes_{k+1} x_{\ell,k}.
\end{align*}
Since $Y_{i-\ell-p,k} = 0$ for $\ell+p > i$, we get $u(\xi_{k}^{i} \otimes_{k-1} 1)  = -1 \otimes_k \xi_{k+1}^{i}$. Using this result together with Lemma~\ref{lem:kk+1rel}, we compute
\begin{align*}
\xi_{k}^{i} \otimes_{k-1} \xi_{k}^{j}%
&\overset{(\ref{eq:xitosum}), (\ref{eq:ytosum})}{=}  (-1)^{j} \sum_{\ell=0}^{j} \sum_{p=0}^{\ell}(-1)^p \psi^*(x_{\ell-p,k}) \xi_{k}^{i+p} \otimes_{k-1} Y_{j-\ell,k}, \\
u(\xi_{k}^{i} \otimes_{k-1} \xi_{k}^{j}) &= - (-1)^{j} \sum_{\ell=0}^{j} \sum_{p=0}^{\ell}(-1)^p \phi^*(x_{\ell-p,k})  \otimes_{k+1} \xi_{k+1}^{i+p} \phi^*(Y_{j-\ell,k}) \\
&=  - (-1)^{j} \sum_{r=0}^{j} \sum_{s=0}^{j-r}(-1)^s x_{r,k} \otimes_{k+1} \xi_{k+1}^{i+s} \phi^*(Y_{j-r-s,k}), \\
-\xi^{j}_{k+1} \otimes_{k+1} \xi_{k+1}^{i} &\overset{(\ref{eq:xitosum}), (\ref{eq:Ytosum})}{=} -(-1)^{j} \sum_{\ell=0}^{j} \sum_{p=0}^{j-\ell} (-1)^{p} x_{\ell} \otimes_{k+1} \phi^*(Y_{j-\ell-p,k}) \xi_k^{i+p},
\end{align*}
where we have used a change of variable $r = \ell-p, s = p$ in the middle sum. 
Similar computations beginning with the case $1 \otimes_{k-1} \xi_{k}^{j}$
give the same result for
$(\id \otimes \epsilon)\circ(\id \otimes X^+ \otimes \id)\circ(\eta \otimes \id)$. 
Finally, injectivity follows again from the fact that $\Omega_{k(k-1)k}$ and $\Omega_{k(k+1)k}$ are free $\bQ$-modules.
\end{proof}

Thanks to the injection $u$ we see $\Omega_{k(k-1)k}$ as a sub-superbimodule of $\Omega_{k(k+1)k}$ and
we define the quotient 
$$\frac{\Omega_{k(k+1)k}}{\Omega_{k(k-1)k}} = \frac{\Omega_{k(k+1)k}}{\Image u}.$$
A priori this bimodule may not belong to $\extflag$. However, as we will see, it is isomorphic to some $1$-morphism in $\extflag$.

\begin{lem}\label{lem:mupi}
The maps $\mu$ and $\pi$ induce surjective morphisms on the quotient
\begin{align*}
\overline \mu &: \Omgq \twoheadrightarrow \Pi\Omega_k^\xi, &
\overline {\pi} &: \Omgq \twoheadrightarrow \Omega_k^\xi,
\end{align*}
of degrees respectively $(2k+2, -2)$ and $(-2k, 0)$.
\end{lem}

\begin{proof}
We have to show that $\Image u \subset \ker \mu$ and $\Image u \subset \ker \pi$.
By Lemma~\ref{lem:decomp}, it is sufficient to show that the maps $\mu$ and $\pi$ are zero on
$1 \otimes_{k+1} 1, \xi_{k} \otimes_{k+1} 1,\ \dotsc\ , \xi_k^{k-1} \otimes_{k+1} 1$, which is immediate from the definition of these maps.
\end{proof}

\begin{lem}\label{lem:iotainduced}
The morphism
$$\overline{\iota} \colon  \Omega_k^\xi \to \Omgq,$$
defined as the composite of $\iota$ with the projection on the quotient, is still injective and the inverse of $\overline{\pi}$.
\end{lem}

\begin{proof}
  To show injectivity, we only have to prove that $\Image \iota \cap \Image u = \{0\}$ which is straightforward, since 
  by Lemma \ref{lem:decomp} there are no occurrences of $\xi_k^{\ge k}\otimes_{k+1} 1$ in $\Image u$.
  The invertibility property is immediate. 
\end{proof}

\begin{lem}\label{lem:muinverse}
The induced morphism $\overline \mu$ is right invertible, with inverse given by
\begin{align*}
\overline \mu^{-1} : \Pi\Omega_k^\xi &\hookrightarrow \Omgq ,&
 \xi^i \mapsto& \sum_{\ell=0}^i (-1)^\ell x_{\ell,k} \otimes_{k+1} \xi_{k+1}^{i-\ell} s_{k+1,k+1}+ \Image u \\
&&=& \sum_{\ell=0}^i (-1)^\ell\xi_{k+1}^{i-\ell} \otimes_{k+1}   x_{\ell,k}  s_{k+1,k+1} + \Image u.
\end{align*}
\end{lem}

\begin{proof}
  It suffices to prove that $\xi_{k+1}$ slides over the tensor product, as $s_{k+1,k+1}$ already does. For $i \ge k$, this comes from Proposition~\ref{prop:xislides}. For $i < k$ it follows from the fact that for all $X, Y \in \Omega_k$ and $r+s < k$ we have
\begin{align*}
u(s_{k,k} \xi_k^r X& \otimes_{k-1} \xi_k^s Y  - (-1)^{p(X)+p(Y)} \xi_k^r X \otimes_{k-1}\xi_k^s Y s_{k,k} ) \\
\overset {(\ref{eq:xiraction})}{=}& \xi_k^s X \otimes_{k+1} Y \xi_{k+1}^r (s_{k,k+1} + \xi_{k+1} s_{k+1,k+1}) \\
& - (-1)^{p(X)+p(Y)}  (s_{k,k+1} + \xi_{k+1} s_{k+1,k+1}) \xi^s X \otimes_{k+1} Y \xi_{k+1}^r \\
= &\xi_{k+1}^s X \otimes_{k+1} Y \xi_{k+1}^{r+1} s_{k+1,k+1} - (-1)^{p(X)+p(Y)} \xi_{k+1}^{s+1} s_{k+1,k+1} X \otimes_{k+1} Y \xi_{k+1}^r,
\end{align*}
since $s_{k,k+1}$ commutes. In the quotient, this is zero and thus $\xi_{k+1}s_{k+1,k+1}$ commutes.
The invertibility is showed by the same computations as in the proof of Proposition~\ref{prop:iotainv}.
\end{proof}

\begin{lem}\label{lem:gdimequals}
There is an equality of graded dimensions
$$\gdim \Omgq =\left( q^{2k} + \pi\lambda^2 q^{-2k-2} \right) \gdim \Omega_k^\xi.$$
\end{lem}

\begin{proof}
Let $Q = \gdim \bQ[\xi] = \frac{1}{1-q^2}$. We compute
\begin{align*}\allowdisplaybreaks
  \gdim \Omega_{k(k-1)k} &=  \frac{(\gdim \Omega_{k,k-1})^2}{\gdim \Omega_{k-1}} = \frac{ \left( Q (1+q^{-2k}\pi\lambda^2) \gdim \Omega_{k-1} \right)^2}{\gdim \Omega_{k-1}}
  \\[1ex]
&= Q^2 (1-q^{2k})  (1+q^{-2k}\pi\lambda^2)  \gdim \Omega_k,
\intertext{and} 
\gdim \Omega_{k(k+1)k} &= \frac{(\gdim \Omega_{k,k+1})^2}{\gdim \Omega_{k+1}} =   \frac{ \left( Q (1+q^{-2k-2}\pi\lambda^2) \gdim \Omega_k \right)^2   }{\frac{1 + q^{-2k-2}\pi\lambda^2}{1-q^{2k+2}}\gdim \Omega_k}
\\[1ex]
&= Q^2(1-q^{2k+2})  \left( 1 + q^{-2k-2}\pi\lambda^{2} \right)  \gdim \Omega_k.
\intertext{Therefore,} 
\gdim \Omgq &= \gdim {\Omega_{k(k+1)k}} - \gdim {\Omega_{k(k-1)k}}
\\[1ex]
&= Q  \left( q^{2k} + \pi\lambda^2 q^{-2k-2} \right) \gdim \Omega_k,
\end{align*}
as stated. 
\end{proof}

\begin{lem}\label{lem:mupiortho}
The equalities
\[ 
\overline \mu \circ \overline {\iota} = 0 \mspace{30mu}\text{and}\mspace{30mu}  \overline{\pi} \circ \overline \mu^{-1}  = 0,
\]
hold.
\end{lem}

\begin{proof}
It suffices to consider the occurrences of $s_{k+1,k+1}$ in the image of $\overline\iota$ and $\overline\mu^{-1}$, and the claim follows.
\end{proof}

\begin{thm}\label{thm:bim-quotses} 
There is a degree preserving isomorphism
$$\Omgq \cong \Pi\Omega_k^\xi \langle -2k-2,2 \rangle \oplus \Omega_{k}^\xi \langle 2k, 0 \rangle ,$$
given by $\overline \mu \oplus \overline {\pi}$ with inverse $\overline \mu^{-1} \oplus \overline{\iota}$.
\end{thm}

\begin{proof}
  The Lemmas~\ref{lem:mupi}, \ref{lem:iotainduced}, \ref{lem:muinverse}, \ref{lem:gdimequals} and~\ref{lem:mupiortho} above imply that $\overline \mu \oplus \overline {\pi}$
  is a split surjection with inverse $\overline \mu^{-1} \oplus \overline{\iota}$ and that the dimensions agree, and thus it is an isomorphism.
\end{proof}

\begin{rem}\label{rem:directsum}
We do \emph{not} have a direct sum decomposition  
$$\Omega_{k(k+1)k} \ncong \Omega_{k(k-1)k} \oplus  \Pi\Omega_k^\xi \langle -2k-2,2 \rangle \oplus \Omega_{k}^\xi \langle 2k, 0 \rangle$$
because there is no surjective morphism $\Omega_{k(k+1)k} \rightarrow \Omega_{k(k-1)k}$. 
As a matter of fact, there is no injective morphism $\Omega_k \rightarrow  \Omega_{k(k-1)k}$ either.
\end{rem}

\medskip

\begin{defn}\label{def:shiftedbim}Define the shifted superbimodules 
\begin{align*}
\omg_{k+1,k}&=\Omega_{k+1,k}\brak{-k,0}, &
\omg_{k,k+1}&=\Omega_{k,k+1}\brak{k+2,-1}, &
\omg_{k}^\xi &=  \Pi\Omega_k[\xi]\brak{1,0} .
\end{align*} 
\end{defn}

\medskip

In terms of these superbimodules the isomorphism in Theorem~\ref{thm:bim-quotses} takes the form
(the notation $\omg_{k(k+1)k}$ and $\omg_{k(k-1)k}$ should be clear) 
\begin{equation}\label{eq:omgcommutator}
\frac{\omg_{k(k+1)k}}{\omg_{k(k-1)k}} \cong \omg_k^\xi \brak{-2k-1,1} \oplus \Pi\omg_k^\xi \brak{2k+1,-1}.
\end{equation}

Note that by (\ref{eq:omgcommutator}) we get
\begin{align*}
\gdim \omg_k^\xi &= -\frac{\pi}{q-q^{-1}} \gdim \Omega_k, \\
\gdim \frac{\omg_{k(k+1)k}}{\omg_{k(k-1)k}} &= -\frac{\pi q^{-2k}(\lambda q^{-1}) + q^{2k}(\lambda q^{-1})^{-1}}{q-q^{-1}} \gdim \Omega_k,
\end{align*}
which agrees with the commutator relation (\ref{eq:slcommutator}) for $\lambda q^{-1}$ and $e_{-2k}$ when specializing $\pi = -1$.

\smallskip

We now arrive to the main result in this section as
 a corollary of Theorem~\ref{thm:bim-quotses}.   

\begin{cor}\label{prop:sesbim}
There is a short exact sequence
\begin{equation}
0 \rightarrow \omg_{k(k-1)k}\rightarrow \omg_{k(k+1)k} \rightarrow  \omg_k^\xi \brak{-2k-1,1} \oplus \Pi\omg_k^\xi \brak{2k+1,-1} \rightarrow 0. \label{eq:exseqbim}
\end{equation}
\end{cor}

The following result will be useful in the sequel. 
\begin{prop}\label{prop:sweet}
The superbimodules $\omg_{k(k+1)k}$ and $\omg_{k(k-1)k}$ are sweet and decompose as left $\Omega_k$-modules as 
\begin{align*}
\omg_{k(k+1)k}  &\overset{\text{$\Omega_k\smod$}}{\cong} \oplus_{[k+1]} \left(\Omega_k\brak{k+2,-1} \oplus \Pi\Omega_k\brak{-k,1}\right) \otimes \bQ[\xi],\\
\omg_{k(k-1)k} &\overset{\text{$\Omega_k\smod$}}{\cong}  \oplus_{[k]} \left(\Omega_k\brak{k+1,-1} \oplus \Pi\Omega_k\brak{-k+1,1}\right) \otimes \bQ[\xi],
\end{align*}
and as right $\Omega_k$-modules, as
\begin{align*}
\omg_{k(k+1)k}  &\overset{\mathrm{smod}\text{-}\Omega_k}{\cong} \oplus_{[k+1]} \left(\Omega_k\brak{k+2,-1} \oplus \Pi\Omega_k\brak{-k,1}\right) \otimes \bQ[\xi],\\
\omg_{k(k-1)k} &\overset{\mathrm{smod}\text{-}\Omega_k}{\cong}  \oplus_{[k]} \left(\Omega_k\brak{k+1,-1} \oplus \Pi\Omega_k\brak{-k+1,1}\right) \otimes \bQ[\xi],
\end{align*}
\end{prop}

\begin{proof}
By Lemma~\ref{lem:kk+1rel}, there are decompositions as left supermodules 
\begin{align*}
\Omega_{k+1,k} &\overset{\Omega_k\smod}{\cong} \oplus_{\{k+1\}} \Omega_{k+1},\\
\Omega_{k,k+1} &\overset{\Omega_k\smod}{\cong} (1+s_{k+1,k+1}) \Omega_{k} \otimes \bQ[\xi],
\end{align*}
such that
\begin{align*}
\omg_{k+1,k} &\overset{\mathrm{smod}\text{-}\Omega_k}{\cong} \oplus_{[k+1]} \Omega_{k+1},\\
\omg_{k,k+1} &\overset{\mathrm{smod}\text{-}\Omega_k}{\cong} \left(\Omega_k\brak{k+2,-1} \oplus \Pi\Omega_k\brak{-k,1}\right)\otimes \bQ[\xi].
\end{align*}
We conclude by combining these two decompositions. The proof is similar for the decomposition as right supermodules.
\end{proof}

\begin{rem}
  The decompositions as a left and as a right supermodule are similar but the splitting maps are
  \emph{not} superbimodule maps  (c.f. Remark~\ref{rem:directsum}).
\end{rem}

\subsection{nilHecke action}\label{ssec:nilHecke}

The \emph{nilHecke algebra} $\nh_n$, which appears in the context of cohomologies of flag varieties and 
Schubert varieties~(see for example \cite[\S4]{KK}), 
is an essential ingredient in the categorification of quantum groups 
and has become quite ubiquitous in higher representation theory. 
Recall that it is the unital, associative $\Bbbk$-algebra freely generated by $x_j$ for $1\leq j\leq n$ 
and $\partial_j$ for $1\leq j\leq n-1$ 
with relations 
\begin{equation}\label{eq:nhdef}
\begin{aligned}
x_ix_j &= x_jx_i,  \\
\partial_ix_j &= x_j\partial_i  \quad\text{ if }\vert i-j\vert >1, 
& \partial_i\partial_j &= \partial_j\partial_i\quad\text{ if }\vert i-j\vert >1,\\
\partial_ix_i &= x_{i+1}\partial_i + 1, & \partial_i^2 &= 0, & \\
x_i\partial_i &= \partial_ix_{i+1} + 1, & \partial_i\partial_{i+1}\partial_i &=\partial_{i+1}\partial_{i}\partial_{i+1} .  
\end{aligned}
\end{equation}
Here $\Bbbk$ is a ring which, unless stated otherwise, we will take as $\bQ$.

\begin{prop}
There is an action of the nilHecke algebra $\nh_n$ on 
$\Omega_{m,m+n}$. 
\end{prop}

\begin{proof}
  We view $\Omega_{m,m+n}$ as $\Omega_{m,m+1} \otimes \dots \otimes \Omega_{m+n-1,m+n}$ using Lemma \ref{lem:bimtens}. 
Let $\partial_i$ act as the operator $X^- : \Omega_{m+i-1,m+i} \otimes \Omega_{m+i, m+i+1} \rightarrow \Omega_{m+i-1,m+i} \otimes \Omega_{m+i, m+i+1}$ and  $x_i$ as multiplication by $\xi_{m+i}$ in $\Omega_{m+i-1, m+i}$.
  We get all the relations in~(\ref{eq:nhdef}) from the superbimodule structure of the morphism $X^{-}$
  together with Lemma~\ref{lem:XnilHecke}, except for the last two on the second column,
  which can be checked through computations similar to those in~\cite[Lemma 7.10]{L1}.
\end{proof}

As a matter of fact,
there is an enlarged version of the nilHecke algebra acting on $\Omega_{m,m+n}$, and therefore on $\Omega_{m+n,m}$. 
From the proof of Proposition~\ref{prop:nilbim} we see that the nilHecke algebra $\nh_n$ as defined above acts on
the ring $\bQ[\und{x}_n,\und{\omega}_n]=\bQ[\und{x}_n]\otimes \bV^\bullet(\und{\omega}_n)$, where 
$\omega_j$ is odd and has bidegree $\deg_{\lambda, q}(\omega_j)=(-2(j+m),2)$.
More precisely, $\omega_j$ is identified with $s_{m+j,m+j} \in \Omega_{m+j-1, m+j}$.  

\begin{defn}We define the bigraded (super)algebra $A_{n}(m)$ as the quotient of the product algebra of 
$\nh_n$ with $\bV^\bullet( \und{\omega}_n)$
by the kernel of the action of $\nh_n$ on $\bQ[\und{x}_n,\und{\omega}_n]$. 
\end{defn}

The algebra $A_{n}$ inherits many of the features of $\nh_n$, like the fact that it is left and right noetherian 
and is free as a left supermodule over $\bQ[\und{x}_n,\und{\omega}_n]$ and $\bQ[\und{x}_n]$,  
of ranks $n!$ and $2^n n!$ respectively. 
It can be given an explicit presentation as a smash product as follows.
\begin{prop} 
As an abelian group $A_n(m)=\nh_n\otimes\,\bV^\bullet( \und{\omega}_n)$, where $\nh_n$ and $\bV^\bullet( \und{\omega}_n)$ 
are subalgebras and 
\begin{align*}
 x_i\omega_j &= \omega_jx_i, 
&
\partial_i\omega_j &= \begin{cases}
\omega_j\partial_i  & \text{ if }i \ne j , 
\\[1.5ex] 
\omega_{i}\partial_i 
+ \omega_{i+1} \bigl( \partial_i x_{i+1} - x_{i+1}\partial_i \bigr) 
& \text{ if } i = j .
\end{cases}
\end{align*}
\end{prop}

\begin{proof}
Only the last relation calls for a proof. We have
\begin{align*}
X^-( x s_{k,k} \otimes_k y) &= X^-(x \otimes_k (s_{k,k+1} + \xi_{k+1} s_{k+1,k+1}) y)\\
&= X^-(x \otimes_k y) s_{k,k+1} + X^-(x \otimes_k y (\xi_{k+1} s_{k+1,k+1})) \\
&= s_{k,k} X^-(x \otimes_k y) - X^-(x \otimes_k y) \xi_{k+1} s_{k+1,k+1}  + X^-(x \otimes_k y (\xi_{k+1} s_{k+1,k+1})),
\end{align*}
for any $x \in \Omega_{k-1,k}$ and $y \in \Omega_{k,k+1}$ with $p(x)+p(y) = 0$. The case with parity one is similar. Take $k = m+i-1$ and we get the relation.
Faithfulness comes from the basis constructed in Proposition~\ref{prop:Anbasis}.
\end{proof}

%
%
\section{Topological Grothendieck groups}\label{sec:completedgrothendieck}

Recall that the Grothendieck group (resp. split Grothendieck group), denoted $G_0$ (resp. $K_0$),
is  defined in general for abelian (resp. additive) categories as the free group generated by the classes of objects up to isomorphism quotiented by
\[
[B] = [A] + [C],
\]
whenever there exists a short exact sequence $0 \rightarrow A \rightarrow B \rightarrow C \rightarrow 0$ (resp. an isomorphism $B \cong A \oplus C$). We call \emph{distinguished triplet} $(A,B,C)$
these short exact sequences 
and decompositions into direct sums.
 In the case of abelian categories with finite length objects (resp. Krull--Schmidt categories), they are given by the free $\bZ$-module generated by the classes of simple objects  (resp. indecomposable objects).

\smallskip

If $\cC$ has a (not necessarily strong) $\bZ/2\bZ$-action $\Pi : \cC \rightarrow \cC$, then $G_0(\cC)$ and $K_0(\cC)$ become modules over
$\bZ_\pi=\bZ[\pi]/(\pi^2-1)$ with $\pi[M]=[\Pi M]$. 
 When the category is \emph{strictly $\bZ$-graded}, namely $A \brak{k} \not\cong A$ for all $A \in \cC$ and $k\in \bZ_0$ , then $K_0(\cC)$ (resp. $G_0(\cC)$) becomes as $\bZ[q, q^{-1}]$-module freely generated by the classes of indecomposable objects (resp. simple objects), up to shift. The action is given by a shift in the degree
\begin{align*}
q [M] &= [M\brak{1}], & q^{-1} [M] &= [M\brak{-1}].
\end{align*}
This means the action of a polynomial $p\in\bZ[q,q^{-1}]$ can be viewed as
(cf. \S\ref{ssec:gdims})
\[
p(q,q^{-1}) [M] =  \bigl[M^{\oplus{p}} \bigr].
\]
This story generalizes in the obvious way to the case of multigraded categories.  

\smallskip

We look for similar results when working with objects admitting infinite filtrations
or that decompose into infinitely many indecomposables, such that the Grothendieck groups
become modules over $\bZ\llbracket q\rrbracket[q^{-1}]$.
This will allow us making sense of expressions like
\[
[X] = (1 + q^2 + q^4 + \dots + q^{2i} + \dots)[S] = \frac{1}{1-q^2}[S],
\]
for some objects $X$ and $S$ such that $(X\brak{2},X,S)$ is a distinguished triplet. 
In general this procedure fails and one can see easily that the Grothendieck group collapses using
Eilenberg swindle arguments. 
To avoid this outcome in our construction, we work with categories where
these decompositions and filtrations are controlled and essentially unique. 

\medskip

This section is a bit more technical but can be regarded as having some interest on its own. 
The reader who is only interested in the general construction can easily skip to \S\ref{sec:vermacat} without worrying too much. Most of the arguments are sensibly similar to the ones used in the finite case, which 
can be found for example in the appendix of~\cite{mathas}.

\subsection{Topological split Grothendieck group $K_0$ }\label{ssec:completedK0}

The aim of this section is to define a notion of Krull--Schmidt categories admitting infinite decompositions.
For the split Grothendieck group not to collapse, we need to control the occurrences of the
indecomposables in these decompositions and they should be essentially unique. That is for every other possible decomposition, the indecomposables are in bijection and have the same grading.
Since we are working in a graded context, we require that in each decomposition the indecomposables are in a finite number in each degree and the degrees are bounded from below.

\begin{defn}
We say that a coproduct in a strictly $\bZ$-graded category $\cC$ is \emph{locally finite} if it is finite in each degree. By this we mean the coproduct is of the form 
\[
\coprod_{i \in \bZ}( A_1^{\oplus k_{1,i}} \oplus \dots \oplus A_n^{\oplus k_{n,i}})\brak{i},
\]
for some $A_1, \dots, A_n \in \cC$ and $k_{j,i} \in \bN$.
Moreover, we say that it is \emph{left bounded}, or \emph{bounded from below}, if there is some $m \in \bZ$ such that $k_{j,i} = 0$ for all $i < m$.
\end{defn}

An additive category admits all finite products and coproducts, and those are equivalent and called biproducts. In the same spirit, we define the stronger notion of right complete locally additive category. 

\begin{defn}
  We say that an additive, strictly $\bZ$-graded category $\cC$ is \emph{locally additive}
  if all its locally finite coproducts are biproducts, that is, they are isomorphic to their product counterparts.
  We write them with a $\bigoplus$ sign and sometimes call them direct sums.
  Moreover, we say that $\cC$ is \emph{right complete} if it admits all
  left bounded locally finite coproducts.
\end{defn}

We illustrate this notions in the working example below, that will be developed further 
throughout this section. 

\begin{exe}\label{ex:Rmod}
  Let $R$ be a unital graded $\Bbbk$-algebra, with $\Bbbk$ being a field.
  Suppose $R = \bigoplus_{i \ge 0} R_i$ is locally finite dimensional with positive dimension and $R_0 = \Bbbk$.
  We call \emph{locally finitely generated $R$-module} a graded $R$-module $M$ that can be written as $M = \bigoplus_{i \in I} Rx_i$ with $x_i \in M$ and  $X = \{x_i\}_{i \in I}$ is finite in each degree.

  \smallskip
  
  It is \emph{left bounded} if there is some $m \in \bZ$ such that $\deg(x_i) >m $ for all $i \in I$.
The category $R\fgmod$ of left bounded locally finitely generated $R$-modules with degree $0$ morphisms is right complete and locally additive.

\smallskip

We clearly have all left bounded locally finite coproducts. We show that they are biproducts. Let 
\[
M =  \bigoplus_{i > m} (A_1^{\oplus k_{1,i}} \oplus \dots \oplus A_n^{ \oplus k_{n,i}}) \brak{i} \in R\fgmod
\]
be a coproduct and $Z$ be an object in $R\fgmod$ with morphisms $f_{j,i} : Z \rightarrow A_j^{k_{j,i} }$ (note that $A_j^{\oplus k_{j,i}}$ are biproducts). If $i$ is bounded, then it is a finite coproduct and thus a biproduct, so we suppose without losing generality that for all $m \in \bZ$ there exist $i,j \in \bZ$ with $i > m$ such that $k_{j,i} > 0$.
 By the universal property of the direct product, there is a canonical map (in the category of all $R$-modules, but not in $ R\fgmod$) 
\[
r :  M \rightarrow \prod_{i \in \bZ} (A_1^{\oplus k_{1,i}} \oplus \dots \oplus  A_n^{\oplus k_{n,i}}) \brak{i}.
\]
We want to show that the map of $R$-modules
\[
\prod_{i,j} f_{j,i} : Z \rightarrow \prod_{i \in \bZ} (A_1^{\oplus k_{1,i}} \oplus \dots \oplus  A_n^{\oplus k_{n,i}}) \brak{i}
\]
factors through $r$. This is equivalent to show that for homogeneous $x \in Z$ fixed, $f_{i,j}(x) = 0$ for almost all $i,j$. Suppose this does not hold. Then there is a $j \in \bZ$ such that for each $m \in \bZ$ there exist some $i \in \bZ$ with $i > m$ and $f_{i,j}(x)  \ne 0$. Thus $\deg (f_{j,i}(x)) = \deg(x) - i$. But $f_{j,i}(x)$ is an homogeneous element of $A_j$ which is left bounded and that is absurd. By construction $\prod_{i,j} f_{j,i}$ is the unique morphism satisfying the universal property of the product and thus $M$ is a biproduct.
\end{exe}

\begin{rem}\label{lem:hominj}
In a locally additive category, there is a canonical bijection
\[
\Hom\biggl(X, \bigoplus_{k\in I} Y_k\biggr) \cong \prod_{k\in I} \Hom(X, Y_k).
\]
\end{rem}

\begin{defn}
An object $A$ in a category $\cC$ is \emph{small} if every map $f : A \rightarrow \coprod_{i \subset I} B_i$ factors through $\coprod_{j \in J} B_j$ for a finite subset $J \subset I$.
\end{defn}

\begin{exe}
In a category of modules, finitely generated modules are small~\cite[\S2]{infkrullschmidt}.
\end{exe}

\begin{defn}
We say that a locally additive category is \emph{locally Krull--Schmidt} if every object decomposes into a locally finite direct sum of small objects having local endomorphisms rings.
\end{defn}

\begin{rem}
Note that a locally Krull--Schmidt category must be idempotent complete. Moreover, an object with local endomorphism ring must be indecomposable, and has only $0$ and $1$ as idempotents.
\end{rem}

It  appears the condition of being small allows us to mimic the classical proof of the Krull--Schmidt property of a Krull--Schmidt category. There exists some other results about Krull--Schmidt properties for infinite decompositions~\cite{infkrullschmidt}, where the indecomposables are not necessarily small. However they require the category to admit kernels, which we do not have in our construction. For example, we want to use the category of projective modules which certainly does not admit all kernels.

\begin{lem}\label{lem:cancprop}
\emph{(\cite[Lemma~3.3, p.18]{bass})}
In an additive category $\cC$, for all $A,B$ and $C$ in $\cC$, if $A$ is indecomposable with local endomorphism ring, then
\[
A \oplus B \overset{f = \begin{pmatrix} f_{AA} & f_{AB} \\ f_{CA} & f_{CB} \end{pmatrix}}{\cong} A \oplus C
\] 
with $f_{AA}$ being an unit, implies that $B \cong C$.
\end{lem}

We now prove that each object in a locally Krull--Schmidt category
decomposes into an essentially unique direct sum of indecomposables. The idea of the proof is essentially the same as for the classical Krull--Schmidt theorem (see for example~\cite[Theorem~A6]{mathas}), with only the smallness property of the indecomposable objects allowing us to restrict the infinite sums of morphisms into finite ones such that we can extract units from them. Also the locally finiteness of the direct sums allows us to use inductive arguments.

\begin{thm}
In a locally Krull--Schmidt category, given an isomorphism 
\[
 \bigoplus_{i \in \bZ} (A_1^{\oplus k_{1,i}} \oplus \dots \oplus  A_n^{\oplus k_{n,i}}) \brak{i} \cong  \bigoplus_{i \in \bZ} (B_1^{\oplus k'_{1,i}} \oplus \dots \oplus  B_m^{\oplus k'_{m,i}}) \brak{i} ,
\]
where the $A_j$'s are indecomposables with $A_j \ncong A_{j'}\brak{i}$ for all $j \ne j'$ and $i \in \bZ$, and the same for the $B_j$'s
, then $m=n$, $A_s \cong B_{j_s} \brak{\alpha_s}$ for all $s$, and $k_{s,i} = k'_{j_s,i+\alpha_s}$, with $j_s \ne j_{s'}$ if $s \ne s'$. 
\end{thm}

\begin{proof}
Denote $M = \bigoplus_{i \in \bZ} (A_1^{\oplus k_{1,i}} \oplus \dots \oplus  A_n^{\oplus k_{n,i}}) \brak{i}$. 
Let 
\begin{align*}
f_{j,i,k} &:  B_j\brak{i} \rightarrow M, &
q_{j,i,k} &: M \rightarrow  B_j\brak{i},
\end{align*}
be the injection and projection morphisms given by the biproduct structure, with $1 \le k \le k'_{j,i}$. Fix $i_0 \in \bZ$. We have $\id_M = \prod_{j,i,k} f_{j,i,k}q_{j,i,k}$. Thus $\id_{A_1} = \pi( \prod_{j,i,k} f_{j,i,k}q_{j,i,k}) \imath  \in \End(A_1\brak{i_0})$ for each copy of $A_1\brak{i_0}$ in $  A_1^{\oplus k_{1,i_0}}\brak{i_0}$, with $\pi$ and $\imath$ the projection and inclusion of $A_1\brak{i_0}$ in $M$.  
Since $A_1$ is a small object, we can restrict this sum to a finite one. 
Thus we can write $\pi ( \prod_{j,i,k} f_{j,i,k}q_{j,i,k}) \imath$ as a finite sum over $j$ and $k$. By local property of $\End(A_1\brak{i_0})$, there exists $j,i,k \in \bZ$ such that $x = \pi fq\imath$ is a unit, with $f = f_{j,i,k}$ and $q = q_{j,i,k}$. Take $q \imath x^{-1} \pi f\in \End(B_j\brak{i})$ which is an idempotent, and thus is $0$ or $\id_{B_{j}\brak{i_1}}$. Since it factors through $\id_{A_1\brak{i_0}}$, it cannot be zero. Thus $q \imath x^{-1}$ is an isomorphism with inverse $\pi f$ such that $A_1\brak{i_0} \cong B_j\brak{i}$, hence $A_1 \cong B_j\brak{i - i_0}$. We apply the same reasoning for each $A_j\brak{i}$. Since the argument can be applied for the $B_j$'s as well, we get $n = m$ and $A_s \cong B_{j_s}\brak{\alpha_{s}}$ for some $\alpha_s \in \bZ$.
Now using Lemma~\ref{lem:cancprop} to cancel each $A_s\brak{i}$ with $B_{j_s}\brak{i+\alpha_s}$ we conclude that $k_{s,i} = k'_{j_s,i+\alpha_s}$.
\end{proof}

\begin{cor}
A locally Krull--Schmidt category $\cC$ possesses the cancellation property for direct sums, namely for all $A,B$ and $C$ in $\cC$,
$
A \oplus B \cong A \oplus C
$
implies $B \cong C$.
\end{cor}

Let $K_0'(\cC)$ be the free $\bZ\llbracket q\rrbracket[q^{-1}]$-module generated by
the classes of indecomposable objects in $\cC$, up to shift. We equip it with the $(q)$-adic topology.
The (usual) split Grothendieck group of a right complete locally Krull--Schmidt category $\cC$
has a canonical structure of $\bZ\llbracket q\rrbracket[q^{-1}]$-module with action of a series $p(q,q^{-1}) = \sum_{i > m} k_i q^i$ given by
\[
p(q,q^{-1})[M] = \left[M^{\oplus p}\right].
\]
Since each object $M \in \cC$ admits an essentially unique decomposition into
indecomposable objects
$M \cong \bigoplus_{i > m} (A_1^{\oplus k_{1,i}} \oplus \dots \oplus  A_n^{\oplus k_{n,i}}) \brak{i}$
we get a canonical surjective $\bZ\llbracket q\rrbracket[q^{-1}]$-module map
\[
f : K_0(\cC) \twoheadrightarrow K'_0(\cC) . 
\]
This induces a topology on $K_0(\cC)$, which in general is not Hausdorff as we can have
\[
0 \ne [M] -  \sum_{i > m} k_{1,i} q^i [A_1] + \dots +  k_{n,i} q^i [A_n] \in \bigcap_{n \ge 0} f^{-1}((q)^n).
\]
\begin{defn}
We define the \emph{topological split Grothendieck group} as
\[
\boldsymbol K_0(\cC) = K_0(\cC)/\cap_{n \ge 0} f^{-1}((q)^n) = K_0(\cC)/\ker f.
\]
\end{defn}
The topological split Grothendieck group possesses a canonical $(q)$-adic topology
given by the quotient topology,
making it a topological module over the topological ring $\bZ\llbracket q \rrbracket [q^{-1}]$.
From the definition we see that we have an homeomorphism
$\boldsymbol K_0(\cC) \cong K'_0(\cC)$. Therefore, we have the following theorem.

\begin{thm}
The topological split Grothendieck group of a right complete locally Krull--Schmidt category equipped with the $(q)$-adic topology is a free $\bZ\llbracket q\rrbracket[q^{-1}]$-module generated by the classes of indecomposables (up to shifts).
\end{thm}

\begin{exe}
The category $R\fgmod$ from Example~\ref{ex:Rmod} and its subcategory given by the projective modules are both right complete locally Krull--Schmidt. 
\end{exe}

\subsection{Grothendieck group $G_0$ }\label{ssec:completedG0}

Recall that an object $X$ in an abelian category $\cC$ has finite length if there exists a finite filtration, called a composition series,
\[
X = X_0 \leftarrow X_1 \leftarrow \dots \leftarrow X_n = 0,
\]
where each $X_i/X_{i+1}$ is a simple (non-zero) object. If it exists, it is unique up to permutation thanks to the Jordan--H\"older theorem. 

\smallskip

In general this result does not hold for infinite filtrations. In this section, we present some conditions that are sufficient to have uniqueness of such filtrations in a non-artinian category.

\begin{defn}
  Let $X$ be an object in an abelian category $\cC$.
  A \emph{$\bZ$-filtration} is a sequence of subobjects $X_i \subset X$ indexed by $i\in \bZ$ such that $X_{i+1} \subset X_i$ and $X_0 = X$ or $X_0 = 0$. 
We can write this as 
\begin{align*}
X &= X_0 \leftarrow X_1 \leftarrow X_2 \leftarrow \dots &
\text{or}&&
X &\leftarrow \dots \leftarrow X_{-2} \leftarrow X_{-1} \leftarrow X_0 = 0.
\end{align*}
Such a filtration is called \emph{exhaustive} if the direct limit $\varinjlim_{i} X_i \cong X$, and
\emph{Hausdorff} if the inverse limit $\varprojlim_{i} X_i \cong 0$. We say that it has \emph{simple quotients} if all quotients $X_i/X_{i+1}$ are either $0$ or simple. 
\end{defn}

As for the coproducts in strictly $\bZ$-graded category, we can define a notion of locally finiteness for the $\bZ$-filtrations.

\begin{defn}
If $\cC$ is strictly $\bZ$-graded, we say that a $\bZ$-filtration is \emph{locally finite} if there is some finite set $\{S_j\}_{j \in J}$ of objects in $\cC$ such that for all $i \in \bZ$
\begin{align*}
X_{i}/X_{i+1} &\cong S_{j_i}\{p_i\}, &\text{ or }&& X_{i}/X_{i+1} &\cong 0,
\end{align*}
for some $j_i,p_i \in \bZ$, and for each $S_{j}\{p\}$ there is a finite number $k_{j,p}$ of such $i$:
\[
k_{j,p} = \#\{i \in \bZ | X_i/X_{i+1} \cong S_j\brak{p}\} \in \bN.
\]
 We call $k_{j,p}$ the degree $p$ multiplicity of $S_j$ in the filtration.
We say the filtration is \emph{left bounded} if there exists $m \in \bZ$ such that $k_{j,p} = 0$ for all $p < m$ and $j \in J$.
\end{defn}

The infinite counterpart of a composition series for a $\bZ$-filtration we choose is defined as the following.
 
\begin{defn}
We say an object $X \in \cC$ has a \emph{$\bZ$-composition series} if it admits a locally finite, exhaustive, Hausdorff, $\bZ$-filtration with simple quotients.
\end{defn}

\begin{exe}
Let $R = \bQ[x]$ with $\deg(x) = 2$ and let $S = R/Rx$ be a simple object in the category of (graded) $R$-modules with degree 0 morphisms. Let $M$ be a module isomorphic to $R$. Then $M$ admits a $\bZ$-composition series 
\[
M \cong R \leftarrow Rx  \leftarrow Rx^2 \leftarrow \dots
\]
with $R/Rx \cong S$, $Rx/Rx^2 \cong S\brak{2}$, etc. Note that in general, for a module category and the $A_i \subset X$ being submodules, then $\varprojlim_{i} A_i = \bigcap_i A_i \subset X$ and $\varinjlim_{i} A_i = \bigcup_i A_i \subset X$, and thus $\varprojlim_{i} Rx^i = \bigcap_i Rx^i = 0$.
Another example is $R = \bQ[x,y]$, $\deg(x) = n$, $\deg(y) = m$ which has $\bZ$-composition series given by ``aliased diagonals" in $\bN^2$. 
\end{exe}

\begin{exe}\label{ex:Ralgebrafilt}
More generally, for $R$ a positively graded $\Bbbk$-algebra having locally finite graded dimension as a $\Bbbk$-vector space, one can define a $\bZ$-composition series for $R$ viewed as a module over itself. Indeed we can write $ R = \bigoplus_{i \ge 0} R^i$, with each $R^i = \bigoplus_{j=1}^{n_i} \Bbbk v^i_j$ being finite dimensional vector space in degree $i$ and $R^0 \cong \Bbbk$. Then we get a filtration
\[
R \leftarrow \bigoplus_{i \ge 1} R^i \leftarrow \bigoplus_{i \ge 2} R^i \leftarrow \dots \leftarrow 0,
\]
that can be refined into a filtration with simple quotients if we insert
\[
\leftarrow \bigoplus_{i \ge k+1} R^i \oplus \bigoplus_{j=1}^{n_i-1} \Bbbk v^k_j \leftarrow  \bigoplus_{i \ge k+1} R^i \oplus \bigoplus_{j=1}^{n_i-2} \Bbbk v^k_j \leftarrow  \dots  \leftarrow \bigoplus_{i \ge k+1} R^i \oplus  \Bbbk v^k_1  \leftarrow
\]
between each $\bigoplus_{i \ge k} R^i \leftarrow \bigoplus_{i \ge k+1} R^i$. The simple quotients are given by $\{ S_0 \}$, with $S_0 \cong \Bbbk$, and multiplicities $k_{0,i} = n_i$.
\end{exe}


As a composition series of a finite-length object is essentially unique,
we want to establish  that up to some mild hypothesis a $\bZ$-composition series is also essentially unique. One possible choice of such hypothesis is given by the following.

\begin{defn}
An object $X$ with a $\bZ$-composition series is said to be \emph{stable for the filtrations} if 
\begin{itemize}
\item for all pair of Hausdorff filtrations $X \leftarrow A_1 \leftarrow \dots \leftarrow 0 $ and $X \leftarrow B_1 \leftarrow \dots \leftarrow 0$, 
 for each $i \ge 0$ there exists $k$ such that $B_k \subset A_i$,
\item for all pair of exhaustive filtrations $X \leftarrow \dots \leftarrow A_{-1} \leftarrow 0 $ and $X \leftarrow \dots \leftarrow B_{-1} \leftarrow 0$, 
  for each $i \le 0$ there exists $k$ such that $B_i \subset A_k$.
\end{itemize}
\end{defn}

\begin{exe}\label{ex:stablefilts}
The rings and modules from the previous examples are stable for the filtrations. In general, a $\Bbbk$-algebra as $R$ in Example~\ref{ex:Ralgebrafilt}, viewed as module over itself, is stable for the filtrations, and also are its left bounded locally finite dimensional modules. This comes from the fact that such modules are $\Bbbk$-vector spaces, and therefore a filtration of modules yields a filtration of subspaces. For example, suppose $B_k \not\subset A_i$ for all $k$. Then $A_i \subsetneq A_i \cup B_k$. In addition $X/X_i$ must be a finite dimensional vector space and we get an infinite filtration
\[
X \leftarrow B_1 + A_i \leftarrow B_2  + A_i \leftarrow \dots \leftarrow A_i.
\]
Since, as vector spaces, $B_j = B_{j+1} \bigoplus H_j$ for some $H_j$, and $H_j \not\subset A_i$ for arbitrary large $j$ (if not, we would have a $B_k \subset A_i$ since $B_k \bigoplus_{j \ge k} H_k$). It means that $X = A_i \bigoplus_j H_j$ where $j$ runs over $\{j \in \bZ | H_j \not\subset A_i\}$, which contradicts the fact that $X/X_i$ is finite dimensional.
\end{exe}

\begin{exe}
The ring $\bZ$ view as a module over itself is not stable for the filtrations. Consider a sequence of non-equal prime numbers $p_0, p_1, p_2, \dots$ and the following Hausdorff filtrations
\begin{align*}
\bZ &\leftarrow p_0\bZ \leftarrow p_0p_2\bZ \leftarrow \dots \leftarrow p_0 \dots p_{2k} \bZ \leftarrow \dots \leftarrow 0, \\
\bZ &\leftarrow p_1\bZ \leftarrow p_1p_3\bZ \leftarrow \dots \leftarrow p_1 \dots p_{2k+1} \bZ \leftarrow \dots \leftarrow 0. 
\end{align*}
It is clear that $p_1\dots p_{2k+1} \bZ \not\subset p_0\bZ$ for all $k$ and thus $\bZ$ is not stable for the filtrations. As a matter of fact, the filtrations have simple quotients but are not equivalent. Indeed the quotients are respectively given by $\bZ/p_{2k}\bZ$ and $\bZ/p_{2k+1}\bZ$.
\end{exe}

We now proceed to prove that all $\bZ$-composition series of an object $X$ which is stable for the filtrations 
are essentially the same. 
Since a $\bZ$-filtration can take two forms, reaching $0$ or $X$, there are three cases to consider.
But first we introduce some useful lemmas.

\begin{lem}\label{lem:isomod}
Suppose there is an object $X$ with two subobjects $M$ and $N$ such that $M \neq N$, and $X/M$ and $X/N$ are simple, then
\[
M/(M \cap N) \cong X/N.
\]
\end{lem}

\begin{proof}
$X = M \cup N$ and thus $X/N = (M \cup N)/N \cong M/(M \cap N)$ by the second isomorphism theorem.
\end{proof}

\begin{lem} \label{lem:interfilt}
Let $M$ and $A$ be subobjects of $X$ such that $A$ admits a $\bZ$-composition series
\begin{align*}
A = A_0 &\leftarrow A_1 \leftarrow \dots \leftarrow 0, &\text{or} && A & \leftarrow \dots \leftarrow A_{-1} \leftarrow A_0 = 0.
\intertext{Then we get a $\bZ$-composition series} 
A \cap M &\leftarrow A_1 \cap M \leftarrow \dots \leftarrow 0, &\text{or} &&  A \cap M &\leftarrow \dots \leftarrow A_{-1} \cap M \leftarrow 0.
\end{align*}
\end{lem}

\begin{proof}
 For all $j$ we have
\[
\frac{A_j \cap M}{A_{j+1} \cap M} \cong \frac{A_j \cap M}{A_{j+1} \cap (A_j \cap M)} \cong \frac{A_{j+1} \cup (A_j \cap M)}{A_{j+1}}.
\]
If $A_j \cap M \subset A_{j+1}$ then we get $0$. If not we have $A_{j+1} \subsetneq A_{j+1} \cup (A_j \cap M) \subset A_j$, thus $A_j \cong A_{j+1} \cup (A_j \cap M)$. Therefore $
\frac{A_j \cap M}{A_{j+1} \cap M} \cong \frac{A_j}{A_{j+1}}$ 
and this concludes the proof.
\end{proof}

We begin with the case where the two filtrations reach $X$.

\begin{prop}\label{lem:JHleftfinite}
Let $X$ be a stable for the filtrations object in a strictly $\bZ$-graded abelian category.
If $X$ admits two $\bZ$-composition series of the following form
\begin{align*}
X &\leftarrow A_1 \leftarrow A_2 \leftarrow \dots \leftarrow 0, \\
X &\leftarrow B_1 \leftarrow B_2 \leftarrow \dots \leftarrow 0,
\end{align*}
with respective multiplicities $k_{j,p}$, $k'_{j,p}$ and simple quotients $\{S_j\}_{j\in J}, \{S_j'\}_{j \in J'}$, then for each $s$, $S_s \cong S'_{j_s}\{\alpha_s\}$ for some $j_s,\alpha_s$ and $k_{s,p} = k'_{j_s, p + \alpha_s}$ for all $p \in \bZ$. In other words, the $\bZ$-composition series have the same quotients and multiplicities.
\end{prop}

\begin{proof}
Fix $s$ and $p$. The finiteness condition implies that we can reach all $A_i$ such that $A_i/A_{i+1} \cong S_s\{p\}$ in a finite number of steps. So we can take some minimal sub-filtration
\[
X \leftarrow A_1 \leftarrow A_2 \leftarrow \dots \leftarrow A_n
\]
containing all simple quotients $S_s\{p\}$. 
We prove by induction on $n$ that $k'_{j_s, p + \alpha_s} \ge k_{s,p}$ for some $j_s, \alpha_s$ such that $S_s \cong S'_{j_s}\{\alpha_s\}$, and by symmetry of the argument we get the equality.

\smallskip

By the hypothesis on $X$ there exists some $r$ such that $B_r \subset A_1$. We can suppose $r$ minimal, such that $B_{r-1} \not\subset A_1$. 
Consider the filtration
\begin{align*}
A_1 &\leftarrow A_1 \cap B_1 \leftarrow A_1 \cap B_2 \leftarrow \dots \leftarrow A_1 \cap B_{r-1} \leftarrow B_r \leftarrow B_{r+1} \leftarrow \dots \leftarrow 0.
\end{align*}
We claim that $B_{r-1}/B_r \cong X/A_1$ and that the filtration is in fact a $\bZ$-composition series with quotients given by 
\[X/B_1, B_1/B_2, \dots, B_{r-2}/B_{r-1}, 0, B_r/B_{r+1}, \dots \ ,\]
 and thus can be rewritten as
\begin{align*}
A_1 &\leftarrow A_1 \cap B_1 \leftarrow A_1 \cap B_2 \leftarrow \dots \leftarrow A_1 \cap B_{r-1} \leftarrow B_{r+1} \leftarrow \dots  \leftarrow 0.
\end{align*}
Since this filtration has the same quotients as the one given by the $B_i$ at the exception of $B_{r}/B_{r+1} \cong X/A_1$, we can apply the recursion on $X' = A_1$ with the filtration above together with
\[
A_1 \leftarrow \dots \leftarrow A_n \leftarrow \dots  \leftarrow 0.
\]

We now prove our claim. First observe that if $A_1 \cap B_{r-1} \ne B_r$ then $B_r \subset A_1 \cap B_{r-1}  \subset B_{r-1}$ are strict inclusions and thus $B_{r-1}/B_{r}$ would not be simple, which is absurd. So we get
\[
\frac{A_1 \cap B_{r-1}}{B_r} = 0.
\]
Now by the lemma above, since we can suppose $A_1 \ne B_1$, we have
\[
\frac{A_1}{A_1 \cap B_1} \cong \frac{X}{B_1},
\]
which is simple. Again, if $A_1 \cap B_1 \ne B_2$ (if not, then $r = 2$ and we are finished) the lemma gives
\[
\frac{A_1 \cap B_1}{A_1 \cap B_2}  = \frac{A_1 \cap B_1}{(A_1 \cap B_1) \cap B_2} \cong \frac{B_1}{B_2}. 
\]
Suppose now that $\frac{B_{i-1}}{A_1  \cap B_{i-1}} \cong X/A_1$ is simple. We have for $i < r$
\[
\frac{B_i}{A_1 \cap B_i} = \frac{B_i}{ (A_1 \cap B_{i-1}) \cap B_i} \cong \frac{B_{i-1}}{A_1 \cap B_{i-1}},
\]
and thus it is simple. Then in particular, since $B_r = A_1 \cap B_{r-1}$, we have
\[
\frac{B_{r-1}}{B_r} = \frac{B_{r-1}}{A_1 \cap B_{r-1}} \cong \frac{X}{A_1}.
\]
Moreover, in general for $i < r-1$ we have
\[
\frac{A_1 \cap B_i}{A_1 \cap B_{i+1}} = \frac{A_1 \cap B_i}{(A_1 \cap B_i ) \cap B_{i+1}} \cong \frac{B_i}{B_{i+1}}.
\]
This finishes the proof.
\end{proof}

The case with the filtrations reaching $0$ is similar.

\begin{lem}\label{lem:isomod2}
Suppose there is an object $X$ with two simple subobjects $M$ and $N$ such that $M \neq N$, then
\[
(M \cup N)/M \cong N.
\]
\end{lem}

\begin{proof}
$M \cap N = 0$ and thus $M \cup N = M \bigoplus N$.
\end{proof}

\begin{prop}\label{lem:JHrightfinite}
Let $X$ be a stable for the filtrations object in a strictly $\bZ$-graded abelian category.
If $X$ admits two $\bZ$-composition series of the following form
\begin{align*}
X &\leftarrow \dots \leftarrow A_{-2} \leftarrow A_{-1} \leftarrow 0, \\
X &\leftarrow  \dots \leftarrow B_{-2} \leftarrow B_{-1} \leftarrow 0,
\end{align*}
with respective multiplicities $k_{j,p}$, $k'_{j,p}$ and simple quotients $\{S_j\}, \{S_j'\}$, then the $\bZ$-composition series have the same quotients and multiplicities.
\end{prop}

\begin{proof}
The argument is similar to the one from Proposition~\ref{lem:JHleftfinite}.
\end{proof}

Finally, the case with one filtration reaching $X$ and the other $0$ follows easily from Lemma~\ref{lem:interfilt}.

\begin{prop}\label{lem:JHleftrightfinite}
Let $X$ be a stable for the filtrations object in a strictly $\bZ$-graded abelian category.
If $X$ admits two $\bZ$-composition series 
\begin{align*}
X &\leftarrow A_1 \leftarrow A_2 \leftarrow \dots \leftarrow 0, \\
X &\leftarrow  \dots \leftarrow B_{-2} \leftarrow B_{-1} \leftarrow 0,
\end{align*}
with respective multiplicities $k_{j,p}$, $k'_{j,p}$ and simple quotients $\{S_j\}, \{S_j'\}$, then they have the same quotients and multiplicities.
\end{prop}

\begin{proof}
We prove by induction that all quotients from the first filtration appear as quotients of the second. A similar reasoning shows the converse.

\textbf{\underline{Case 1:}} Suppose $B_i \subset A_1$ for some $i \in \bZ$. Take $i$ minimal such that $B_{i-1} \not\subset A_1$. Then $B_{i-1} \cap A_1 \cong B_i$ and
\[
\frac{B_{i-1}}{B_i} \cong \frac{B_{i-1}}{B_{i-1} \cap A_1} \cong \frac{B_{i-1} \cup A_1}{A_1} \cong \frac{X}{A_1}.
\]
We now apply the proof on
\begin{align*}
A_1 &\leftarrow A_2 \leftarrow A_3 \leftarrow \dots \leftarrow 0, \\
X \cap A_1 &\leftarrow  \dots \leftarrow B_{i-2} \cap A_1 \leftarrow B_{i-1} \cap A_1 \cong B_i \leftarrow  \dots \leftarrow B_{-2}  \leftarrow B_{-1} \leftarrow 0.
\end{align*}
All quotients but $B_{i-1}/B_i$ appears in this filtration since for all $j < i$ we have $B_{j-1} \cap A_1 \not\subset B_{j}$. 

\textbf{\underline{Case 2:}} Suppose $B_{-1} \not\subset A_1$. Then $X = B_{-1} \bigoplus A_1$ and $X/A_1 \cong B_{-1}$. We then apply the argument recursively on
\begin{align*}
A_1 &\leftarrow A_2 \leftarrow \dots \leftarrow 0, \\
A_1 \cong X/B_{-1}  &\leftarrow \dots \leftarrow B_{-2}/B_{-1} \leftarrow 0,
\end{align*}
and this concludes the proof.
\end{proof}

Now we introduce some tools and we prove that given a subobject $M$ of $X$ admitting a $\bZ$-composition series, then the quotients in the filtration of $M$ and $X/M$ are essentially the same as the ones in the filtration of $X$.

\begin{lem}
Let $M \subset X$. Suppose we have $\bZ$-composition series
\begin{align*}
X \leftarrow X_1 \leftarrow X_2 \leftarrow \dots \leftarrow 0, \\
M \leftarrow M_1 \leftarrow M_2 \leftarrow \dots \leftarrow 0.
\end{align*}
Then we get a $\bZ$-composition series
\[
X/M \leftarrow \frac{X_1 \cup M}{M} \leftarrow \frac{X_2 \cup M}{M} \leftarrow \dots \leftarrow 0.
\]
\end{lem}

\begin{proof}
First observe that for all $i$, thanks to the third isomorphism theorem, we have
\[
\frac{(X_i \cup M)/M}{(X_{i+1} \cup M)/M} \cong \frac{X_i \cup M}{X_{i+1} \cup M}.
\]
Then we get
\[
\frac{X_i \cup M}{X_{i+1} \cup M} \cong \frac{X_i \cup (X_{i+1} \cup M)}{X_{i+1} \cup M} \cong \frac{X_i}{X_i \cap (X_{i+1} \cup M).}
\]
If $X_i \subset X_{i+1} \cup M$, it is $0$. If not, we have $X_{i+1} \subset X_i \cap (X_{i+1} \cup M) \subsetneq X_i$ and thus $
\frac{X_i \cup M}{X_{i+1} \cup M} \cong  \frac{X_i}{X_{i+1}}.$
\end{proof}

\begin{rem}\label{rem:quottofilt}
Also note that if $M$ is a subobject of $X$ such that there are $\bZ$-composition series
\begin{align*}
M &\leftarrow M_1 \leftarrow \dots \leftarrow 0, \\
X/M &\leftarrow Z_1 \leftarrow \dots \leftarrow 0,
\end{align*}
then there is a filtration with simple quotients
\[
X \leftarrow X_1 \leftarrow \dots \leftarrow M \leftarrow M_1 \leftarrow \dots \leftarrow 0.
\]

\begin{proof}
Define $X_i$ as the pull-back
\[
\xymatrix{
X_i \ar[r] \ar[d]_{\imath'} & Z_i \ar[d]^{\imath} \\
X \ar[r] & X/M
}
\]
where $\imath$ is a monomorphism and thus so is $\imath'$. We clearly have $
\varprojlim_{i} X_i = M
$. Moreover the following diagram has 3 exact column and 2 exact rows 
\[
\xymatrix{
M \ar[r] \ar[d]_{\cong} & X_{i+1} \ar[r]\ar[d] & Z_{i+1}\ar[d] \\
M \ar[r]\ar[d] & X_{i}\ar[d]\ar[r] & Z_i\ar[d] \\
0 \ar[r] & X_{i}/X_{i+1} \ar[r] &Z_i/Z_{i+1}
}
\]
and by the $3\times 3$ lemma the third row must be exact too.
This means
$
\frac{X_i}{X_{i+1}} \cong \frac{Z_i}{Z_{i+1}}
$.
\end{proof}

\end{rem}

\begin{prop}\label{prop:subseries}
Let $M \subset X$. Suppose $X, M$ and $X/M$ are stable for the filtrations. If we have $\bZ$-composition series
\begin{align}
X &\leftarrow X_1 \leftarrow X_2 \leftarrow \dots \leftarrow 0, \label{eq:Xfilt} \\
M &\leftarrow M_1 \leftarrow M_2 \leftarrow \dots \leftarrow 0,  \label{eq:Mfilt} \\
X/M &\leftarrow \frac{X_1 \cup M}{M} \leftarrow \frac{X_2 \cup M}{M} \leftarrow \dots \leftarrow 0,  \label{eq:XMfilt}
\end{align}
then all simple quotients of~(\ref{eq:Xfilt}) appears has quotients with the same multiplicities in~(\ref{eq:Mfilt}) plus~(\ref{eq:XMfilt}).
\end{prop}

\begin{proof}
Consider the filtrations
\begin{align}
M = M \cap X &\leftarrow M \cap X_1 \leftarrow M \cap X_2 \leftarrow \dots \leftarrow 0, \label{eq:Mbisfilt} \\
X/M &\leftarrow \frac{X_1 \cup M}{M} \leftarrow \frac{X_2 \cup M}{M} \leftarrow \dots \leftarrow 0. \label{eq:XMbisfilt}
\end{align}
We claim that together they contain exactly all the quotients of~(\ref{eq:Xfilt}) with the same multiplicities. Indeed, note that $X_j \cap (X_{j+1} \cup M) \cong X_{j+1} \cup (X_j \cap M)$ such that we must have $X_j \cong X_j \cap (X_{j+1} \cup M)$ or $X_{j+1} \cong X_j \cap (X_{j+1} \cup M)$ and thus 
\begin{align*}
\frac{X_j \cap M}{X_{j+1} \cap M} &\cong \frac{X_j}{X_{j+1}} &&\text{or} & \frac{X_j \cap M}{X_{j+1} \cap M} &= 0 \\
\frac{X_j \cup M}{X_{j+1} \cup M} &\cong 0 && & \frac{X_j \cup M}{X_{j+1} \cup M} &= \frac{X_j}{X_{j+1}}.
\end{align*}
Then we conclude the proof using the Proposition~\ref{lem:JHleftfinite} on~(\ref{eq:Mfilt}) with~(\ref{eq:Mbisfilt}), and on~(\ref{eq:XMfilt}) with~(\ref{eq:XMbisfilt}).
\end{proof}
All the above can also be proved for all combinations of the two types of $\bZ$-composition series,
using similar arguments.

\begin{rem}Restricting the filtrations to only $\bZ$-filtrations is too strong for what we want to do. Indeed, suppose we have a non-split short exact sequence 
\[
A \rightarrow B \rightarrow \bigoplus_{i \ge 0} X\brak{2i}
\]
where $X$ and $A$ are simple, then using Remark~\ref{rem:quottofilt} we can construct a filtration
\[
B \leftarrow B_1 \leftarrow B_2 \leftarrow \dots \leftarrow A \leftarrow 0
\]
which is not a $\bZ$-composition series, despite the fact that $B_i/B_{i+1} \cong X\brak{2i}$ and $\varprojlim_{i} B_i \cong A$. Thus we want to be able to glue $\bZ$-filtrations together.
\end{rem}

Write $\xymatrix{X &\ar@{-->}[l] A}$ if there exists a $\bZ$-filtration 
\begin{align*}
X &\leftarrow X_1 \leftarrow \dots \leftarrow A, &&\text{or}& X &\leftarrow \dots \leftarrow X_{-1} \leftarrow A,
\end{align*}
inducing a $\bZ$-composition series on $X/A$.

\begin{defn}
 We say that $X$ has locally finitely many composition factors if there is a finite filtration 
\[
\xymatrix{X = X^0 &\ar@{-->}[l] X^1 &\ar@{-->}[l] X^2  &\ar@{-->}[l] \dots  &\ar@{-->}[l] X^n=0},
\] 
and each $X^i/X^{i+1}$ admits a $\bZ$-composition series. We can consider the set $\{S_j\}_{j \in J}$ of all quotients of the filtrations for all the $X^i/X^{i+1}$, up to isomorphism and grading shift, which we call composition factors of $X$. Each of the composition factor has a finite degree $p$ multiplicity for each $p \in \bZ$, which is given by the sum of the multiplicities in the $\bZ$-composition series of all the quotients $X^i/X^{i+1}$.
Moreover, we say that $X$ is stable for the filtrations if for all such filtrations, the quotients $X^i/X^{i+1}$ are stable for the filtrations.
\end{defn}

We have now all the tools to prove the main result in this subsection.

\begin{thm}\label{thm:uniquefiltJH}
Let $X$ be a stable for the filtrations object having locally finitely many composition factors. If there are two filtrations
\begin{align}
\xymatrix{X = A^0 &\ar@{-->}[l] A^1 &\ar@{-->}[l] A^2  &\ar@{-->}[l] \dots  &\ar@{-->}[l] A^n=0}, \label{eq:filtA} \\
\xymatrix{X = B^0 &\ar@{-->}[l] B^1 &\ar@{-->}[l] B^2  &\ar@{-->}[l] \dots  &\ar@{-->}[l] B^m=0}, \label{eq:filtB}
\end{align}
with all $A^i/A^{i+1}$ and $B^i/B^{i+1}$ having $\bZ$-composition series,  then the composition factors are in bijection and have the same multiplicities.
\end{thm}

\begin{proof}
We can suppose $m \ge n$. The proof follows by a double induction on $n$ and $m$. If $n = 1$ and $m = 1$, then the result is given by Propositions~\ref{lem:JHleftfinite} ,~\ref{lem:JHrightfinite} or~\ref{lem:JHleftrightfinite}. If $n = 1$, then by Lemma~\ref{lem:interfilt} we can construct
\[
\xymatrix{X &\ar@{-->}[l]  B^1 &\ar@{-->}[l]  0},
\]
which has the same composition factors as $\xymatrix{X &\ar@{-->}[l] 0}$ thanks to Proposition~\ref{prop:subseries}. Then we can split this filtration and~(\ref{eq:filtB}) into two parts  at the level of $B^1$, on which we can apply the argument recursively. 
Suppose now $n > 1$. We can construct 
\begin{align}
\xymatrix{X &\ar@{-->}[l]  A^1 &\ar@{-->}[l]  A^1 \cap B^1  &\ar@{-->}[l]  A^2 \cap B^1 &\ar@{-->}[l] \dots &\ar@{-->}[l]  A^n \cap B^1 = 0}, \label{eq:filtA1} \\
\xymatrix{X &\ar@{-->}[l]  B^1 &\ar@{-->}[l]  A^1 \cap B^1  &\ar@{-->}[l]  A^2 \cap B^1 &\ar@{-->}[l] \dots &\ar@{-->}[l]  A^n \cap B^1 = 0}.\label{eq:filtB1} 
\end{align}
By splitting~(\ref{eq:filtA1}) at the level of $A^1$ we can apply the reasoning recursively to prove it has the same composition factors as~(\ref{eq:filtA}). It is important to note that~(\ref{eq:filtA1}) and~(\ref{eq:filtB1}) have the same tail after $A^1 \cap B^1$ and it is given by $n$ which is smaller than $m$. This allows us to use the induction hypothesis.
 By the same arguments,  (\ref{eq:filtB1}) has the same composition factors as~(\ref{eq:filtB}). Now by splitting~(\ref{eq:filtA1}) and~(\ref{eq:filtB1}) at the level of $A^1 \cap B^1$ and using the induction hypothesis, these two must have the same composition factors as well. This concludes the proof.
\end{proof}

\begin{defn}
We say that a strictly $\bZ$-graded abelian category has \emph{local Jordan--H\"older property} if every object has locally finitely many composition factors and is stable for the filtrations.
\end{defn}

Let $\cC$ be a strictly $\bZ$-graded, right complete, locally additive category with the local Jordan--H\"older property and every $\bZ$-composition series being left bounded.
Following the same path as in the previous subsection, define $G_0'(\cC)$ to be the free $\bZ\llbracket q \rrbracket [q^{-1}]$-module generated by the classes of simple objects, up to shift. 
By Theorem~\ref{thm:uniquefiltJH} there is a surjective $\bZ\llbracket q \rrbracket [q^{-1}]$-module map
\[
g : G_0(\cC) \twoheadrightarrow G'_0(\cC),
\]
inducing a topology on $G_0(\cC)$.
We make it Hausdorff by the following.
\begin{defn}
We define the topological Grothendieck group of $\cC$ as
\[
\boldsymbol G_0(\cC) = G_0(\cC)/\cap_{n \ge 0} g^{-1}((q)^n).
\]
\end{defn}
Again, the topological Grothendieck group 
forms a topological module over $\bZ\llbracket q \rrbracket [q^{-1}]$ with the $(q)$-adic topology and we get the following theorem.

\begin{thm}
The topological Grothendieck group $\boldsymbol G_0(\cC)$ of a right complete category $\cC$, with the local Jordan--H\"older property and every $\bZ$-composition series being left bounded, is a topological $\bZ\llbracket q\rrbracket[q^{-1}]$-module freely generated by the classes of simple objects (up to degree shift).
\end{thm}

An exact, graded functor $\Phi\colon \cC \rightarrow \cC'$ gives rise to a $\bZ\llbracket q \rrbracket [q^{-1}]$-linear map $G_0(\cC) \rightarrow G_0(\cC')$. 
Also, there is an obvious $\bZ\llbracket q \rrbracket [q^{-1}]$-linear injection $\boldsymbol G_0(\cC) \hookrightarrow G_0(\cC)$.
We define $[\Phi] : \boldsymbol G_0(\cC) \rightarrow \boldsymbol G_0(\cC')$ as the composition $\boldsymbol G_0(\cC) \hookrightarrow  G_0(\cC) \rightarrow G_0(\cC') \twoheadrightarrow \boldsymbol G_0(\cC')$.

\begin{exe}
Let $R$ be a $\Bbbk$-algebra as in the Example~\ref{ex:Rmod}. Consider the category $R\lfmod$ of locally finite dimensional $R$-modules, with the dimensions left bounded and degree $0$ morphisms. It has the local Jordan--H\"older property and every $\bZ$-composition series is left bounded. Moreover we have the following inclusions of full subcategories
\[
R\prmod \subset R\lfmod \subset R\fgmod.
\]
\end{exe}

Moreover, if the category has enough projectives, then projective resolutions of the simple objects can give a change of basis, such that the topological Grothendieck group becomes also freely generated by the projective objects.

\begin{exe}
Take $R = \bQ[x,y]/{(y^2)}$, with $\deg(x)=\deg(y)=2$, and its topological Grothendieck group $\boldsymbol G_0(R\lfmod)$ is generated either by the classes of the simple object $S = \bQ$ or the projective object~$R$, with change of basis given by
\begin{align*}
[R] &= \frac{1+q^2}{1-q^2}[S] = (1+q^2)(1+q^2+q^4+\dots)[S], \\
 [S] &= \frac{1-q^2}{1+q^2}[R] = (1-q^2)(1-q^2+q^4- \dots)[R].
\end{align*}
\end{exe}

\begin{rem}
Likely it is possible to adapt the results in~\cite{as}, weakening the artinian assumption to locally Jordan--H\"older and mixing it with our results. The local Jordan--H\"older property gives for each object a unique weight filtration, bounded from above (the weight is the opposite of the $\bZ$-grading, thus this a $\bZ$-filtration bounded from below), with all quotients being finite. This should be usable to compute the topological Grothendieck group of $\cD^{\nabla}(\cC)$ and get a continuous isomorphism $\boldsymbol K(\cD^{\nabla}(\cC)) \cong \boldsymbol G_0(\cC)$.
\end{rem}

\subsection{Multigrading and field of formal Laurent series}\label{ssec:laurent}

We now investigate the case of multigrading. But first we need to chose a construction of the field of formal Laurent series $\bQ\pp{x_1, \dots, x_p}$. 

\subsubsection{Field of formal Laurent series}
We follow the description given in \cite{laurent}. We fix a grading by $\bZ^p$ with $p \in \bN$, and we choose an additive order $\prec$ on it. That is, $a \prec b$ implies $a + c \prec b + c$, for every $a,b,c \in \bZ^p$.

\begin{defn}
We call \emph{cone} a subset $C \subset \bR^p$ such that
\[
C = \{ \alpha_1 v_1 + \dots + \alpha_p v_p | \alpha_1, \dots, \alpha_p \ge 0 \},
\]
for some generating elements $v_1, \dots, v_p \in \bZ^p$. Moreover we say $C$ is \emph{compatible with the order $\prec$} if $0 \prec v_i$ for all $i \in \{1, \dots, p\}$.
\end{defn}

\begin{rem}
Usually the definition of cone is more general, but we are interested only in these ones for our discussion.
\end{rem}

\begin{exe}
If $p =1$, then there are only two possible orders:
\begin{itemize}
\item if $0 \prec 1$, then there is only one (non-zero) compatible cone given by $[0, \infty[$,
\item if $0 \prec -1$, then the only compatible cone is $]-\infty, 0]$.
\end{itemize}
\end{exe}

Let $C \subset \bR^p$ be a cone compatible with $\prec$ and define
\[
\bQ_C\llbracket x_1, \dots, x_p \rrbracket = \left\{ \sum_{\boldsymbol{k} = (k_1, \dots, k_p) \in \bN^p} a_{\boldsymbol{k}} x_1^{k_1} \dots x_p^{k_p} | a_{\boldsymbol{k}} = 0 \text{ if } \boldsymbol{k} \notin C \right\}.
\]

\begin{prop}\emph{(\cite[Theorem 10]{laurent})}
The set $\bQ_C\llbracket x_1, \dots, x_p \rrbracket$ together with the natural addition and multiplication forms a ring.
\end{prop}

\begin{proof}\emph{(Sketch)}
The important point in the proof of this proposition is that the restriction to cones compatible with the order ensures we only have to multiply a finite number of elements to define each coefficient in a product.
\end{proof}

The next definition is \cite[Definition~14]{laurent}.

\begin{defn}
We put 
\[
\bQ_\prec \llbracket x_1, \dots, x_p\rrbracket = \bigcup_{C} \bQ_C\llbracket x_1, \dots, x_p \rrbracket,
\]
where the union is over all cones compatibles with $\prec$, and we define the field of formal Laurent series as
\[
\bQ_\prec \pp{x_1, \dots, x_p} = \bigcup_{\boldsymbol e = (e_1, \dots, e_p) \in \bZ^p} x_1^{e_1} \dots x_p^{e_p} \bQ_\prec \llbracket x_1, \dots, x_p\rrbracket.
\]
\end{defn}

\begin{thm}\emph{(\cite[Theorem 15]{laurent})}
The set $\bQ_\prec \pp{x_1, \dots, x_p}$ together with the natural addition, multiplication and division forms a field.
\end{thm}

\begin{proof}\emph{(Sketch)}
There are three main ideas in the proof of this theorem. First, given any pair of cones $C_1, C_2$ compatible with $\prec$, their sum yields a cone $C_3 = C_1 + C_2$, also compatible with $\prec$. Hence we can define a product $ \bQ_{C_1}\llbracket x_1, \dots, x_p \rrbracket \otimes  \bQ_{C_2}\llbracket x_1, \dots, x_p \rrbracket \rightarrow  \bQ_{C_3}\llbracket x_1, \dots, x_p \rrbracket$, which in turns define a product on $\bQ_\prec \llbracket x_1, \dots, x_p\rrbracket$. Secondly, given $f(\boldsymbol x) \in\bQ_C\llbracket x_1, \dots, x_p \rrbracket $ such that $f(0) \neq 0$, one can define recursively a unique $g(\boldsymbol x) \in \bQ_C\llbracket x_1, \dots, x_p \rrbracket $ such that $g(x)f(x) = 1$. Finally taking the union of all $\bQ_\prec \llbracket x_1, \dots, x_p\rrbracket$ shifted by a monomial allows to write any $f(\boldsymbol x) \in \bQ_\prec \pp{x_1, \dots, x_p}$ as $\boldsymbol x^{\boldsymbol e} h(x)$ where $h( \boldsymbol x) \in \bQ_C\llbracket x_1, \dots, x_p \rrbracket $ is such that $h(0) \neq 0$. Therefore, we have $\boldsymbol x ^{-\boldsymbol e}h^{-1}(\boldsymbol x) \boldsymbol x ^{\boldsymbol e}h(\boldsymbol x) =1$, which concludes the proof.
\end{proof}
%

\begin{exe}
Again, if $p=1$, we have two possible ways to construct $\bQ\pp{x}$:
\begin{itemize}
\item if $0 \prec 1$, then we get $\bQ_\prec \pp{x} = \bQ\llbracket x \rrbracket [x^{-1}]$ and $\frac{1}{x-x^{-1}} = -x (1+x^2+\dots)$,
\item if $0 \prec -1$, then $\bQ_\prec \pp{x} = \bQ\llbracket x^{-1} \rrbracket [x]$ and $\frac{1}{x-x^{-1}} = x^{-1}(1+x^{-2}+\dots)$.
\end{itemize}
\end{exe}

\begin{exe}\label{ex:laurent}
Take $p=2$, then we have 6 ways to construct $\bQ\pp{x_1, x_2}$. In this case we abuse the notation and say, for example, that we choose the order $0 \prec x_1 \prec x_2$ for the order induced by the choice $(0,0) \prec (1,0) \prec (0,1)$, where we suppose $(1,0)$ corresponds to $x_1$ and $(0,1)$ to $x_2$. Then we get 
\[
\frac{1}{1-x_1^{-2}x_2^2} = (1 + x_1^{-2} x_2^2 + x_1^{-4} x_2^4 + \dots).
\]
\end{exe}

%

\subsubsection{Grothendieck groups for multigrading}

We fix a multigrading and an additive order on it. 
Every definition in \S\ref{ssec:completedK0} and \S\ref{ssec:completedG0} extends naturally to the multigraded case, except for left-bounded.

\begin{defn}
We say that a locally finite coproduct (or filtration) is \emph{cone bounded} if all its non-zero coefficients are contained in some cone compatible with $\prec$, up to a shift.
\end{defn}

It is then straightforward to adapt all results  from \S\ref{ssec:completedK0} and \S\ref{ssec:completedG0} to the multigraded case, replacing ``left-bounded'' by ``cone bounded''. In accordance to this denomination, we will also say \emph{cone complete} for a category which admits all cone bounded, locally finite coproducts.

\smallskip

The next example represents the classical case that will be used later on.

\begin{exe}\label{ex:multigradedR}
As in Example~\ref{ex:Ralgebrafilt}, we can construct filtrations for some multigraded $\Bbbk$-algebras. Suppose $R$ is a unital $\bZ^p$-graded $\Bbbk$-algebra, having locally finite dimension. Also suppose its graded dimension is contained in a cone compatible with the chosen order on $\bZ^p$, and suppose it is isomorphic to $\Bbbk$ in degree $0$. Then the additive order on $\bZ^p$ restricts to a total order $0 = i_0 \prec i_1 \prec \dots$ on the homogeneous components of $R = \bigoplus_{k} R^{i_k}$, which allows us to construct a filtration of submodules
\[
R = \bigoplus_{i \succeq i_0 } R^i \leftarrow \bigoplus_{i \succeq i_1 } R^i \leftarrow \dots \leftarrow 0,
\]
where each subquotient is an homogeneous component $R^{i_j}$, and thus finite dimensional. Then we can apply the same arguments as in Example~\ref{ex:stablefilts} to show locally finite, cone bounded (i.e. its graded dimension is contained in a cone compatible with $\prec$, up to a shift), left $R$-modules are stable for the filtrations, and thus $R\lfmod$ has the local Jordan--H\"older property. Therefore $\boldsymbol G_0(R\lfmod)$ is a free $\bZ_\prec\pp{x_1, \dots, x_p}$-module generated by the classes of simple modules.
\end{exe}

The hypothesis in the example above can be weakened a bit by only requiring $R$ to admits a finite collection of indecomposable projective modules, up to isomorphism and shift, each having their dimension locally finite and contained in a cone compatible with $\prec$. 

%
%
\section{Categorification of the Verma module $M(\lambda)$}\label{sec:vermacat}

Following the explanations in the preceding section, in order to define $\bQ\pp{q,\lambda}$ we choose the additive order on $\bZ^2$ given by $0 \prec q \prec \lambda$, with the abuse of notation explained in Example~\ref{ex:laurent}.

\subsection{Categories of modules}\label{ssec:polymod}

Each of the superrings $\Omega_k$ is a noetherian $\bZ\times\bZ$-graded local superring
whose degree $(0,0)$ part is isomorphic to $\bQ$. 
Then every graded projective supermodule (not necessarily finitely-generated) is a free  
graded supermodule~\cite[Prop.~1.5.15]{bz}, and $\Omega_k$ has (up to isomorphism and grading shift) a unique 
graded indecomposable projective supermodule. 

\smallskip

Let  $\Omega_k\lfmods$ be the abelian $\bZ\times\bZ$-graded
category of locally finite-dimensional, cone bounded $\Omega_k$-supermo\-du\-les, together with the grading preserving supermodule maps.
These are graded supermodules which are finite-dimensional in each degree. Explicitly a bigraded supermodule $M=\oplus_{i,j}M_{i,j}$ is cone bounded if there exist a cone $C \subset \bR^2$  compatible with the fixed order $\prec$ and $m,n \in  \bZ$, such that $M_{i+m,j+n} = 0$ whenever $(i,j) \notin C$. In other words, it is cone bounded if its graded dimension is in $\bQ_\prec \pp{q,\lambda}$.

\smallskip

Every graded projective supermodule $P$ of $\Omega_k$ is of the form 
$P\cong \Omega_k\otimes A$ where $A$ is a graded abelian group. 
The superring $\Omega_k$ has (up to isomorphism and grading shift) a unique 
simple supermodule $S_k=\Omega_k/(\Omega_k)_+$ (here $(\Omega_k)_+$ 
denotes the submodule of $\Omega_k$ generated by the elements of nonzero bidegree).
Every cone bounded graded $\Omega_k$-supermodule has a projective cover~\cite[Thm. 2]{cartan}. 
As a matter of fact, it is not hard to construct such a projective cover. 
For cone bounded $\Omega_k$-supermodule $M$  form the non-trivial graded abelian group 
$M/( (\Omega_k)_+M )\cong\bQ\otimes_{\Omega_k}M$ and form 
$P=\Omega_k\otimes \bQ\otimes_{\Omega_k}M$. 
Then $P$ is a projective cover of $M$ with the surjection $p\colon P\to M$ given by 
$a\otimes b\otimes m\mapsto a \sigma(b\otimes m)$ where $\sigma$ is a section of the canonical projection 
$M\to M/( (\Omega_k)_+M )\cong\bQ\otimes_{\Omega_k}M$. 

\smallskip
%
%

The graded dimension of $\Omega_k$ is contained in the cone in $\bR^2$ generated by $(2,0)$ and $(-2k,2)$. Hence it is a special case of Example~\ref{ex:multigradedR} (with some minor adjustments for the ``super''), thus $\Omega_k\lfmods$ possesses the local Jordan--H\"older property and we get the following:

\begin{prop}\label{prop:1dimGhat0}
The topological Grothendieck group $\boldsymbol G_0(\Omega_k\lfmods)$ is a one dimensional  module over the ring
  $\bZ_\pi\pp{q,\lambda}$ with the $(q,\lambda)$-adic topology, 
freely generated by either the class of the simple object $S_k$ or the projective object $\Omega_k$.
\end{prop}

Consider the full subcategory of $\Omega_k\lfmods$ generated by modules with $\lambda$-grading bounded above and below.
We write it $\Omega_k\lfmodsl$ and it possesses the local Jordan--H\"older property, with a topological Grothendieck group having $(q)$-adic topology. We get the following proposition.

\begin{prop}\label{prop:1dimG0}
  The topological Grothendieck group $\boldsymbol G_0(\Omega_k\lfmodsl)$ is a one dimensional topological module over the ring
  $\bZ_\pi\llbracket q\rrbracket[\lambda^{\pm 1}, q^{-1}]$ with the $(q)$-adic topology, 
freely generated by the class of the simple object $S_k$. 
\end{prop}

Now consider the full subcategory $\Omega_k\prmodsl \subset \Omega_k\lfmodsl$ consisting of locally finitely generated, cone bounded projective modules. For the $q$-grading, it is a cone complete, locally Krull--Schmidt category, and we get the following.

\begin{prop}\label{prop:1dimK0}
  The topological split Grothendieck group $\boldsymbol K_0(\Omega_k\prmodsl)$ is a one dimensional module over the ring
  $\bZ_\pi\llbracket q\rrbracket[\lambda^{\pm 1}, q^{-1}]$ with the $(q)$-adic topology, 
freely generated by the class of the projective object $\Omega_k$.
\end{prop}

\subsection{The Verma categorification}\label{ssec:vermacat}

Set $\cM_k=\Omega_k\fgmods$ and $\cM_{k+1,k}=\Omega_{k+1,k}\fgmods$ and for $k\geq 0$ consider the functors 
\begin{align*}
\Ind_k^{k+1,k} &\colon \cM_{k} \to \cM_{k+1,k}, & \Res_{k}^{k+1,k} &\colon \cM_{k+1,k} \to \cM_{k},
\\
\Ind_{k+1}^{k+1,k} &\colon \cM_{k+1} \to \cM_{k+1,k}, & \Res_{k+1}^{k+1,k} &\colon \cM_{k+1,k} \to \cM_{k+1} .
\end{align*}
 
\n It is sometimes useful to arrange them using a diagram as below.
\[ 
\raisebox{25pt}{
\xymatrix@C=16mm@R=10mm
{
\cdots\;\ar@<-1.2ex>[dr] 
&& \cM_{k+1,k}\ar@<.6ex>[dl]^{\Res_{k+1}^{k+1,k}}\ar@<-.6ex>[dr]_{\Res_k^{k+1,k}}  
&& \;\cdots\ar@<.6ex>[dl] 
\\
&  \cM_{k+1} \ar[ul]\ar@<.6ex>[ur]^{\Ind_{k+1}^{k+1,k}} 
&& \cM_k \ar@<-.6ex>[ul]_{\Ind_k^{k+1,k}}\ar@<.6ex>[ur] 
}}
\]

Since $\Omega_{k+1,k}$ is sweet the functors  $\Ind_k^{k+1,k}$ 
and $\Ind_{k+1}^{k+1,k}$ are exact. 

\medskip

For each $k\geq 0$ define exact functors 
$\F_k \colon \cM_k \to \cM_{k+1}$  
and 
$\E_k\colon \cM_{k+1}\to \cM_{k}$ 
by 
\[
\F_k = \Res_{k+1}^{k,k+1}\circ\Ind_k^{k,k+1}\brak{-k,0} 
\mspace{40mu}\text{and}\mspace{40mu}
\E_k = \Res_{k}^{k,k+1}\circ\Ind_{k+1}^{k,k+1}\brak{k+2,-1} . 
\]
Using the language of bimodules, $\F_k$ and $\E_{k}$ can also be written as 
\begin{align*}
\F_k(-) &= \bigl(\Omega_{k+1} \otimes_{k+1} \Omega_{k+1,k}\otimes_{k}(-)\bigr) \brak{-k,0},
\intertext{and}
\E_k(-) &= \bigl(\Omega_k \otimes_k \Omega_{k,k+1}\otimes_{k+1}(-)\bigr)\brak{k+2,-1}, 
\end{align*}  
where 
$\Omega_{k+1,k}$ is seen as a $(\Omega_{k+1,k},\Omega_{k})$-superbimodule 
and $\Omega_{k,k+1}$ as a $(\Omega_k,\Omega_{k,k+1})$-superbimodule.

\begin{prop}\label{prop:FEadj}
Up to a grading shift the functors $(\F_k,\E_k)$ form an adjoint pair of functors. 
\end{prop}

\begin{proof}
The superbimodule maps $\eta$ and $\epsilon$ from Definitions~\ref{def:unit} and~\ref{def:counit} induce respectively natural transformations
$\mathbbm{1}_k\to \E_k\F_k$ and $\F_k\E_k\to\mathbbm{1}_{k+1}$ 
which are the unit and counit  of the adjunction $\F_k \dashv\E_k$ by Remark~\ref{rem:unitcounit}.
\end{proof}

The functor $\F$ does not admit $\E$ as a left adjoint. 
As explained in \S\ref{ssec:catactslt} this is necessary to
categorify infinite-dimensional highest weight $\slt$-modules.

\medskip 

We denote by $\Q_k$ the functor $\cM_k\to\cM_k$ 
of tensoring on the left with the shifted $(\Omega_k,\Omega_k)$-superbimodule $\omg_k^{\xi_{}}$. 
In this context
, Corollary~\ref{prop:sesbim} reads as follows. 
\begin{prop}\label{prop:EFses}
For each $k\in\bN_0$ we have an exact sequence 
\[
0 \xra{\quad} 
\F_{k-1}\circ\E_{k-1} 
\xra{\quad}
\E_k\circ\F_k  
\xra{\quad} 
\Q_k\brak{-2k-1,1}\ \oplus \Pi\Q_k\brak{2k+1,-1} 
\xra{\quad} 0 
\] 
of endofunctors on $\cM_k$ .
\end{prop}

Since the superbimodules used to construct $\F_k$ and $\E_k$ are sweet, see Proposition \ref{prop:sweet},
we have the following.
\begin{cor} 
For every $M\in\cM_k$ we have an isomorphism 
\[
\E_k\circ\F_k (M) \cong 
\F_{k-1}\circ\E_{k-1}(M) 
\,\oplus\,
\Q_k(M)\brak{-2k-1,1}\, \oplus\, \Pi\Q_k(M)\brak{2k+1,-1} .
\] 
\end{cor}

\medskip

Define the functor $\K_k$ as the endofunctor of $\cM_k$ which is the auto-equivalence 
that shifts the bidegree by $(-2k-1,1)$ 
$$\K_k\colon\cM_k\to\cM_k,\mspace{40mu}  \K_k(-)= (-)\brak{-2k-1,1 }.$$
We have isomorphisms $\K_k\circ\E_k\cong \E_k\circ\K_{k+1}\brak{2,0}$ and
$\K_{k+1}\circ\F_k\cong \F_k\circ\K_{k}\brak{-2,0}$.  
Moreover, since $\Pi\circ\Q_k\cong\Q_k\circ\Pi$ we have $\Pi\Q_k\circ\K_k\cong\Q_k\circ\Pi\K_k$.  

\begin{defn} 
Define the category $\cM$ and the endofunctors $\F$, $\E$, $\K$ and $\Q$, as 
\begin{equation*}
\cM = \bigoplus\limits_{k\geq 0}\cM_k ,
\mspace{60mu}
\E = \bigoplus\limits_{k\geq 0}\E_k ,
\mspace{60mu}
\F = \bigoplus\limits_{k\geq 0}\F_k ,
\mspace{60mu}
\K = \bigoplus\limits_{k\geq 0}\K_k ,  
\mspace{60mu}
\Q = \bigoplus\limits_{k\geq 0}\Q_k .  
\end{equation*}
\end{defn}

All the above adds up to the following.
\begin{thm}\label{thm:vermacat}
We have natural isomorphisms of functors
\[
\K\circ\K^{-1} \cong \id_{\cM} \cong \K^{-1}\circ\K , 
\]
\[
\K\circ\E \cong \E\circ\K\brak{2,0},\qquad \K\circ\F\cong \F\circ\K\brak{-2,0} , 
\]
and an exact sequence 
\[
0\xra{\quad} 
\F\circ\E
\xra{\quad} 
\E\circ\F 
\xra{\quad} 
\Q\circ \bigl( \K \oplus\ \Pi\K^{-1} \bigr) 
\xra{\quad} 0 .
\]
\end{thm}

\begin{rem}
Theorem~\ref{thm:vermacat} is suggestive from the point of view of categorification of the deformed version $\dot{U}_\lambda$ 
of quantum $\slt$ from \S\ref{ssec:defidemp} if
we identify $\cM_k$ with the $(\lambda q^{-1-2k})$-th weight space. 
\end{rem}

\subsubsection{NilHecke action}\label{sssec:VnilHecke}

In \S\ref{ssec:nilHecke} we have constructed an action of the nilHecke algebra $\nh_n$ on 
the superbimodules $\Omega_{k,k+n}$ and $\Omega_{k+n,k}$.  
The definition of $\F_k$ and $\E_k$ imply the following. 
\begin{prop}
The nilHecke algebra action on $\Omega_{k,k+n}$ descends to an action on $\F^n$ 
and on $\E^n$. 
\end{prop}

\subsubsection{Grothendieck groups}\label{sssec:groth}

For the sake of simplicity we write $K_0(\cM_k) =  \boldsymbol K_0(\Omega_k\prmodsl)$, $G_0(\cM_k) = \boldsymbol G_0(\Omega_k\lfmodsl)$ and  $\widehat G_0(\cM_k) = \boldsymbol G_0(\Omega_k\lfmods)$. We also write
\begin{align*}
K_0(\cM) &= \bigoplus_{k \ge 0} K_0(\cM_k) \otimes_\bZ \bQ, & G_0(\cM) &= \bigoplus_{k \ge 0} G_0(\cM_k) \otimes_\bZ \bQ, & \widehat G_0(\cM) &= \bigoplus_{k \ge 0} \widehat G_0(\cM_k) \otimes_\bZ \bQ.
\end{align*}

Regarding the behavior of tensoring with $\bQ[\xi]$ we have the following.
\begin{lem}\label{lem:polyactionK}
The functor of tensoring with $\Omega_k[\xi]$ descends to multiplication by 
$\frac{1}{1-q^2}$ on the different Grothendieck groups $ K_0(\cM)$, $ G_0(\cM)$ and $ { \widehat G}_0(\cM)$.
\end{lem}

\begin{proof}
This follows from the fact that, since the variable $\xi$ commutes with the variables used to construct $\Omega_{k}$, we get $\Omega_k[\xi] \cong \bigoplus_{i \ge 0} \Omega_k\{2i\}$  and thus $[\Omega_k[\xi]] = (1+q^2+\dots)[\Omega_k] = \frac{1}{1-q^2} [\Omega_k]$ .
\end{proof}

\smallskip

The categorical $\slt$-action on projective supermodules is very nice, 
the functors $\F_k$, $\E_k$ and $\Q_k$ satisfy 
\[ 
\F_k(\Omega_k) = \oplus_{[k+1]}\Omega_{k+1},
\mspace{40mu}
\E_k(\Omega_{k+1}) = \Q_k\Omega_k\brak{-k-1,1}\oplus \Pi\Q_k\Omega_k
\brak{k+1,-1} .
\] 
On the Grothendieck groups we have 
\begin{align}
[\F_k(\Omega_k)]& = [k+1]_q[\Omega_{k+1}] , \label{eq:actionprojF}
\intertext{and} 
  [\E_k(\Omega_{k+1})] &=
  -\frac{\pi(\lambda q^{-1}) q^{-k} + (\lambda q^{-1})^{-1}q^k}{q-q^{-1}}[\Omega_k]  .  \label{eq:actionprojE}
\end{align}
  Here we have used the notation $[-]_q$ for quantum integers to avoid confusion with
  the notation for equivalence classes in the Grothendieck groups.  
The action on simples can also be computed to be    
\[  
\F_k(S_k) =  \bQ[x_{1,k+1},s_{k+1,k+1}]\brak{-k,0} , 
\mspace{40mu}
 \E_k(S_{k+1}) =  \oplus_{ \{k+1\} }S_{k}\brak{k+2,-1} .
\] 
On the Grothendieck group $G_0(\cM)$ (and thus on $\widehat G_0(\cM)$) we have 
\begin{align}
[\F_k(S_k)] = [\bQ[x_{1,k+1},s_{k+1,k+1}]\brak{-k,0}]  
= - \frac{\pi(\lambda q^{-1})q^{-k} + (\lambda q^{-1})^{-1}q^{k}}{q-q^{-1}}\lambda q^{-2k-2}[S_{k+1}] ,   \label{eq:actionsimF}
\intertext{and}
[\E_k(S_{k+1})] = [\oplus_{ \{k+1\} }S_{k}\brak{k+2,-1}] = [k+1]_q \lambda^{-1} q^{2k+2}[S_k] . \label{eq:actionsimE}
\end{align}

\medskip

We can now state the main result of this section.

\begin{thm} \label{thm:vermacatKo}
  The functors $\F$ and $\E$ induce an action of quantum $\slt$ on the Grothendieck groups $K_0(\cM)$, $G_0(\cM)$ and $\widehat G_0(\cM)$, after specializing $\pi = -1$. With this action there are
  $\bQ\llbracket q \rrbracket [q^{-1},\lambda^{\pm 1}]$-linear isomorphisms
\begin{align*}
K_0(\cM) \cong M_A(\lambda),\mspace{100mu} G_0(\cM) &\cong M_A^*(\lambda),
\end{align*}
of $\dot U_\lambda$-modules, with $A= \bQ\llbracket q \rrbracket [q^{-1},\lambda^{\pm 1}]$, and a $\bQ\pp{q,\lambda}$-linear isomorphism 
\[
\widehat G_0(\cM) \cong  M(\lambda),
\]
of $U_q(\slt)$-representations.
  Moreover, these isomorphisms send classes of projective indecomposables to canonical basis elements and
  classes of simples to dual canonical elements, whenever this makes sense.
\end{thm}

\begin{proof}
By exactness and Theorem~\ref{thm:vermacat}, the action of the functors
$\F$, $\E$ and $\K$ descend to an action on the Grothendieck groups that satisfies the $\slt$-relations. 

Propositions \ref{prop:1dimG0} and~\ref{prop:1dimK0} yield two isomorphisms
$K_0(\cM) \cong M_A(\lambda)$ and $G_0(\cM) \cong M_A^*(\lambda)$ as $\bQ\llbracket q \rrbracket [q^{-1},\lambda^{\pm 1}]$-modules by sending respectively $[\Omega_k]$ to $m_k$ and  $[S_k]$ to $m^k$.

Proposition~\ref{prop:1dimGhat0} gives an isomorphism $\widehat G_0(\cM) \cong  M(\lambda)$ of $\bQ\pp{q,\lambda}$-vector spaces.
Comparing the action of $E$ and $F$ on the canonical basis~(\ref{eq:cbaction}) with~(\ref{eq:actionprojF}) and~(\ref{eq:actionprojE}), and on the dual canonical basis~(\ref{eq:dcbaction}) with~(\ref{eq:actionsimF}) and~(\ref{eq:actionsimE}), concludes the proof.
\end{proof}


We finish this section with a categorification of the Shapovalov forms defined in \S~\ref{ssec:ushapo}. 
For $M,N \in \cM$ denote by $M^{\opp}$ the right module given by acting with the opposite algebra. Then we consider the bigraded (super)vector space 
\[
M^{\opp} \otimes_{(\oplus_k \Omega_k)} N.
\]
Since both $M$ and $N$ are cone bounded, locally finite dimensional, we get
\[
\sdim M^{\opp} \otimes_{(\oplus_k \Omega_k)} N \in \bQ\pp{q, \lambda}.
\]
 For the sake of keeping the notations short, we will write $\otimes_\cM$ for $\otimes_{(\oplus_k \Omega_k)}$.

\begin{thm}
In the Grothendieck groups,
\[ 
\left([M],[N]\right)_\lambda = \sdim M^{\opp} \otimes_{\cM} N , 
\] 
where $(-,-)_\lambda$ is the universal Shapovalov from \S\ref{ssec:ushapo}.
\end{thm}

\begin{proof}
Let $\bQ$ be the unique projective indecomposable in $\Omega_0\fgmods$.  
We have
$\left( [\bQ],[\bQ] \right)_{\lambda} = \sdim \bQ \otimes_{\Omega_0} \bQ=1$. 
Moreover, by construction
\begin{align*}
  (\F X)^{\opp} \otimes_{\cM} Y &\cong   ( \left(\oplus_k \Omega_{k+1,k} \brak{-k,0} \right) \otimes_\cM X)^{\opp} \otimes_{\cM} Y \\
&\cong X^{\opp} \otimes_\cM \left( \oplus_k  \Omega_{k,k+1}  \brak{-k,0} \right) \otimes_\cM Y\\
&\cong X^{\opp} \otimes_\cM \left( \oplus_k  \left(\Omega_{k,k+1}  \brak{k+2,-1}\right) \brak{-2k-1,1} \right) \brak{-1,0} \otimes_\cM Y\\
 &\cong  X^{\opp} \otimes_{\cM} (\K\E Y \brak{-1,0}) \\
&\cong X^{\opp} \otimes_\cM (\rho(\F) Y),
\end{align*}
for any $X,Y \in \cM$. 
Finally, the bilinearity is obvious from the behavior of the dimension with respect to direct sum and tensor product.
\end{proof}


For, $M, N \in \cM$, denote by 
$$\HOM_\cM(M,N) = \bigoplus_{i,j,k \in \bZ \times \bZ \times \bZ/2\bZ} \Hom_{\cM}(M,\Pi^k N\brak{i,j})$$ 
the enriched $\Hom$-spaces.  
They consist of $\bZ\times\bZ$-graded  $\bQ$-(super)vector spaces of morphisms. 

\begin{thm}
In the Grothendieck groups,whenever $\HOM_{\cM}(M,N)$ has a locally finite cone bounded dimension, we have 
\[ 
\brak{[M],[N]}_\lambda = \sdim \HOM_{\cM}(M,N) , 
\] 
where $\brak{-}_\lambda$ is the twisted Shapovalov from \S\ref{ssec:ushapo}.
\end{thm}

\begin{proof}
We have 
$\brak{ [\bQ],[\bQ] }_{\lambda} = \sdim\HOM_{\cM_0}(\bQ,\bQ)=1$ by construction. 
It follows from Proposition~\ref{prop:FEadj} that for any $X,Y \in \cM$,
\[
\HOM_{\cM}(\F X,Y) \cong \HOM_{\cM}(X,\K\E Y\brak{-1,0}) \cong \HOM_{\cM}(X,\tau(\F) Y).
\]
Finally, 
$q^m\lambda^n\brak{ \F^i[\bQ],\F^j[\bQ] }_\lambda=\brak{q^{-m}\lambda^{-n} \F^i[\bQ],\F^j[\bQ] }_\lambda
  =\brak{ \F^i[\bQ], q^m\lambda^n \F^j[\bQ] }_\lambda$, 
  is a consequence of the definition of the (enriched) Hom spaces in a bigraded category.
\end{proof}

\subsection{2-Verma modules}

As explained in Subsections~\ref{ssec:catactslt} and~\ref{ssec:vermacat} above, the functors $(\F,\E)$
are adjoint (up to grading shifts) but not biadjoint. 
In order to accommodate our construction to the concept of strong 2-representations and $Q$-strong 2-representations
(in the sense of Rouquier~\cite{R1} and Cautis-Lauda~\cite{cl} respectively) we adjust their definitions into 
the notion of a 2-Verma module for~$\slt$.

\smallskip 

Since we need to work with short exact sequences of $1$-morphisms, 
we require the $\Hom$-categories of a 2-Verma module to be Quillen exact~\cite[Sec. 2]{quillen} (see also~\cite[App. A]{keller}). 
Recall that a full subcategory $\cC$ of an abelian category $\cA$ is closed under extensions if for all short exact sequence $0 \rightarrow A \rightarrow B \rightarrow C \rightarrow 0$ in $\cA$ with $A$ and $C$ in $\cC$,
then $B$ is also in $\cC$. An additive full subcategory of an abelian category, closed under extensions, is called \emph{Quillen exact}.

\smallskip

All results from \S\ref{ssec:completedG0} apply for Quillen exact categories if the category admits unions and intersections of admissible objects. That is whenever there are short exact sequences 
\begin{align*}
0 \rightarrow A_1 \rightarrow &B \rightarrow C_1 \rightarrow 0, & 0 \rightarrow A_2 \rightarrow &B \rightarrow C_2 \rightarrow 0
\end{align*}
in $\cC$ then there are also short exact sequences
\begin{align*}
0 \rightarrow A_1 \cap A_2 \rightarrow &B \rightarrow X \rightarrow 0, & 0 \rightarrow A_1 \cup A_2 \rightarrow &B \rightarrow Y \rightarrow 0
\end{align*}
in $\cC$.

\medskip

\begin{defn} 
  Let $c$ be either an integer or a formal parameter and define $\varepsilon_c$ 
to be zero if $c\in\bZ$ and to be 1 otherwise. Let $\Lambda_c = c  - 2\bN_0$ be the support. 
A  \emph{2-Verma module} for $\slt$ with heighest weight $q^c$
consists of a bigraded $\Bbbk$-linear idempotent complete, 2-category $\tM$ admitting a parity 2-functor
  $\Pi : \tM \rightarrow \tM$, where:
\begin{itemize}
\item The objects of $\tM$ are indexed by weights $\mu\in \Lambda_c$.
\item There are identity 1-morphisms $\un_{\mu}$ for each $\mu$,
  as well as 1-morphisms $\F\un_\mu\colon\mu \to \mu -2$ in $\tM$ and their grading shift.
  We also assume that $\F\un_\mu$ has a right adjoint and define the 1-morphism $\E \un_{\mu-2} \colon \mu-2 \to \mu$ as
a grading shift of a right adjoint of $\F\un_\mu$, 
\[
\E \un_{\mu-2}  = (\F\un_\mu)_R \brak{\mu +2-c,-\varepsilon_c} .
\]
\item The $\Hom$-spaces between objects are locally additive, cone complete, Quillen exact categories.
\end{itemize}
On this data we impose the following conditions:
\begin{enumerate}
\item The identity 1-morphism $\un_{\mu}$ of the object $\mu$ is isomorphic to the zero 1-morphism 
  if $\mu\notin\Lambda_c$. 
\item The enriched $\HOM_{\tM}(\un_{\mu}, \un_{\mu})$ is cone bounded for all $\mu$. 
\item \label{co:EF} There is an exact sequence 
\[
0\xra{\quad} 
\F\E\un_\mu
\xra{\quad} 
\E\F\un_\mu
\xra{\quad} 
 \Q_\mu \brak{-c+\mu,\epsilon_c} \oplus\ \Pi\Q_\mu  \brak{c-\mu,-\epsilon_c} 
\xra{\quad} 0 , 
\]
where $\Q_\mu := \bigoplus_{k \ge 0} \Pi \un_\mu \brak{2k+1,0}$.
\item  For each $k \in \bN_0$, $\F^k\un_{\mu}$ carries a faithful action of the
  enlarged nilHecke algebra. 
\end{enumerate}
\end{defn}

Let $\cC_0$ be a full subcategory of an abelian category $\cA$. For all $i \ge 0$, define recursively $\cC_{i+1}$ as the full subcategory of $\cA$ containing all the objects of $\cC_i$ and all $B$ for all short exact sequence $0 \rightarrow A \rightarrow B \rightarrow C \rightarrow 0$ in $\cA$ with $A$ and $C$ in $\cC_i$. We call $\bigcup_i \cC_i$ the
\emph{completion under extensions of $\cC_0$ in $\cA$}.
It is clear that if $\cC_0$ is also additive, then its completion under extension is Quillen exact.

\medskip

Form the
2-category $\tM'(\lambda q^{-1})$ whose objects 
are the categories $\cM_k$, the 1-morphisms are locally finite, cone bounded direct sums of shifts of functors from $\{E_k, F_k, Q_k, \id_k \}$ and the 
2-morphisms are (grading preserving) natural transformations of functors.
We define $\tM(\lambda q^{-1})$ as the completion under extensions of $\tM'(\lambda q^{-1})$ in the abelian category of all functors.
In this case $\tM(\lambda q^{-1})$ is a 2-Verma module for $\slt$. 
Now take the cone completion of $\extflag$ from \S\ref{subsec:commutator}, namely add the cone bounded, locally finite coproduct of $\Omega_k$. Then the completion under extensions of this 2-category in $\sbim$ is also a 2-Verma module, equivalent to $\tM(\lambda q^{-1})$.


%
%
\section{Categorification of the Verma modules with integral highest weight}\label{sec:intvermacat}

In this section we give a categorical interpretation of the evaluation map $\mathrm{ev}_{n}\colon M(\lambda q^{-1})\to M(n)$ for $n\in\bZ$.

%
%
\subsection{Categorification of $M(-1)$}\label{ssec:Mminusone}

Forgetting the $\lambda$-gradings of the superrings $\Omega_k$ and $\Omega_{k,k+1}$ from \S\ref{ssec:hoch} 
results in single graded superrings that we denote $\Omega_k(-1)$ and $\Omega_{k,k+1}(-1)$ respectively. 
We write $\brak{s}$ for the shift of the $q$-grading up by $s$ units.

Define $\cM_k(-1)=\Omega_k(-1)\lfmods$ and $\cM_{k+1,k}(-1)=\Omega_{k+1,k}(-1)\lfmods$ with the functors  
\begin{align*} 
\F_k&\colon\cM_{k}(-1)\to\cM_{k+1}(-1),
\\  
\E_k&\colon\cM_{k+1}(-1)\to\cM_k(-1),
\\
\Q_k&\colon\cM_{k}(-1)\to\cM_k(-1),
\\
\K_k&\colon\cM_{k}(-1)\to\cM_k(-1),
\end{align*}
as in \S\ref{sec:vermacat}. 
Denote also $\cM(-1)=\oplus_{k\geq0}\cM_k(-1)$.

Since the $q$-grading in $\Omega_k(-1)$ and $\Omega_{k+1,k}(-1)$ is bounded from below 
and both superrings have one-dimensional lowest-degree part, all the results in \S\ref{sec:vermacat} 
can be transported to the singly-graded case.  
Note that either $\Omega_k(-1)$ and $\Omega_{k+1,k}(-1)$ are
the product of a graded local ring with degree 0 part isomorphic to $\bQ$ with a finite dimensional superring.

\smallskip 

 The (topological) Grothendieck group $G_0(\cM_k(-1))$ is one-dimensional and generated either by the class
of the projective indecomposable, either by the class of the simple object, both unique up to
isomorphism and grading shift. 
Also note that $K_0(\cM_k(-1))$ is generated by the unique projective and it is
homeomorphic to $G_0(\cM_k(-1))$.
However, we prefer to use $G_0(\cM_k(-1))$ which seems to be a more natural choice
since the dual canonical basis in $K_0$ only exists as a formal power series.

\smallskip

Define the $2$-category $\tM(-1)$ like $\tM(\lambda q^{-1})$ but with the $\cM_k(-1)$s as objects.
Collapsing the $\lambda$-grading defines a forgetful 2-functor $U : \tM(\lambda q^{-1})\to\tM(-1)$.
It is clear that $\tM(-1)$ is a 2-Verma module. 

\smallskip

Having the forgetful 2-functor $U$ at hand, our strategy is to first categorify the shifted universal Verma module 
$M(\lambda q^n)$ for arbitrary $n$ and then to apply $U$ to get a categorification of the Vermas with integral highest weight.

\begin{rem}
  While this approach yields a categorification of $M(n)$, it is interesting 
  challenge to construct one where the $\slt$-commutator relation is given in the form of a finite direct sum. 
\end{rem}

%
%
\subsection{Categorification of the shifted Verma module $M(\lambda q^{n})$ for $n\in\bZ$}\label{ssec:vermacatshifted}

\medskip

Let $n \in \bZ, n < 0$ be fixed and let $G_{-n-1,k}$ and $G_{-n-1,k,k+1}$ the varieties of partial flags in $\bC^\infty$  
\begin{align*}
G_{-n-1,k} &= \{ (U_{-n-1},U_k) \vert  \dim_{\bC}U_{-n-1} = -n-1, \dim_{\bC}U_k = k, \ 
0 \subset U_{-n-1} \subset U_k  \subset \bC^{\infty} \},  
\intertext{and} 
  G_{-n-1,k,k+1} &= \{ (U_{-n-1},U_k,U_{k+1}) \vert
\dim_{\bC}U_{-n-1} = -n-1,
 \dim_{\bC}U_k = k, \dim_{\bC}U_{k+1} = k+1,  
\\ & \mspace{25mu} 
0 \subset U_{-n-1} \subset U_k \subset U_{k+1} \subset \bC^{\infty} \}. 
\end{align*}
Their cohomologies are generated by the Chern classes 
\[
H(G_{-n-1,k})  \cong \bQ[x_{1,k},\dotsc ,x_{-n-1,k}, z_{1,k},\dotsc , z_{k,k}] ,
\]
with $\deg_q(x_{i,k})=2i$, $\deg_q(z_{i,k})=2i$,
and
\[
H(G_{-n-1,k,k+1})  \cong \bQ[x_{1,k+1},\dotsc ,x_{-n-1,k+1}, z_{1,k+1},\dotsc , z_{k,k+1},\xi_{k+1}], 
\]
with $\deg_q(x_{i,k+1})=2i$, $\deg_q(z_{i,k+1})=2i$ and $\deg_q(\xi_{k+1})=2$. 

\medskip

The forgetful map $G_{-n-1,k}\to G_k$ gives $H(G_{-n-1,k})$ 
the structure of a 
$(H(G_{k}),H(G_{-n-1,k}))$-superbimodule and similarly 
for $H(G_{-n-1,k,k+1})$, which becomes a 
$(H(G_{k,k+1}),H(G_{-n-1,k,k+1}))$-bimo\-dule  
under the forgetful map $G_{-n-1,k,k+1}\to G_{k,k+1}$. 
Tensoring on the left with $H(G_{-n-1,k})$ (over $H(G_k)$) and with 
$H(G_{-n-1,k,k+1})$ (over $H(G_{k,k+1})$) gives 
 exact functors 
from $H(G_k)\fgmods$ to $H(G_{-n-1,k})\fgmods$ and 
from $H(G_{k,k+1})\fgmods$ to $H(G_{-n-1,k,k+1})\fgmods$ respectively. 

\medskip

For each $j\in\bN_0$ put 
\begin{align*} 
X_{-n-1,j} &= 
\begin{cases}
H(G_{-n-1,j})\otimes \Ext_{H(G_j)}(S_j,S_j) & \text{ if } j\geq -n-1
\\
0 & \text{else} ,
\end{cases}
\intertext{and} 
X_{-n-1,j,j+1} &= 
\begin{cases}
H(G_{-n-1,j,j+1})\otimes \Ext_{H(G_{j+1})}(S_{j+1},S_{j+1}) & \text{ if } j\geq -n-1
\\
0 & \text{else} ,
\end{cases}
\end{align*}
and for all $k\in\bN_0$  define the superrings 
\begin{align*}
\Omega_{k}^{n} &=  X_{-n-1,-n-1+k}\otimes_k \Omega_{-n-1+k} ,
\\
\Omega_{k,k+1}^{n} &=   X_{-n-1,-n-1+k,-n+k}\otimes_{k,k+1} \Omega_{-n-1+k,-n+k}  .
\end{align*} 

\medskip

Now take $n\in \bN_0$. Let also $\Omega_{k}^{n}\subset \Omega_k  $ and $\Omega_{k,k+1}^{n}\subset \Omega_{k,k+1}$ be the sub-superrings  
\begin{align*} 
\Omega_{k}^{n}  &=\bQ[x_{1,k},\dotsc ,x_{k,k}, s_{-n,k},\dotsc , s_{k-1-n,k}],
\intertext{and} 
\Omega_{k,k+1}^{n} &=\bQ[w_{1,k},\dotsc ,w_{k,k},\xi_{k+1}, \sigma_{-n,k+1},\dotsc ,\sigma_{k-n,k+1}] ,
\end{align*}  
where we compute $s_{i,k}$ and $\sigma_{i,k+1}$ for $i \le 0$ recursively with the formulas
\begin{align*}
s_{i,k} &= -  \sum_{\ell = 1}^k x_{\ell, k} s_{i+\ell,k}, &
\sigma_{i,k+1} &= -  \sum_{\ell = 1}^k (x_{\ell, k} + \xi_{k+1} x_{\ell-1,k} ) \sigma_{i+\ell,k+1}.
\end{align*}

After a suitable change of variables we can write
\begin{equation}\label{eq:subpos}\begin{aligned} 
\Omega_{k}^{n}  &=\bQ[x_{1,k},\dotsc ,x_{k,k}, \tilde{s}_{1,k},\dotsc , \tilde{s}_{k,k}]
\\ 
\Omega_{k,k+1}^{n} &=\bQ[x_{1,k},\dotsc ,x_{k,k},\xi_{k+1}, \tilde{s}_{1,k+1},\dotsc , \tilde{s}_{k+1,k+1}]  , 
\end{aligned}\end{equation}  
with $\deg_{q,\lambda,q}(\tilde{s}_{i,k})=\deg_{q,\lambda}(\tilde{s}_{i,k+1})=(2n-2i ,2)$. 

\medskip

In order to define the analogous of the category $\cM$ from \S\ref{sec:vermacat} note that we get
\begin{align}\label{eq:stildephi}
\phi^*_{k}(x_{i,k}) &= x_{i,k+1}, &
 \phi^*_{k}(\tilde s_{i,k}) &= \tilde s_{i,k+1} + \xi_{k+1} \tilde s_{i+1,k+1}, \\
\label{eq:stildepsi}
\psi^*_{k}(x_{i,k+1}) &= x_{i,k+1} + \xi_{k+1} x_{i-1,k+1} , &
 \psi^*_{k+1}(\tilde s_{i,k+1}) &= \tilde s_{i,k+1}. 
\end{align}

\medskip 
%

As in Definition~\ref{def:shiftedbim},  we define the shifted superbimodules. 

\begin{defn}  
For $n\in\bZ$ and $k\in\bN_0$ we put 
\begin{align*}
\omg_{k+1,k}^{n} &= \Omega_{k+1,k}^{n}\brak{-k,0} ,
&
\omg_{k,k+1}^{n} &= \Omega_{k,k+1}^{n}\brak{k-n+1,-1} ,
\intertext{and define}
 \omg_{k(k\pm 1)k}^{n} &= \omg_{k,k\pm 1}^{n}\otimes_{k\pm 1}\omg_{k\pm 1,k}^{n} ,
& 
 \omg_k^{\xi, n} & = \Pi\Omega_k^{ n}[\xi]\brak{1,0} .
\end{align*}
\end{defn}

The analogue of Corollary~\ref{prop:sesbim} reads as below. 
\begin{lem}\label{lem:sesbimsh}
There are short exact sequences of $(\Omega_k^n,\Omega_k^n)$-superbimodules 
\[ 
0  \xra{\ \ } \omg_{k(k-1)k}^{n} \xra{\ \ } \omg_{k(k+1)k}^{n} \xra{\ \ } 
\omg_k^{\xi,n}\brak{n-2k,1} \oplus \Pi\omg_{k}^{\xi,n} \brak{2k-n,-1}  \xra{\ \ } 0 .
\]
\end{lem}
\begin{proof}
 For $n<0$ we regard $\Omega_k^{n}$ and $\Omega_{k,k+1}^{n}$ as algebras over $\Omega_{-n -1}$ and apply 
  the analysis leading to Corollary~\ref{prop:sesbim}. 
For $n\geq 0$ we use the presentations of $\Omega_k^{n}$ and $\Omega_{k,k+1}^{n}$ in Eq.~\eqref{eq:subpos}  and apply 
 again the analysis leading to Corollary~\ref{prop:sesbim}. In both cases  the claim follows by tracing carefully the bidegrees in the
  homomorphisms in \S\ref{subsec:commutator}.
\end{proof}

We now define  
categories $\cM_k^n = \Omega_k^{n}\fgmods$ and $\cM^n = \oplus_{k \ge 0} \cM_k^n$ as we did $\cM$
in \S\ref{sec:vermacat}.
We also define $\tM(\lambda q^n)$ as the 2-category with objects $\cM_k^n$, for $k\in\bN_0$. 
The 1-morphisms of $\tM(\lambda q^n)$ are direct sums of grading shifts of the functors
\begin{align*}
\F_k&\colon\cM_k^n \to\cM_{k+1}^n & \F_k(-) 
&= \Res_{k+1}^{k,k+1}\circ\bigl(\Omega_{k+1,k}^{n}\otimes_{k}(-)\bigr)\brak{-k,0},
\\ 
\E_k&\colon\cM_{k+1}^n\to\cM_{k}^n & \E_k(-) 
&= \Res_k^{k,k+1}\circ\bigl(\Omega_{k,k+1}^{n})\otimes_{k+1}(-)\bigr)\brak{k-n+1,-1}, 
\\
\Q_k&\colon\cM_{k}^n\to\cM_{k}^n & \Q_k(-) & = \Omega_k^{n}[\xi] \otimes_k (-)  \brak{1,0},
\\
\K_k&\colon\cM_{k}^n \to\cM_{k}^n & \K_k(-) & = (-)\brak{n-2k,1} ,
\end{align*}  
and the 2-morphisms are (grading-preserving) natural transformations between these functors. 
As before, we put 
\[
\E = \bigoplus\limits_{k\geq 0}\E_k, \mspace{20mu}
\F = \bigoplus\limits_{k\geq 0}\F_k, \mspace{20mu}
\K = \bigoplus\limits_{k\geq 0}\K_k, \mspace{15mu}\text{and}\mspace{15mu}
\Q = \bigoplus\limits_{k\geq 0}\Q_k .
\]  

The analogs of Theorems~\ref{thm:vermacat} and~\ref{thm:vermacatKo} follow as in \S\ref{sec:vermacat} by using Lemma~\ref{lem:sesbimsh}. 
We state them below for the record. 
\begin{thm}\label{thm:vermacatsh}
We have natural isomorphisms of exact endofunctors of $\cM^n$ 
\[
\K\circ\K^{-1} \cong \id_{\cM} \cong \K^{-1}\circ\K , 
\]
\[
\K\circ\E \cong \E\circ\K\brak{2,0},\qquad \K\circ\F\cong \F\circ\K\brak{-2,0} , 
\]
and an exact sequence 
\[
0\xra{\quad} 
\F\circ\E
\xra{\quad} 
\E\circ\F 
\xra{\quad} 
\Q\circ \bigl( \K^{} \oplus\ \Pi\K^{-1} \bigr) 
\xra{\quad} 0 .
\]
\end{thm}

All the above implies that $\tM(\lambda q^n)$ is a 2-Verma module for $\slt$. 

\begin{thm} \label{thm:vermacatshnegKo}
The three Grothendieck groups $G_0(\cM^n)$, $K_0(\cM^n)$ and $\widehat G_0(\cM^n)$, together with
  the action induced by functors $\F$ and $\E$, are respectively isomorphic with the $\dot U_\lambda$-modules  $M_A^*(\lambda q^n)$ and $M_A(\lambda q^n)$, and with the $U_q(\slt)$-module $\widehat M(\lambda q^{n})$.
\end{thm}

%
%
\subsection{Categorification of the Verma module $M(n)$ for $n\in\bZ$}\label{ssec:vermacatint}

As before, we apply the map that forgets the $\lambda$-grading on the superrings $\Omega_k^{n}$ and $\Omega_{k,k+1}^{n}$  
to obtain singly-graded superrings $\Omega_k(n)$ and $\Omega_{k,k+1}(n)$.  
In this case we still have that the $q$-grading in $\Omega_k(n)$ and $\Omega_{k+1,k}(n)$ is bounded from below 
and both superrings have one-dimensional lowest-degree part, and we can make use of  all the results
in \S\ref{sec:vermacat}.  
Note that again both $\Omega_k(n)$ and $\Omega_{k+1,k}(n)$ are the product of a graded local ring 
with degree 0 part isomorphic to $\bQ$ with a finite dimensional superring.

\medskip

Denote by $\cM(n)$ the image of $\cM^n$ under the forgetful functor.
We keep the notation $\F_k$, $\E_k$, $\Q_k$ and $\K_k$ for $U(\F_k)$,  $U(\E_k)$,  $U(\Q_k)$ and  $U(\K_k)$.  
The 2-categories $\tM(n)$ constructed in the obvious way, yield 2-Verma modules. 

\medskip

It is easy to see that 
\begin{align*}  
\F_k\Omega_k^{n} &\cong \oplus_{[k+1]}\Omega_{k+1}^{n},  
\intertext{and} 
\E_k\Omega_{k+1}^{n} &\cong \Q_k\Omega_k^{n}\brak{n-k} \oplus\Pi\Q_k\Omega_k^{n}\brak{k-n}, 
\end{align*} 
which means that in the Grothendieck group we have 
\begin{equation}\label{eq:vermacatintEFKo} 
[\F][\Omega_k^{n}]=[k+1]_q[\Omega_{k+1}^{n}], 
\mspace{40mu}
[\E][\Omega_{k+1}^{n}] = [n-k+1]_q[\Omega_{k}^{n}]
\end{equation}
after specializing $\pi = -1$.

\medskip 

We now proceed to analyze the cases $n<0$ and $n\geq 0$ separately.
\subsubsection{The case $n<0$} \label{ssec:vermacatMneg}

The arguments in the  proof of Theorem~\ref{thm:vermacatKo} can be
applied almost unchanged to get the following.

\begin{thm} \label{thm:vermacatnegKo}
For $n<0$ the Grothendieck groups $G_0(\cM(n))$ and $K_0(\cM(n))$, together with
the action induced by functors $\F$ and $\E$, are isomorphic with the Verma module $ M(n)$.
\end{thm}

The composite of the functor on $\cM_k$ obtained by tensoring with the appropriate
superbimodules $X_{n-1,k}$ and $X_{n-1,k,k+1}$ with 
the forgetful functor $U$ defines a 2-functor $\mathrm{EV}_{-|n|}\colon \tM\to \tM(n)$ which is exact, takes projectives 
to projectives, 
and categorifies the evaluation map 
$\mathrm{ev}_{-|n|}\colon M(\lambda q^{-1})\to M(n)$.

\medskip

The categorification of $M(n)$ using $\cM(n)$ for $n<0$ is not minimal, in the sense that there is a 
smaller category with the same properties we now describe. 
Note that $\Omega_k(-\vert n\vert)$ and $\Omega_{k,k+1}(-\vert n\vert)$  have presentations  
\begin{align*} 
\Omega_k(-\vert n\vert) &= \bQ[\und{x}_{n-1},\und{s}_{n-1}][z_1,\dotsc ,z_{k}, s_n,\dotsc , s_{n+k}] ,
\intertext{and}
\Omega_{k,k+1}(-\vert n\vert) &= \bQ[\und{x}_{n-1},\und{s}_{n-1}][z_1,\dotsc ,z_{k+1},\xi_{k+1}, s_n,\dotsc , s_{n+k+1}] ,
\end{align*} 
where $z_i$ and $\xi_{k+1}$ are even with $\deg_q(z_i)=2i$, $\deg_q(\xi_{k+1})=2$ and $s_i$ is odd with $\deg_q(s_i)=-2i$ 
(recall that $\bQ[\und{x}_{n-1},\und{s}_{n-1}]=\Omega_k$).

Let $J_k\subset \Omega_k(-\vert n\vert)$ and $J_{k,k+1}\subset \Omega_{k,k+1}(-\vert n\vert)$ 
be the 2-sided ideals generated by $(\und{x}_{n-1},\und{s}_{n-1})$ 
and 
define the 2-category $\cM^{min}(-\vert n\vert)$ as before but using the quotient superrings 
\begin{align*}
\Omega_{k}^{min}(-\vert n\vert) &= \Omega_k(-\vert n\vert)/{J_k} ,
\intertext{and} 
\Omega_{k,k+1}^{min}(-\vert n\vert) &= \Omega_{k,k+1}(-\vert n\vert)/{J_{k,k+1}} ,
\end{align*} 
instead.  
We get functors $\F^{min}$, $\E^{min}$, $\Q^{min}$ and $\K^{min}$ with the same properties as in 
Theorem~\ref{thm:vermacatsh}
while the Grothendieck group of $\cM^{min}(n)$ is still isomorphic to $M(n)$. 
Using the surjection from $\Omega_k(-\vert n\vert)$ to $\Omega_{k}^{min}(-\vert n\vert)$  we can construct 
an obvious functor $\Psi$ from $\cM(n)$ to $\cM^{min}(n)$ that sends 
 $\cM_k(n)$ to $\cM_k^{min}(-\vert n\vert)$  
and $(\F_k,\E_k,\Q_k,\K_k)$ to $(\F_k^{min},\E_k^{min}, \Q_k^{min},\K_k^{min})$. 
Moreover, $\Psi$ sends projectives to projectives, simples to simples, is exact, full 
and bijective on objects.

\subsubsection{The case $n\geq 0$}

We have to be a bit careful for this case, as the dual canonical basis does not exist
for $M(n)$. Indeed, in $G_0(\cM(n))$ we get the equality
\[
[\Omega_{n+1}^n] =  \prod_{i=1}^{n+1} \frac{1+\pi q^{2n-2i+2} }{1-q^{2i}} [S_{n+1}^n]
\]
and thus after specializing to $\pi = -1$, it gives $[\Omega_{n+1}^n]=0$.
Of course, the element $(1+\pi)$ is not invertible in $\bZ_\pi\llbracket q \rrbracket[q^{-1}]$ and thus we cannot use the projective resolution of $S_{n+1}^n$ to generate it with the projective $\Omega_{n+1}^n$ neither.
In fact, we have $[F_n] = 0$, and $G_0(\cM(n))$ contains $V(n)$ as a submodule. The action of $U_q(\slt)$ on $G_0(\cM(n))/V(n)$ is trivial with $E = F = 0$.

Therefore, the Grothendieck group $G_0(\cM(n))$ is not the way we want to 
decategorify $\cM(n)$ and we will work only with $K_0(\cM(n))$, which is freely generated by the classes or projectives $[\Omega_k^n]$. Again comparing the action on the canonical basis and (\ref{eq:vermacatintEFKo}) gives the following theorem.

\begin{thm} \label{thm:vermacatposKo}
For each $n\geq 0$ the Grothendieck group $K_0(\cM(n))$, together with
  the action induced by functors $\F$ and $\E$, is isomorphic to the Verma module $M(n)$. 
\end{thm}
%


%
%

\section{Categorification of the finite dimensional irreducibles from the Verma categorification}\label{sec:Vn}

\subsection{The ABC of the DG world}

We start by recollecting some basic facts about DG-algebras and their (derived) categories of modules 
following closely the exposition in~\cite[\S10]{BL}. 

\medskip

A \emph{differential graded algebra} (DG-algebra for short) $(A,d)$ is a $\bZ$-graded associative unital algebra $A$  
with 1 in degree zero, equipped with an additive endomorphism $d$ 
of degree $-1$ satisfying
\[ 
d^2=0,\mspace{40mu} d(ab) = d(a)b+(-1)^{\deg{a}}ad (b),\mspace{40mu} d(1)=0 .
\]
A \emph{homomorphism} between DG-algebras $(A,d)$ and $(A',d')$ is a homomorphism 
$\phi\colon A\to B$ of algebras intertwining the differentials, $\phi\circ d=d\circ\phi$. 

\medskip

A left \emph{DG-module} $M$ over $A$ is a $\bZ$-graded left $A$-module with a differential 
$d_M\colon M_{i}\to M_{i-1}$ such that, for all $a\in A$ and all $m\in M$,
\[ 
d_M(am) = d(a)m + (-1)^{\deg(a)}ad_M(m) . 
\] 
We have the analogous notion for right $A$-modules and bimodules. 
Denote by $(A,d)\amod$ the abelian category of (left) DG-modules over $A$. 
We say $P\in(A,d)\amod$ is \emph{projective} if for every acyclic $M$ in $(A,d)\amod$  
the complex $\Hom_A(P,M)$ is also acyclic. 
The homology $H(M)$ of a DG-module over $A$ is the usual homology of the chain complex $M$. 
It is a graded module over the graded ring $H(A)$.   
We say that $A$ is \emph{formal} if it is quasi-isomorphic to $H(A)$, that is if there exists a map $A \rightarrow H(A)$ or $H(A) \rightarrow A$ inducing an isomorphism on the homology.

\medskip

Two morphisms $f,g\colon M\to N$ in $(A,d)\amod$ are \emph{homotopic} if there is a 
degree 1 map $s\colon N\to M$ such that $f-g=sd_M+d_Ns$. 
The \emph{homotopy category} $\cK_A$, which is a triangulated category, is obtained from $(A,d)\amod$ by 
modding out by null-homotopic maps. 
Inverting quasi-isomorphisms results in the \emph{derived category} $\cD(A)$, which is triangulated and idempotent complete.  

\medskip

The localization functor gives an isomorphism from the full subcategory of $\cK_A$ consisting of projective objects 
to the derived category $\cD(A)$. 
We say $M$ is compact if the canonical map 
\[ 
\oplus_{i\in I}\Hom_{\cD(A)}(M,N_i)\to \Hom_{\cD(A)}(M,\oplus_{i\in I} N_i) 
\]  
is an isomorphism for every arbitrary direct sum of DG-modules. 
Note that $A$ is compact projective. 
Let $\cD^c(A)$ be the full subcategory of $\cD(A)$ consisting of compact modules. 
It is also idempotent complete.  

\medskip

For a homomorphism of DG-algebras $\phi\colon A\to B$ the \emph{derived induction} functor is 
the derived functor associated with 
the bimodule $_BB_A$,
\[
\Ind_A^B = B\otimes^{\mathbf L}_A(-) \colon \cD(A)\to\cD(B).
\]
The \emph{derived restriction} functor is given by taking the derived hom-functor:
\[ 
\Res_A^B = \RHom_B(B_A, -) \colon \cD(B)\to\cD(A) . 
\] 
The forgetful functor via the map $\phi$ is exact and therefore, lifts trivially to the derived setting. This lift coincides with the derived restriction functor.
The above functors are adjoint:
\[
\Hom_{\cD(B)}(\Ind_A^B(M),N) \cong \Hom_{\cD(A)}(M,\Res_A^B(N)) . 
\]
If $\phi$ is a quasi-isomorphism then $\Ind_A^B$ and $\Res_A^B$ are mutually inverse equivalences of categories.

\medskip

We define the Grothendieck group of $A$ as the Grothendieck group of the triangulated category $\cD^c(A)$.  
As Khovanov pointed out in~\cite{KhovanovGL}, if $A$ is formal and (graded) noetherian 
we can describe the Grothendieck group of $A$ 
via finitely-generated $H(A)$-modules. 


\medskip

Of course, all the above generalizes easily to the case where $A$ has additional gradings and where the differential is graded over $\bZ/2\bZ$. 
In this case we speak of graded (or bigraded) DG-algebras, graded (or bigraded) homomorphisms and 
graded versions of all the categories above.   

\subsection{The differentials $d_{n}$}

We next introduce differentials on $\Omega_k$ and $\Omega_{k,k+1}$ 
turning them into DG-algebras for the parity degree. 
Recall from \S\ref{ssec:hoch} that the rings $\Omega_k$ and $\Omega_{k,k+1}$ have presentations 
\[
\Omega_k=\bQ[x_{1,k},\dotsc, x_{k,k},s_1,\dotsc, s_k]
\mspace{20mu}\text{and}\mspace{20mu}
\Omega_{k,k+1}=\bQ[x_{1,k},\dotsc, x_{k,k},\xi_{k+1}, s_{1,k+1},\dotsc, s_{k+1,k+1}].
\]

\begin{defn}\label{def:dminusn}
Define maps 
$d_{n}^k\colon\Omega_k\to\Omega_k$ and 
$d_{n}^{k,k+1}\colon\Omega_{k,k+1}\to\Omega_{k,k+1}$ 
of bidegree $\brak{2n+2,-2}$ and parity $-1$ by 
\begin{align*}
d_{n}^k(x_{r,k}) &=0,&  &&
d_{n}^k(s_{r,k}) &= Y_{n-r+1,k}, 
\intertext{and} 
d_{n}^{k,k+1}(x_{r,k}) &=0,&
 d_{n}^{k,k+1}(\xi_{k+1})  &=0,&  
d_{n}^{k,k+1}(s_{r,k+1}) &= Y_{n-r+1,k+1}, 
\end{align*} 
respecting the Leibniz rule
\[ 
d_n^k(ab) = d_n^k(a)  b + (-1)^{p(a)} a d_n^k(b)
\mspace{40mu} 
d_n^{k,k+1}(ab) = d_n^{k,k+1}(a)  b + (-1)^{p(a)} a d_n^{k,k+1}(b).
\]  
\end{defn}

\medskip

From now on we use $d_{n}$ to denote either $d_{n}^{k}$ and $d_{n}^{k,k+1}$ whenever 
the $k$ or the $k,k+1$ are clear from the context.
The maps $d_{n}$ satisfy $d_{n}\circ d_{n}=0$ and therefore 
$\Omega_k$ and $\Omega_{k,k+1}$
become DG-algebras with $d_{n}$ of bidegree $\brak{2n+2,-2}$ which we denote $(\Omega_k,d_{n})$
and $(\Omega_{k,k+1},d_{n})$.  
These algebras are bigraded and \emph{Differential Graded} with respect to the $\bZ/2\bZ$-grading (a.k.a. the parity).  

\begin{lem} 
For $k>n$, the DG-algebras $(\Omega_k,d_n)$ and $(\Omega_{k,k+1},d_n)$ are acyclic. 
\end{lem}
\begin{proof}
Let $k>n$. Then $d_n^k(s_{n+1,k})=Y_{0,k}=1$ and $d_n^{k,k+1}(s_{n+1,k+1})=1$. 
\end{proof}
\begin{rem}
For $n=0$ the DG-algebra $(\Omega_{k,k+1},d_0)$ is acyclic for all $k$ and
the DG-algebra $(\Omega_{k},d_0)$ is acyclic unless $k=0$ and in this case $(\Omega_{0},d_0)\cong \bQ$.
\end{rem}

\medskip

The DG-rings $(\Omega_k,d_{n})$ and $(\Omega_{k,k+1},d_{n})$ have a nice geometric interpretation.
\begin{prop}\label{prop:qihflag} 
The DG-rings $(\Omega_k,d_{n})$ and $(\Omega_{k,k+1},d_{n})$ are formal. 
\emph{(1)} The DG-ring $(\Omega_k,d_{n})$ is quasi-isomorphic to the cohomology of the Grassmannian variety 
$H(G_{k;n})$ of $k$-planes in $\bC^n$. 
(2) The DG-ring $(\Omega_{k,k+1},d_{n})$ is quasi-isomorphic to the cohomology of the partial flag variety 
$H(G_{k,k+1;n})$ of $k,k+1$-planes in $\bC^n$.  
\end{prop}
\begin{proof}
Let $( d_{n}(\und{s}_k) )$ denote the two-sided ideal
of $\bQ[\und{x}_k]$ 
generated by  $d_{n}(s_{1,k}), \dotsc ,d_{n}(s_{k,k})$ 
and let $( d_{n}(\und{\sigma}_{k+1}) )$ denote the two-sided ideal of
$\bQ[\und{x}_k,\xi_{k+1}]$ generated by 
$d_{n}(\sigma_{1,k+1}), \dotsc ,d_{n}(\sigma_{k+1,k+1})$.
If we equip the quotient rings 
\[
\bQ[\und{x}_k]/ ( d_{n}(\und{s}_k) ) 
\mspace{20mu}\text{and}\mspace{20mu} 
\bQ[\und{x}_k,\xi_{k+1}]/ ( d_{n}(\und{\sigma}_{k+1}) ),
\]  
with the zero differential, an easy exercise shows that the obvious surjections 
$\Omega_k\twoheadrightarrow \bQ[\und{x}_k]/ ( d_{n}(\und{s}_k) )$  
and  
$\Omega_{k,k+1}\twoheadrightarrow \bQ[\und{x}_k,\xi_{k+1}]/ ( d_{n}(\und{\sigma}_{k+1}) )$, 
are quasi-isomorphisms. 

The formula (\ref{eq:recY}) together with the definition of the differential show that $Y_{n-k+r,k} \in ( d_{n}(\und{s}_k) )$ for all $r \ge 1$ and thus give a presentation
\[
\bQ[\und{x}_k]/ ( d_{n}(\und{s}_k) ) \cong \bQ[\und{x}_k,\und{Y}_{(n-k)}]/I_{k,n},
\]
with $I_{k,n}$ the ideal generated by the homogeneous terms in the equation
\[
(1 +  x_{1,k}t + \dotsc + x_{k,k}t^k)(1 + Y_{1,k}t + \dotsc + Y_{(n-k),k} t^{n-k}) = 1,
\]
which is a presentation for the cohomology ring $H(G_{k;n})$ and thus proves part (1).
The second claim is proved in the same way.
\end{proof}

\subsection{A category of DG-bimodules}\label{ssec:dgbimod}

\begin{prop}\label{prop:diffcomm}
The maps $\phi_k^*$ and $\psi_{k+1}^*$ from \S\ref{ssec:modbim} commute with 
the differentials $d_{n}$.  
\end{prop}

\begin{proof}
From the definitions and the Leibniz rule we have that the diagrams
\[
\xymatrix
{
s_{r,k} \ar@{|->}[r]^{d_n} \ar@{|->}[d]_{\phi^*} & Y_{n-r+1,k}\ar@{|->}[d]^{\phi^*}\\
s_{r,k} + \xi_{k+1} s_{r-1,k} \ar@{|->}[r]_-{d_n} & Y_{n-r+1,k+1} + \xi_{k+1} Y_{n-r,k+1}
}
\]
and
\[
\xymatrix
{
s_{r,k+1} \ar@{|->}[r]^-{d_n} \ar@{|->}[d]_{\psi^*} & Y_{n-r+1,k+1}\ar@{|->}[d]^{\psi^*}\\
s_{r,k+1} \ar@{|->}[r]_-{d_n} & Y_{n-r+1,k+1}
}
\]
commute. The statement now follows by using the Leibniz rule recursively. 
\end{proof}

In the following we write DG $(k,s)_{n}$-bimodule for a
$((\Omega_k,d_{n}),(\Omega_s,d_{n}))$-bimodule. 
From the proposition,  
$(\Omega_{k,k+1},d_{n})$ is a DG $(k+1,k)_n$-bimodule.  
The $(\Omega_k,\Omega_k)$-bimodules $\Omega_{k,k+1}\otimes_{k+1}\Omega_{k+1,k}$  and 
$\Omega_{k,k-1}\otimes_{k-1}\Omega_{k-1,k}$ get structures of DG $(k,k)_n$-bimodules with 
$d_n$ satisfying the Leibniz rule 
\( 
d_n(a\otimes b)=d_n(a)\otimes b + (-1)^{p(a)}a\otimes d_n(b) . 
\)  

\smallskip

Define the DG $(k+1,k)_n$-bimodule 
\begin{align*} 
(\Omg_{k+1,k},d_{n})&=(\omg_{k+1,k},d_{n})\brak{0,0}=(\Omega_{k+1,k},d_{n})\brak{-k,0} , 
\intertext{and the DG $(k,k+1)_n$-bimodule} 
(\Omg_{k,k+1},d_{n})&=(\omg_{k,k+1},d_{n})\brak{-n,1}=(\Omega_{k,k+1},d_{n})\brak{k+1-n,0} .  
\end{align*} 

Note that some of the maps in \S\ref{subsec:commutator} do not extend to the various DG-bimodules above, as for example 
$\pi$ from Proposition~\ref{prop:iotainv}. However, we can equip $\Omega_k^\xi \oplus \Pi\Omega_k^\xi\brak{-2k-2}$ with a differential given by $d_n(\xi^i  \oplus 0) = 0$ and
\[
d_n(0\oplus Y_{j,k}^\xi) = \pi(Y_{n-k,k+1} \otimes \xi_{k+1}^j)  \oplus 0 = \sum_{p=k-j}^{n-k} Y_{p-k+j,k}^\xi Y_{n-k-p,k} \oplus 0,
\] 
for all $i,j \ge 0$, such that it becomes a DG-bimodule over $(\Omega_k,d_n^k)$. This differential commutes with $\pi + \mu$, as $d_n^{k,k+1}$ does with $u$, and by consequence we get a short exact sequence of $(k,k)_n$-bimodules
\[
0 \rightarrow \left(\Omega_{k(k-1)k}, d_n\right) \rightarrow \left(\Omega_{k(k+1)k}, d_n\right) \rightarrow \left( \Omega_k^\xi \brak{2k,0} \oplus \Pi\Omega_k^\xi \brak{-2k-2,2}, d_n\right) \rightarrow 0.
 \]
By the snake lemma, it descends to a long exact sequence of $H(\Omega_k, d_n) \cong H(G_{k;n})$-bimodules
\[
\xymatrix{
\dots \ar[r] & H^1\left(\Omega_{k(k+1)k}, d_n\right)  \ar[r] &  H^1\left( \Omega_k^\xi \brak{2k,0} \oplus \Pi\Omega_k^\xi \brak{-2k-2,2}, d_n\right)  \ar[lld] \\
 H^0\left(\Omega_{k(k-1)k}, d_n\right) \ar[r] & H^0\left(\Omega_{k(k+1)k}, d_n\right)  \ar[r] &  H^0\left( \Omega_k^\xi \brak{2k,0} \oplus \Pi\Omega_k^\xi \brak{-2k-2,2}, d_n\right) \ar[lld] \\
 H^1\left(\Omega_{k(k-1)k}, d_n\right)  \ar[r]& \dots&
}
 \]
Proposition~\ref{prop:qihflag} tells us the homology of $\left(\Omega_{k(k+1)k}, d_n\right)$ is concentrated in parity $0$ and thus we have a long exact sequence
\[
\xymatrix{
&0  \ar[r] &  H^1\left( \Omega_k^\xi \brak{2k,0} \oplus \Pi\Omega_k^\xi \brak{-2k-2,2}, d_n\right)  \ar[lld] \\
 H^0\left(\Omega_{k(k-1)k}, d_n\right) \ar[r] & H^0\left(\Omega_{k(k+1)k}, d_n\right)  \ar[r] &  H^0\left( \Omega_k^\xi \brak{2k,0} \oplus \Pi\Omega_k^\xi \brak{-2k-2,2}, d_n\right) \ar[lld] \\
0.&&
}
 \]
For $n -2k \ge 0$, we get that $d_n(0 \oplus 1)$ is a polynomial with a dominant monomial $\xi_{k+1}^{n-2k}$ and $d_n(0 \oplus Y_{j,k}^\xi) \neq 0$. It means the homology  of $\left( \Omega_k^\xi \brak{2k,0} \oplus \Pi\Omega_k^\xi \brak{-2k-2,2}, d_n\right)$ is concentrated in parity $0$ and given by $\bigoplus_{\{n-2k\}} q^{2k}  H(G_{k;n})$. Therefore we get the following short exact sequence
\begin{align*}
 0 \rightarrow H(G_{k,k-1;n}) &\otimes_{H(G_{k-1;n})} H(G_{k-1,k;n}) \\ &\hookrightarrow H(G_{k,k+1;n}) \otimes_{H(G_{k+1;n})} H(G_{k+1,k;n})  \twoheadrightarrow \bigoplus_{\{n-2k\}} q^{2k} H(G_{k;n}) \rightarrow 0.
\end{align*}
For $n - 2k \le 0$, we get $d_n(0 \oplus Y_{j,k}^\xi) = 0$ for $j < 2k-n$ and $d_n(0 \oplus Y_{2k-n, k}^\xi) = 1 \oplus 0$.
Thus the homology is concentrated in parity $1$ and isomorphic to $\bigoplus_{\{2k-n\}}  q^{-2k-2} \lambda^2 \Pi H(G_{k;n})$. After shifting by the degree of the connecting homomorphism, it yields the short exact sequence
\begin{align*}
 0 \rightarrow \bigoplus_{-\{2k-n\}} q^{2k} H(G_{k;n})\hookrightarrow H(G_{k,k-1;n}) &\otimes_{H(G_{k-1;n})} H(G_{k-1,k;n}) \\
&\twoheadrightarrow H(G_{k,k+1;n}) \otimes_{H(G_{k+1;n})} H(G_{k+1,k;n}) \rightarrow 0  .
\end{align*}
It is not hard to see that both short exact sequences split, with an obvious splitting morphism for the projection in the first one and an obvious left inverse for the injection in the second one, obtained by the same kind of expressions as $\iota$ and $\pi$.

\smallskip

In conclusion we recover the well-known commutator of the categorical $\slt$ action using
cohomology of the finite Grassmannians and 1-step flag varieties, which is employed in the categorification of
the irreducible finite-dimensional $\slt$-modules. 

\begin{prop}\label{prop:EFbimisominusn}
We have quasi-isomorphisms of bigraded DG $(k,k)_{n}$-bimodules 
\begin{align*} 
(\Omg_{k,k+1}\otimes_{k+1}\Omg_{k+1,k},{d_{n}}) &\cong 
(\Omg_{k,k-1}\otimes_{k-1}\Omg_{k-1,k},{d_{n}})  
\oplus_{ [n-2k] } 
\Omega_{k}^{d_{n}},\mspace{20mu}\text{if }\ n-2k\geq 0,
\\ 
(\Omg_{k,k-1}\otimes_{k-1}\Omg_{k-1,k},{d_{n}}) &\cong 
(\Omg_{k,k+1}\otimes_{k+1}\Omg_{k+1,k},{d_{n}})  
\oplus_{ [2k-n] } 
\Omega_{k}^{d_{n}},\mspace{20mu} \text{if }\ n-2k \leq 0 .
\end{align*} 
\end{prop}

\subsection{A categorification of $V(n)$}

Let $\cV_k(n)$ (resp. $\cV_{k,k+1}(n)$)  be the derived category of 
bigraded, left, compact $(\Omega_k,d_n)$ modules (resp. bigraded, left, compact $(\Omega_{k,k+1},d_{n})$-modules), 
and define the functors 
\begin{align*}
\Ind_k^{k+1,k} &\colon \cV_{k}(n) \to \cV_{k+1,k}(n), & \Res_{k}^{k+1,k} &\colon \cV_{k+1,k}(n) \to \cV_{k}(n),
\\
\Ind_{k+1}^{k+1,k} &\colon \cV_{k+1}(n) \to \cV_{k+1,k}(n), & \Res_{k+1}^{k+1,k} &\colon \cV_{k+1,k}(n) \to \cV_{k+1}(n) .
\end{align*}
For each $k\geq 0$ define the functors 
\begin{align*}
\F_k(-) &= \Res_{k+1}^{k,k+1}\circ\bigl(\Omega_{k+1,k}^{d_n}\otimes_{k}^{\mathbf L}(-)\bigr)\brak{-k,0},
\intertext{and}
\E_k(-) &= \Res_k^{k,k+1}\circ\bigl(\Omega_{k,k+1}^{d_n}\otimes_{k+1}^{\mathbf L}(-)\bigr)\brak{k+1-n,0}, 
\end{align*}  
where 
$\Omega_{k+1,k}^{d_n}$ is seen as a DG $(\Omega_{k+1,k},\Omega_{k})_n$-bimodule 
and $\Omega_{k,k+1}^{d_n}$ as a DG $(\Omega_k,\Omega_{k,k+1})_n$-bimodule.

\smallskip

Proposition~\ref{prop:EFiso} and Theorem~\ref{thm:Lniso} 
below 
are a direct consequence of Propositions~\ref{prop:qihflag} and~\ref{prop:diffcomm}, 
together with well know results:  
see for example~\cite[\S6.2]{fks} or~\cite[\S3.4,\S5.3]{L3}
(see also~\cite[\S5.3]{CR} for the ungraded case). 
There is an equivalence of triangulated categories between $\cV_k(n)$ and the bounded derived category $\cD^b(H(G_{k;n})\amod)$.

\begin{prop}\label{prop:EFiso}
The functors $\F_k$ and $\E_k$ are biadjoint up to a shift. 
Moreover we have natural isomorphisms of functors  
\begin{align*}
\E_k\circ\F_k \cong \F_{k-1}\circ\E_{k-1}\, \oplus_{ [n-2k] } \id_{k},\mspace{40mu}\text{if }\ n-2k\geq 0,
\intertext{and}
\F_{k-1}\circ\E_{k-1} \cong \E_k\circ\F_k\, \oplus_{ [2k-n] } \id_{k},\mspace{40mu}\text{if }\ n-2k\leq 0.
\end{align*} 
\end{prop}

\begin{thm}\label{thm:Lniso}
Define the category 
$\cV(n) = \bigoplus\limits_{k\geq 0}\cV_k(n)$.  We have a $\bZ[q,q^{-1}]$-linear isomorphism of $U_q(\slt)$-modules, 
$K_0(\cV(n))\cong V(n)$, for all $n \ge 0$.
\end{thm}

All the above can be applied without difficulty to $\Omega_k^{m}$ and $\Omega_{k,k+1}^{m}$ in $\cM(\lambda q^m)$,  
with $m \ge 0$.  After passing to the derived category we get a category isomorphic to $\cV(n+m+1)$.
In particular if we take $m = N$ and $d_{0}$, this yields a differential on $\cM(N)$ with $q$-grading $2$.

\subsection{nilHecke action}\label{sec:nhactdg}

Recall the ring $\Omega_{k, k+1, \dotsc ,k+s}=\bQ[\und{x}_{k},\und{\xi}_{s}, \und{s}_{k+s}]$
from \S\ref{ssec:modbim}.
It has the structure of a DG-algebra with differential 
$d_{n}$ given by 
\begin{align*}
d_{n}(x_r) &=0,&
d_{n}(\xi_r) &= 0,&
d_{n}(s_r) &= Y_{n-r+1}.
\end{align*}
and $Y_{i}$ such that
\[
(1 +  x_{1}t + \dotsc + x_{k}t^k)(1 + \xi_1 t + \dotsc + \xi_{s}t^s)(1 + Y_{1}t + \dotsc + Y_{i} t^{i} + \dotsc) = 1.
\]
It is a DG $(k,k+s)_n$-bimodule quasi-isomorphic to 
\[  
(\Omega_{k,k+1}\otimes_{k+1}\Omega_{k+1,k+2}\otimes_{k+2} \dotsm \otimes_{k+n-1}\Omega_{k+n-1,k+n},d_{n}).  
\] 

\medskip

As explained in~\S\ref{ssec:nilHecke}, the nilHecke algebra $\nh_m$  acts on the ring 
$\Omega_{k,\dotsc ,k+m}$ as endomorphisms of $(\Omega_{k+m},\Omega_k)$-bimodules.   
We now show this action extends to the DG context.

\begin{prop}
The nilHecke algebra $\nh_m$ acts as endomorphisms of the DG $(k,k+m)_n$-bimodule $(\Omega_{k,\dotsc , k+m},d_n)$ and 
of the DG $(k+m,k)_n$-bimodule $(\Omega_{k+m,\dotsc , k},d_n)$.
\end{prop}
\begin{proof}  
It is sufficient to verify that the action of the $\partial_i$s from $\nh_m$ commute with the differential $d_n$ on $(\Omega_{k,\dotsc , k+m},d_n)$ since the action of $x_i$, which comes to multiplying by $\xi_{i}$, obviously commutes from the definition $d_n(\xi_i) = 0$. 
The commutation with the action of $\partial_i$ follows from the fact that $X^-$ is a bimodule morphism with parity $0$ and that the differential is a bimodule endomorphism. Indeed, suppose $f$ is a polynomial in $x_{i,k}$ and $\und\xi_{m}$, then $X^-(d_n(f)) = 0 = d_n(X^-(f))$. Suppose now recursively that $f \in \Omega_{k,k+m}$ respects such a relation. From the bimodule structure, we get
\begin{align*}
d_n(X^-(fs_{i,k+m})) &= d_n(X^-(f)s_{i,k+m} ) = d_n(X^-(f))s_{i,k+m} + (-1)^{p(f)}  X^-(f) d_n(s_{i,k+m}) \\
 &=  X^-(d_n(f))s_{i,k+m} + (-1)^{p(f)}X^-(f)Y_{n-i+1}   , \\
X^-(d_n(fs_{i,k+m})) &= X^-(d_n(f)s_{i,k+m}) + (-1)^{p(f)}  X^-(f d_n(s_{i,k+m}))  \\
&=   X^-(d_n(f))s_{i,k+m}  + (-1)^{p(f)} X^-(f)Y_{n-i+1},
\end{align*}
which is enough to conclude the proof since we can express every $s_{i,k+j}$ as a
combination of $s_{i,k+m}$ and $\xi_{k+i}$.
\end{proof}

\begin{cor}
The nilHecke algebra $\nh_s$ acts as endomorphisms of $\E^s$ and of $\F^s$.
\end{cor}

This action coincides with the one from Lauda~\cite{L1} and
Chuang-Rouquier~\cite{CR}.


\section{Verma categorification and a diagrammatic algebra}\label{sec:nhecke}

In \S\ref{ssec:nilHecke} we have constructed an extended version of the nilHecke algebra
which acts on $\Omega_{k,k+n}$. 
In this section we study this algebra more closely and give it a diagrammatic version in the same spirit as the KLR 
algebras~\cite{KL1,L1,R1}, as isotopy classes of diagrams modulo relations.

\subsection{The superalgebra $A_n$}
Consider the collection of braid-like diagrams on the plane connecting $n$ points on the horizontal 
axis $\bR \times \{0\}$ to $n$ points on the horizontal line $\bR \times \{1\}$ with no critical point when projected onto the $y$-axis, 
such that the strands can never turn around. 
We allow the strands to intersect each other without triple intersection points.
We decorate the strands with (black) dots and white dots, equipping the diagram with a height function such that we cannot have two white dots on the same height.  Moreover, the regions are labeled by integers such that each time we cross a strand from left to right the label increases by $1$ :
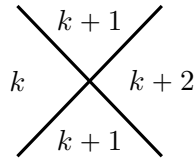
\[
\tikz[very thick,xscale=2.4,yscale=2,baseline={([yshift=-.5ex]current bounding box.center)}]{
          \draw (.1,-.5)-- (.9,.5);
          \draw (.9,-.5)-- (.1,.5);
    	  \node at (0.1,0){$k$};
    	  \node at (0.5,-.4){$k+1$};
    	  \node at (0.5,.4){$k+1$};
    	  \node at (.9,0){$k+2$};
}
\] 
From this rule, it is enough to write a label for the left most region of a diagram. Furthermore, these diagrams are taken up to isotopy which does not create any critical point and preserves the relative height of the white dots as well as the labeling of the regions. 
An example of such diagram is:
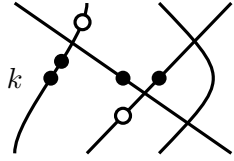
\[
\tikz[very thick,xscale=1.2,baseline={([yshift=-.5ex]current bounding box.center)}]{
\draw (.1, -1) .. controls (.1,-.5) and (.9, .5) ..   (.9, 1)  node [midway,fill=black,circle,inner sep=2pt]{}  node [pos=0.6,fill=black,circle,inner sep=2pt]{}  node [pos=0.85,fill=white, draw=black,circle,inner sep=2pt]{};
\draw (.9, -1)-- (2.5, 1)  node [near start,fill=white, draw=black,circle,inner sep=2pt]{}  node [midway,fill=black,circle,inner sep=2pt]{};
\draw (1.7, -1) .. controls (2.5,0) ..(1.7, 1);
\draw (2.5, -1)-- (.1, 1)  node [midway,fill=black,circle,inner sep=2pt]{};
 \node at (0.1,0){$k$};
}
\]

Fix a field $\Bbbk$ and denote by $A_n$ the $\Bbbk$-superalgebra obtained by the linear combinations of these diagrams together 
with multiplication given by gluing diagrams on top of each other whenever the labels of the regions 
agree and zero otherwise. 
In our conventions $ab$ means stacking the diagram for $a$ atop the one for $b$, whenever they are composable. 
Our diagrams are subjected to the following local relations: 
\begin{align}
	\tikz[very thick,xscale=1.2,baseline={([yshift=-.5ex]current bounding box.center)}]{
	          \draw (.1,-.5)-- (.1,.5) node [near start,fill=white, draw=black,circle,inner sep=2pt]{};
      	 	   \node at (.5,0){$\cdots$};
	          \draw (.9,-.5)-- (.9,.5) node [near end,fill=white, draw=black,circle,inner sep=2pt]{};
	    	  \node at (-0.25,0){$k$};
  	}
	&\quad=\quad -\quad
	\tikz[very thick,xscale=1.2,baseline={([yshift=-.5ex]current bounding box.center)}]{
	          \draw (.1,-.5)-- (.1,.5) node [near end,fill=white, draw=black,circle,inner sep=2pt]{};
      	 	   \node at (.5,0){$\cdots$};
	          \draw (.9,-.5)-- (.9,.5) node [near start,fill=white, draw=black,circle,inner sep=2pt]{};
	    	  \node at (-0.25,0){$k$};
	} \mspace{10mu},
 	&&&
	\tikz[very thick,xscale=1.2,baseline={([yshift=-.5ex]current bounding box.center)}]{
	          \draw (.1,-.5)-- (.1,.5) node [near start,fill=white, draw=black,circle,inner sep=2pt]{}
			node [near end,fill=white, draw=black,circle,inner sep=2pt]{};
	    	  \node at (-0.25,0){$k$};
	}
	&\quad=\quad 0  \mspace{10mu},
	\label{eq:bdots}
	\\[1ex]
	\tikz[very thick,xscale=1.2,baseline={([yshift=-.5ex]current bounding box.center)}]{
		\node at (-0,0){$k$};
	      \draw  +(.1,-.75) .. controls (1,0) ..  +(.1,.75);
	      \draw  +(.9,-.75) .. controls (0,0) ..  +(.9,.75);
	 }
	&\quad=\quad 0  \mspace{10mu}, 
	&&&
	\tikz[very thick,xscale=1.2,baseline={([yshift=-.5ex]current bounding box.center)}]{
		\node at (-0.1,0){$k$};
	     	 \draw  +(.75,-.75) .. controls (0,0) ..  +(.75,.75);
		 \draw (0,-.75)-- (1.5,.75);
	          \draw (0,.75)-- (1.5,-.75);
	 }
	&\quad=\quad 
	 \tikz[very thick,xscale=1.2,baseline={([yshift=-.5ex]current bounding box.center)}]{
		\node at (-0,0){$k$};
	     	 \draw  +(.75,-.75) .. controls (1.5,0) ..  +(.75,.75);
		 \draw (0,-.75)-- (1.5,.75);
	          \draw (0,.75)-- (1.5,-.75);
	} \mspace{10mu},
	\label{eq:crossings}
\end{align}
\begin{align}
	\tikz[very thick,xscale=1.2,baseline={([yshift=-.5ex]current bounding box.center)}]{
	          \draw (.1,-.5)-- (.9,.5);
	          \draw (.9,-.5)-- (.1,.5) node [near end,fill=black,circle,inner sep=2pt]{};
	    	  \node at (-0,0){$k$};
	}
	&\quad=\quad
	\tikz[very thick,xscale=1.2,baseline={([yshift=-.5ex]current bounding box.center)}]{
	          \draw (.1,-.5)-- (.9,.5);
	          \draw (.9,-.5)-- (.1,.5) node [near start,fill=black,circle,inner sep=2pt]{};
	    	  \node at (-0,0){$k$};
	}
	\quad + \quad
	\tikz[very thick,xscale=1.2,baseline={([yshift=-.5ex]current bounding box.center)}]{
	          \draw (.1,-.5)-- (.1,.5);
	          \draw (.9,-.5)-- (.9,.5);
	    	  \node at (-0.25,0){$k$};
	}  \mspace{10mu},\label{eq:relnh1}
	\\[1ex]
	\tikz[very thick,xscale=1.2,baseline={([yshift=-.5ex]current bounding box.center)}]{
	          \draw (.1,-.5)-- (.9,.5) node [near start,fill=black,circle,inner
	          sep=2pt]{};
	          \draw (.9,-.5)-- (.1,.5);
	    	  \node at (-0,0){$k$};
	} 
	&\quad=\quad
	\tikz[very thick,xscale=1.2,baseline={([yshift=-.5ex]current bounding box.center)}]{
	          \draw (.1,-.5)-- (.9,.5) node [near end,fill=black,circle,inner
	          sep=2pt]{};
	          \draw (.9,-.5)-- (.1,.5);
	    	  \node at (-0,0){$k$};
	}
	\quad + \quad
	\tikz[very thick,xscale=1.2,baseline={([yshift=-.5ex]current bounding box.center)}]{
	          \draw (.1,-.5)-- (.1,.5);
	          \draw (.9,-.5)-- (.9,.5);
	    	  \node at (-0.25,0){$k$};
	}  \mspace{10mu},\label{eq:relnh2}
\end{align}
\begin{align}
	\tikz[very thick,xscale=1.2,baseline={([yshift=-.5ex]current bounding box.center)}]{
	          \draw (.1,-.5)-- (.9,.5);
	          \draw (.9,-.5)-- (.1,.5) node [near start,fill=white, draw=black,circle,inner sep=2pt]{};
	    	  \node at (-0,0){$k$};
	} 
	&\quad = \quad
	\tikz[very thick,xscale=1.2,baseline={([yshift=-.5ex]current bounding box.center)}]{
	          \draw (.1,-.5)-- (.9,.5) node [near end,fill=white, draw=black,circle,inner sep=2pt]{};
	          \draw (.9,-.5)-- (.1,.5) ;
	    	  \node at (-0,0){$k$};
	}  \mspace{10mu}, \label{eq:relomega1} \\[1ex]
	\tikz[very thick,xscale=1.2,baseline={([yshift=-.5ex]current bounding box.center)}]{
	          \draw (.1,-.5)-- (.9,.5) node [near start,fill=white, draw=black,circle,inner sep=2pt]{};
	          \draw (.9,-.5)-- (.1,.5) ;
	    	  \node at (-0,0){$k$};
	}
	\quad + \quad
	\tikz[very thick,xscale=1.2,baseline={([yshift=-.5ex]current bounding box.center)}]{
	          \draw (.1,-.5)-- (.9,.5) node [pos=.85,fill=white, draw=black,circle,inner sep=2pt]{}
				 node [pos=.65,fill=black,circle,inner sep=2pt]{};
	          \draw (.9,-.5)-- (.1,.5) ;
	    	  \node at (-0,0){$k$};
	}
	&\quad = \quad
	\tikz[very thick,xscale=1.2,baseline={([yshift=-.5ex]current bounding box.center)}]{
	          \draw (.1,-.5)-- (.9,.5);
	          \draw (.9,-.5)-- (.1,.5) node [near end,fill=white, draw=black,circle,inner sep=2pt]{};
	    	  \node at (-0,0){$k$};
	}
	\quad + \quad
	\tikz[very thick,xscale=1.2,baseline={([yshift=-.5ex]current bounding box.center)}]{
	          \draw (.1,-.5)-- (.9,.5);
	          \draw (.9,-.5)-- (.1,.5) node [pos=.15,fill=white, draw=black,circle,inner sep=2pt]{}
				 node [pos=.35,fill=black,circle,inner sep=2pt]{};
	    	  \node at (-0,0){$k$};
	}  \mspace{10mu} . \label{eq:relomega2}
\end{align}
We turn $A_n$ into bigraded superalgebras by setting the \emph{parity} 
\[ 
p\left(
\tikz[very thick,xscale=1.2,baseline={([yshift=-.5ex]current bounding box.center)}]{
          \draw (0,-.5)-- (0,.5) node [midway,fill=black,circle,inner
          sep=2pt]{};
    	  \node at (-0.35,0){$k$};
        }\quad
\right) =  
p\left(
\tikz[very thick,xscale=1.2,baseline={([yshift=-.5ex]current bounding box.center)}]{
          \draw (.1,-.5)-- (.9,.5);
          \draw (.9,-.5)-- (.1,.5);
    	  \node at (-0,0){$k$};
        }\quad
\right) = 0,
\mspace{50mu}
p\left(
\tikz[very thick,xscale=1.2,baseline={([yshift=-.5ex]current bounding box.center)}]{
          \draw (0,-.5)-- (0,.5) node [midway,fill=white, draw=black,circle,inner
          sep=2pt]{};
    	  \node at (-0.35,0){$k$};
        }\quad
\right) = 1 , 
\] 
and the $\bZ\times\bZ$-degrees as 
\[
\deg\left(
\tikz[very thick,xscale=1.2,baseline={([yshift=-.5ex]current bounding box.center)}]{
          \draw (0,-.5)-- (0,.5) node [midway,fill=black,circle,inner
          sep=2pt]{};
    	  \node at (-0.35,0){$k$};
        }\quad
\right) = (2,0), 
\mspace{35mu}
\deg\left(
\tikz[very thick,xscale=1.2,baseline={([yshift=-.5ex]current bounding box.center)}]{
          \draw (.1,-.5)-- (.9,.5);
          \draw (.9,-.5)-- (.1,.5);
    	  \node at (-0,0){$k$};
        }\quad
\right) = (-2,0) ,
\mspace{35mu}
\deg\left(
\tikz[very thick,xscale=1.2,baseline={([yshift=-.5ex]current bounding box.center)}]{
          \draw (0,-.5)-- (0,.5) node [midway,fill=white, draw=black,circle,inner
          sep=2pt]{};
    	  \node at (-0.35,0){$k$};
        }\quad
\right) = (-2k-2,2) . 
\]
One can check easily that all relations preserve the bidegree and the parity. 

\medskip 

We write $A_{n}(m)$ for the subsuperalgebra consisting of diagrams with a 
label $m$ on the leftmost region. This superalgebra is generated by the diagrams 
 \begin{equation*}
    \tikz{
      \node[label=below:{$\quad 1$}] at (-6,0){ 
        \tikz[very thick,xscale=1.2]{
          \draw (-.5,-.5)-- (-.5,.5); \node at (-.5,.75){$1$};
          \draw (0,-.5)-- (0,.5); \node at (0,.75){$2$};
          \draw (1.5,-.5)-- (1.5,.5); \node at (1.5,.75){$n$};
          \node at (.75,0){$\cdots$};
    	  \node at (-0.75,0){$m$};
        }
      };
	\node at (-4,-0.2){,};
      \node[label=below:{$\quad x_i$}] at (-2,0){ 
        \tikz[very thick,xscale=1.2]{
          \draw (-.5,-.5)-- (-.5,.5); \node at (-.5,.75){$1$};
          \draw (.5,-.5)-- (.5,.5) node [midway,fill=black,circle,inner
          sep=2pt]{}; \node at (.5,.75){$i$};
          \draw (1.5,-.5)-- (1.5,.5); \node at (1.5,.75){$n$};
          \node at (1,0){$\cdots$};
          \node at (0,0){$\cdots$};
    	  \node at (-0.75,0){$m$};
        }
      };
	\node at (-0,-0.2){,};
      \node[label=below:{$ \quad \omega_i$}] at (2,0){ 
        \tikz[very thick,xscale=1.2]{
          \draw (-.5,-.5)-- (-.5,.5);\node at (-.5,.75){$1$};
          \draw (.5,-.5)-- (.5,.5) node [midway,fill=white, draw=black,circle,inner sep=2pt]{};\node at (.5,.75){$i$};
          \draw (1.5,-.5)-- (1.5,.5);\node at (1.5,.75){$n$};
          \node at (1,0){$\cdots$};
          \node at (0,0){$\cdots$};
    	  \node at (-0.75,0){$m$};
        }
      };
	\node at (4,-0.2){,};
      \node[label=below:{$\quad \partial_i$}] at (6,0){ 
        \tikz[very thick,xscale=1.2]{
          \draw (-.5,-.5)-- (-.5,.5);\node at (-.5,.75){$1$};
          \draw (.1,-.5)-- (.9,.5);\node at (.1,.75){$i$};
          \draw (.9,-.5)-- (.1,.5);\node at (.9,.75){$i+1$};
          \draw (1.5,-.5)-- (1.5,.5);\node at (1.5,.75){$n$};
          \node at (1,0){$\cdots$};
          \node at (0,0){$\cdots$};
    	  \node at (-0.75,0){$m$};	
          \node at (1.65,0){.};
        }
      };
    }
\end{equation*}

 There is an obvious canonical inclusion of the nilHecke algebra $\nh_n$ into $A_{n}(m)$, the
 former seen as a superalgebra concentrated in parity zero.  
Note that the superalgebra $A_{n}(m)$ is the superalgebra introduced in \S\ref{ssec:nilHecke}. 

\begin{prop}\label{prop:Anbasis}
The superalgebra $A_{n}(m)$ admits a basis given by the elements
\[
x_1^{k_1} \dotsc x_n^{k_n} \omega_1^{\delta_1} \dotsc \omega_n^{\delta_n} \partial_\vartheta  
 \]
 for all reduced word $\vartheta \in S_n$, $k_i \in \bN$ and $\delta_i \in \{0,1\}$, with 
\[
\partial_\vartheta = \partial_{i_1} \dotsc \partial_{i_r}
\]
for $\vartheta = \tau_{i_1}\dotsc \tau_{i_r}$, $\tau_i$ being the transposition exchanging $i$ with $i+1$.
\end{prop}

\begin{proof} 
Using the relations \eqref{eq:relnh1}-\eqref{eq:relomega2} we can push all dots and white dots to the top of the diagrams. 
By the relations \eqref{eq:bdots}, we can have a maximum of one white dot per strand on the top of each diagram  
and by relations \eqref{eq:crossings} we get the decomposition in $\partial_w$ as in the 
nilHecke algebra (see \cite[\S2.3]{KL1} for example). Thus the family above is generating. The action described in the section below shows easily they act as linearly independent operators, concluding the proof.
\end{proof}

\subsection{The action of $A_{n}(m)$ on polynomial rings} \label{sec:nhpol}
The superalgebra $A_{n}(m)$ acts on $\bQ[\und{x}_n] \otimes \bV^\bullet(\und{\omega}_{n})$ 
with $x_i$ and $\omega_i$ acting by left multiplication while the action of $\partial_i$ is defined by 
\begin{align*}
\partial_i(1) &= 0,
 & 
\partial_i(x_j) &= 
\begin{cases}
1 &\text{ if } j=i, \\
-1 &\text{ if } j=i+1,\\
0 &\text{otherwise,}
\end{cases}
 & \partial_i(\omega_j) &= 
\begin{cases}
-\omega_{j+1} &\text{ if } j=i, \\
0 &\text{otherwise,}
\end{cases}
\end{align*}
together with the rule
\begin{equation}\label{eq:extleibniz}
\partial_i(fg) = \partial_i(f)g +f\partial_i(g) - (x_i - x_{i+1})\partial_i(f)\partial_i(g) , 
\end{equation}
for all $f,g \in \bQ[\und{x}_n] \otimes \bV^\bullet(\und{\omega}_{n})$.

\begin{prop}
Formulas above define an action of $A_{n}(m)$ on $\bQ[\und{x}_n] \otimes \bV^\bullet(\und{\omega}_{n})$.   
\end{prop}

\begin{proof}
The commutation relations produced by isotopies are immediate from $\partial_i$ being zero 
on $x_j$ and $\omega_j$ for $j \notin \{i,i+1\}$ together with the rule (\ref{eq:extleibniz}). 
The relation $\partial_i^2 = 0$ can easily be proved by induction. It is straightforward on $1$ from the definition. 
Let $f,g \in \bQ[\und{x}_n] \otimes \bV^\bullet(\und{\omega}_{n})$ be such that $\partial_i^2(f) = \partial_i^2(g) = 0$ and compute 
\[
\partial_i^2(fg) \overset{(\ref{eq:extleibniz})}{=} \partial_i \left( \partial_i(f)g + f \partial_i(g) - (x_i - x_{i+1})\partial_i(f)\partial_i(g) \right)\overset{(\ref{eq:extleibniz})}{=} 0.
\]
The Reidemeister III relation~\eqref{eq:crossings} is proved in a same way, we leave the details for the reader.
The nilHecke relations (\ref{eq:relnh1}) and (\ref{eq:relnh2}) are easy computations. 
Indeed we have 
\begin{align*}
\partial_i(x_{i+1}f) + f &\overset{(\ref{eq:extleibniz})}{=} - f + x_{i+1}\partial_i(f) + (x_i - x_{i+1})\partial_i(f) \\
&= x_i\partial_i(f),\\
\partial_i(x_{i}f) & \overset{(\ref{eq:extleibniz})}{=} f + x_i\partial_i(f) - (x_i - x_{i+1})\partial_i(f) \\
&= x_{i+1}\partial_i(f) + f,
\end{align*}
for all  $f \in \bQ[\und{x}_n] \otimes \bV^\bullet(\und{\omega}_{n})$. 
The second to last relation is direct from $\partial_i$ acting as zero on $\omega_{i+1}$.
Finally, computing  
\begin{align*}
\partial_i(\omega_if) &\overset{(\ref{eq:extleibniz})}{=} -\omega_{i+1}f + \omega_{i}\partial_i(f) + (x_i - x_{i+1})\omega_{i+1}\partial_i(f) \\
&\overset{(\ref{eq:relomega1}),(\ref{eq:relnh1})}{=}  \omega_i \partial_i(f) + \partial_i(x_{i+1}\omega_{i+1}f) - \omega_{i+1}x_{i+1} \partial_i(f)
\end{align*} 
gives \eqref{eq:relomega2}. 
\end{proof}

\subsection{Symmetric group action}

The action of $\nh_n$ on $\bQ[\und{x}_n]$ with $x_i$ acting by multiplication and $\partial_i$ by divided difference operators
$$\partial_i(f) = \frac{f-s_i(f)}{x_i - x_{i+1}},$$
induces an action of $\nh_n$ on $\bQ[\und{x}_n] \otimes \bV^\bullet(\und{\omega}_{n})$. 
This action goes through an action of the symmetric group $S_n$ on  $\bQ[\und{x}_n]\otimes\bV^\bullet(\und{\omega}_{n})$, given by 
\begin{align}
s_i(x_j) = 
\begin{cases}
x_{i+1} &\text{ if } j=i, \\
x_{i} &\text{ if } j=i+1, \\
x_j &\text{otherwise,}
\end{cases}
\intertext{and}
s_i(\omega_j) = 
\begin{cases}
\omega_{i} + (x_i-x_{i+1})\omega_{i+1} &\text{ if } j=i, \\
\omega_j &\text{otherwise,}
\end{cases}
\end{align}
and $s_i(fg)=s_i(f)s_i(g)$.

\smallskip 

With this $S_n$-action, the action of $\partial_i$ satisfies the usual Leibnitz rule for the Demazure operator, 
$\partial_i(fg)=\partial_i(f)g+s_i(f)\partial_i(g)$. 
Moreover, this action coincides with the one defined before since 
\begin{align*}
\partial_i(f)g + f\partial_i(g) - (x_i - x_{i+1})\partial_i(f)\partial_i(g) &= \frac{(f-s_i(f))g + f(g - s_i(g)) - (f-s_i(f))(g-s_i(g))}{x_i - x_{i+1}} \\
&= \frac{fg-s_i(fg)}{x_i - x_{i+1}}.
\end{align*}

\subsection{The action of $A_{n}(k+m)$ on $\Omega_{k,k+n}^{m}$.}
We have already observed in \S\ref{ssec:vermacatshifted} that the ring 
$ \Omega_{m,m+n}$ can be written as
\[
\Omega_{m,m+n} \cong \Omega_{m,m+1} \otimes_{m+1} \dotsc \otimes_{m+n-1} \Omega_{m+n-1,m+n} . 
\]
We therefore get an action of $A_{n}(m)$ on $\Omega_{m,m+n}$ given by
\begin{align*}
\tikz[very thick,xscale=1.2,baseline={([yshift=-.5ex]current bounding box.center)}]{
          \draw (0,-.5)-- (0,.5) node [midway,fill=black,circle,inner
          sep=2pt]{};
    	  \node at (-.65,0){$m+i$};
        }\quad
 & \mapsto \xi_{i+1}, &
\tikz[very thick,xscale=1.2,baseline={([yshift=-.5ex]current bounding box.center)}]{
          \draw (0,-.5)-- (0,.5) node [midway,fill=white, draw=black,circle,inner
          sep=2pt]{};
    	  \node at (-.65,0){$m+i$};
        }\quad
 & \mapsto s_{i+1,i+1},
\end{align*}
\[
\tikz[very thick,xscale=1.2,baseline={([yshift=-.5ex]current bounding box.center)}]{
          \draw (.1,-.5)-- (.9,.5);
          \draw (.9,-.5)-- (.1,.5);
    	  \node at (-.25,0){$m+i$};
        }\quad
\mapsto 
\left( X^- : \Omega_{i,i+1} \otimes \Omega_{i+1,i+2} \rightarrow  \Omega_{i,i+1} \otimes \Omega_{i+1,i+2}\right).
\]
This action agrees with the one defined in \S\ref{sec:nhpol} and this can be generalized to get an action of $A_{n}(k+m)$ on $\Omega_{k,k+n}^{m}$.

\subsection{Categorification}

We define
\[
A(m) = \bigoplus_{n \ge 0}  A_{n}(m) .
\] 
The usual inclusion $A_{n}(m)\hookrightarrow A_{n+1}(m)$ that adds a strand at the right of a diagram from $A_{n}(m)$ 
gives rise to induction and restriction functors $\F$ and $\E$ on $A(m)\fgmods$ that satisfy the $\slt$-relations. 
Our results in Sections~\ref{sec:higher},~\ref{sec:vermacat} and~\ref{sec:intvermacat} 
imply that 
$A_n(m)\fgmods$ categorifies the $(c-1+m-2n)$th-weight space of $M(\lambda q^{m-1})$ 
and that 
$A(m)\fgmods$ categorifies the Verma module $\cM(\lambda q^{m-1})$. 
The categorification of the Verma modules with integral highest weight using specializations of the
superalgebras $A_{n}(m)$
follows as a consequence of our results in \S\ref{sec:intvermacat}.

\subsection{Cyclotomic quotients} 
We can turn $A_{n}(N+1)$ into a DG-algebra, equipping it with a differential of degree $(0,-2)$ defined by
\begin{align*}
d_N(x_i) &= 0,& d_N(\partial_i) &= 0,& d_N(\omega_i) &= (-1)^i h_{N-i+1}(\und x_i),
\end{align*}
together with the parity graded Leibniz rule
(as before, the parity is the cohomological degree of $d_n$).
We stress that we take the complete homogeneous symmetric polynomial on only the first $i$ variables, 
and therefore it commutes with $\partial_j$ for all $j \ne i$ and respects \eqref{eq:relomega2}
for $\partial_i$ (recall that $Y_{i,k} = (-1)^i h_i(\xi_k)$).

\begin{prop}
  The DG-algebra $(A_{n}(N+1), d_N)$ is quasi-isomorphic to the cyclotomic quotient of the 
  nilHecke algebra $\nh_n^N = \nh_n/(x_1^N)$,
\[
(A_{n}(N+1), d_N) \cong \nh_n^N.
\]
\end{prop}

\begin{proof}
  It suffices to see that $d_N(\omega_1) = h_{N}(x_1) = x_1^N$ since $d_N(\omega_i)$ lies in the ideal 
  generated by $x_1^N$ (see for example~\cite[Proposition~2.8]{LH}). 
\end{proof}



\vspace*{1cm}


\end{document}